\documentclass[a4paper, 11pt]{article}

\usepackage{mathtools, amssymb, amsthm} % AMS series
\usepackage[top=40truemm,bottom=30truemm,left=30truemm,right=30truemm]{geometry}
\usepackage{bm}
\usepackage{enumitem}
\usepackage{xcolor}

\usepackage[
	unicode=true,
	bookmarks=true,
	bookmarkstype=toc,
	bookmarksnumbered=true
	]{hyperref}

\usepackage{tikz}
\usetikzlibrary{intersections, calc, arrows.meta}
\usepackage{subcaption}

\DeclareMathOperator{\dom}{dom}
\DeclareMathOperator*{\esssup}{ess\,sup}

\DeclareMathOperator{\Var}{Var}

\newcommand{\dd}{\mathrm{d}}

\newcommand{\ep}{\varepsilon}

\newcommand{\amd}{\mspace{10mu} \text{and} \mspace{10mu}}
\newcommand{\argdot}{\mspace{1.5mu} \cdot \mspace{1.5mu}}

\newcommand{\abs}[1]{{\left\lvert #1 \right\rvert}}
\newcommand{\norm}[1]{{\left\lVert #1 \right\rVert}}

\newcommand{\dv}[1]{\frac{\dd}{\dd #1}}

\newcommand{\Set}[2]{{\left\{ \mspace{1mu} #1 : #2 \mspace{1mu} \right\}}}

\theoremstyle{plain}
\newtheorem{theorem}{Theorem}[section]
\newtheorem{lemma}[theorem]{Lemma}

\newtheorem{corollary}[theorem]{Corollary}

\theoremstyle{definition}
\newtheorem{definition}[theorem]{Definition}
\newtheorem{example}[theorem]{Example}
\newtheorem{notation}{Notation}

\theoremstyle{remark}
\newtheorem{remark}[theorem]{Remark}

\numberwithin{equation}{section}

\begin{document}

\title{Mild solutions, variation of constants formula, and linearized stability for delay differential equations}
\author{Junya Nishiguchi\thanks{Mathematical Science Group, Advanced Institute for Materials Research (AIMR), Tohoku University, Katahira 2-1-1, Aoba-ku, Sendai, 980-8577, Japan}
\footnote{E-mail: junya.nishiguchi.b1@tohoku.ac.jp}}
\date{}

\maketitle

\begin{abstract}
The method and the formula of variation of constants for ordinary differential equations (ODEs) is a fundamental tool to analyze the dynamics of an ODE near an equilibrium.
It is natural to expect that such a formula works for delay differential equations (DDEs), however, it is well-known that there is a conceptual difficulty in the formula for DDEs.
Here we discuss the variation of constants formula for DDEs by introducing the notion of a \textit{mild solution}, which is a solution under an initial condition having a discontinuous history function.
Then the \textit{principal fundamental matrix solution} is defined as a matrix-valued mild solution, and we obtain the variation of constants formula with this function.
This is also obtained in the framework of a Volterra convolution integral equation, but the treatment here gives an understanding in its own right.
We also apply the formula to show the principle of linearized stability and the Poincar\'{e}-Lyapunov theorem for DDEs, where we do not need to assume the uniqueness of a solution.

\begin{flushleft}

\end{flushleft}

\begin{flushleft}
\end{flushleft}

\end{abstract}

\tableofcontents

\section{Introduction}

Studies concerning with the variation of constants formula for delay differential equations (DDEs) have a long history of over fifty years.
Nevertheless, the reason why we try to discuss the variation of constants formula in this paper is that such a consideration gives rise to a conceptual difficulty that is peculiar to the theory of DDEs.
Specifically, it is usual to discuss DDEs within the scope of continuous history functions, but a class of discontinuous history functions emerges as initial conditions when we try to obtain the variation of constants formula.
In connection with this, a matrix-valued solution having a certain discontinuous matrix-valued function as the initial condition is called the \textit{fundamental matrix solution}.
However, it is quite difficult to understand why the solution is called the ``fundamental matrix solution'' when compared with the theory of ordinary differential equations (ODEs).

This conceptual difficulty has arisen in the theoretical development about the variation of constants formula in the texts~\cite{Hale_1971_Springer} and \cite{Hale_1977} by Jack Hale.
In the revised edition~\cite{Hale--VerduynLunel_1993}, the theoretical development is rewritten based on the consideration in \cite{VerduynLunel_1989}.
There also exist studies to understand the conceptual difficulty of the variation of constants formula for DDEs within the framework of Functional Analysis (e.g., see \cite{Clement--Diekmann--Gyllenberg--Heijmans--Thieme_1987_dual_semigroups_I}, \cite{Diekmann_1987}, and \cite{Diekmann--Gyllenberg_2012}).
In this framework, it is essential that the Banach space of continuous functions on closed and bounded interval endowed with the supremum norm is not reflexive, and the theory is constructed by using the so called ``sun-star calculus''.
See \cite{Diekmann--vanGils--Lunel--Walther_1995} for the details.
See also \cite{Walther_2014_topics} for a survey article.

The idea of discussing the variation of constants formula for DDEs in this paper is to define a solution under an initial condition having a discontinuous history function as a \textit{mild solution}.
This concept comes from the analogy of the notion of mild solutions of abstract linear evolution equations, and its terminology also originates from this.
It can be said that the notion of mild solutions is to elevate the technique to exchange the order of integration to a concept.

The dependence of the derivative $\Dot{x}(t)$ of an unknown function $x$ on the past value of $x$ is abstracted to the concept of \textit{retarded functional differential equations} (RFDEs).
In this paper, we consider an autonomous linear RFDE
\begin{equation}\label{eq: linear RFDE}
	\Dot{x}(t) = Lx_t
	\mspace{25mu}
	(t \ge 0)
\end{equation}
for a continuous linear map $L \colon C([-r, 0], \mathbb{K}^n) \to \mathbb{K}^n$.
Here $\mathbb{K} = \mathbb{R}$ or $\mathbb{C}$, $n \ge 1$ is an integer, and $r > 0$ is a constant, which are fixed throughout this paper.
The derivative of $x$ at $0$ is interpreted as the right-hand derivative.
We are using the following notations:
\begin{itemize}
\item $C([-r, 0], \mathbb{K}^n)$ denotes the Banach space of all continuous functions from $[-r, 0]$ to $\mathbb{K}^n$ endowed with the supremum norm $\norm{\argdot}$.
Here a norm $\abs{\argdot}$ on $\mathbb{K}^n$, which is not necessarily the Euclidean norm, is fixed throughout this paper.
\item For each $t \ge 0$, $x_t \colon [-r, 0] \to \mathbb{K}^n$ is a continuous function defined by
\begin{equation*}
	x_t(\theta) \coloneqq x(t + \theta)
	\mspace{25mu}
	(\theta \in [-r, 0])
\end{equation*}
when $x \colon [-r, \infty) \to \mathbb{K}^n$ is continuous.
See also Definition~\ref{dfn: history segment}.
\end{itemize}
In addition to the linear RFDE~\eqref{eq: linear RFDE}, we also consider a non-homogeneous linear RFDE
\begin{equation}\label{eq: non-homogeneous linear RFDE}
	\dot{x}(t) = Lx_t + g(t)
	\mspace{25mu}
	(\text{a.e.\ $t \ge 0$})
\end{equation}
for some $g \in \mathcal{L}^1_\mathrm{loc}([0, \infty), \mathbb{K}^n)$.
Here $\mathcal{L}^1_\mathrm{loc}([0, \infty), \mathbb{K}^n)$ denotes the linear space of all locally Lebesgue integrable functions from $[0, \infty)$ to $\mathbb{K}^n$ defined almost everywhere.
See also the notations given below.
We refer the reader to \cite{Stein--Shakarchi_2005} and \cite{Rudin_1987} as references of the theory of Lebesgue integration for scalar-valued functions.

To study these differential equations, the following expression of $L$ by a \textit{Riemann-Stieltjes integral}
\begin{equation}\label{eq: L as RS integral}
	L\psi
	= \int_{-r}^0 \mathrm{d}\eta(\theta) \mspace{2mu} \psi(\theta)
\end{equation}
for $\psi \in C([-r, 0], \mathbb{K}^n)$ is useful.
Here $\eta \colon [-r, 0] \to M_n(\mathbb{K})$ is an $n \times n$ matrix-valued function of bounded variation.
The above representability is ensured by a corollary of the Riesz representation theorem (see Corollary~\ref{cor: a cor of Riesz rep thm}).
It is a useful convention that the domain of definition of $\eta$ is extended to $(-\infty, 0]$ by letting
\begin{equation*}
	\eta(\theta) \coloneqq \eta(-r)
\end{equation*}
for $\theta \in (-\infty, -r]$.
See Appendix~\ref{sec: RS integrals wrt matrix-valued functions} for the Riemann-Stieltjes integrals with respect to matrix-valued functions.
For the use of Riemann-Stieltjes integrals in the context of RFDEs, see \cite[Chapters 6 and 7]{Hale_1977}, \cite[Chapter 2]{VerduynLunel_1989}, \cite[Chapter 4]{Hino--Murakami--Naito_1991}, \cite[Chapters 6 and 7]{Hale--VerduynLunel_1993}, and \cite[Chapter I]{Diekmann--vanGils--Lunel--Walther_1995}, for example.

\vspace{0.5\baselineskip}
This paper is organized as follows:

In Section~\ref{sec: mild sol and fundamental matrix sol}, we introduce the notion of a history segment $x_t$ for a discontinuous function $x \colon [-r, \infty) \supset \dom(x) \to \mathbb{K}^n$.
By using this, we also introduce the notion of a mild solution to the linear RFDE~\eqref{eq: linear RFDE} under an initial condition
\begin{equation}\label{eq: initial condition}
	x_0 = \phi \in \mathcal{M}^1([-r, 0], \mathbb{K}^n).
\end{equation}
Here $\mathcal{M}^1([-r, 0], \mathbb{K}^n)$ consists of elements of $\mathcal{L}^1([-r, 0], \mathbb{K}^n)$ that are defined at $0$.
Roughly speaking, a function $x \colon [-r, \infty) \supset \dom(x) \to \mathbb{K}^n$ is said to be a mild solution of \eqref{eq: linear RFDE} under the initial condition~\eqref{eq: initial condition} if it satisfies
\begin{equation*}
	x(t)
	= \phi(0) + L\int_0^t x_s \mspace{2mu} \mathrm{d}s
	\mspace{25mu}
	(t \ge 0).
\end{equation*}
Here $\int_0^t x_s \mspace{2mu} \mathrm{d}s \in C([-r, 0], \mathbb{K}^n)$ is defined by
\begin{equation*}
	\biggl( \int_0^t x_s \mspace{2mu} \mathrm{d}s \biggr)(\theta)
	\coloneqq \int_0^t x(s + \theta) \mspace{2mu} \mathrm{d}s
	\mspace{25mu}
	(\theta \in [-r, 0]).
\end{equation*}
See Definitions~\ref{dfn: mild sol} and \ref{dfn: int_0^t x_s ds} for the details.
After proving the existence and uniqueness of a mild solution of the linear RFDE~\eqref{eq: linear RFDE} under the initial condition~\eqref{eq: initial condition}, we define the \textit{principal fundamental matrix solution} of \eqref{eq: linear RFDE} as a matrix-valued mild solution $X^L \colon [-r, \infty) \to M_n(\mathbb{K})$ under the initial condition $X^L_0 = \Hat{I}$.
Here $\Hat{I} \colon [-r, 0] \to M_n(\mathbb{K})$ is a discontinuous function defined by
\begin{equation}\label{eq: Hat{I}}
	\Hat{I}(\theta) \coloneqq
	\begin{cases}
		O & (\theta \in [-r, 0)), \\
		I & (\theta = 0).
	\end{cases}
\end{equation}

In Section~\ref{sec: DE satisfied by X^L}, we derive a differential equation
\begin{equation*}
	\Dot{X}^L(t)
	= \int_{-t}^0 \mathrm{d}\eta(\theta) \mspace{2mu} X^L(t + \theta)
\end{equation*}
satisfied by the principal fundamental matrix solution $X^L$ of \eqref{eq: linear RFDE}.
In the derivation, it is useful to use the notions of \textit{Volterra operator} and \textit{Riemann-Stieltjes convolution}.
See Subsection~\ref{subsec: definitions of Volterra op and RS convolution} for the definitions and Subsection~\ref{subsec: properties of Volterra op and RS convolution} for the fundamental properties.
The above differential equation is the key to obtain a variation of constants formula.

In Section~\ref{sec: non-homogeneous liner RFDEs}, we consider the non-homogeneous linear RFDE~\eqref{eq: non-homogeneous linear RFDE}.
To study a mild solution of \eqref{eq: non-homogeneous linear RFDE} under the initial condition~\eqref{eq: initial condition}, we also consider an integral equation
\begin{equation}\label{eq: integral eq with G, t_0 = 0}
	x(t)
	= \phi(0) + L\int_0^t x_s \mspace{2mu} \mathrm{d}s + G(t)
	\mspace{25mu}
	(t \ge 0)
\end{equation}
for a continuous function $G \colon [0, \infty) \to \mathbb{K}^n$ with $G(0) = 0$.
We show that the above integral equation has a unique solution $x^L(\argdot; \phi, G)$ under the initial condition~\eqref{eq: initial condition}.

In Section~\ref{sec: convolution and Volterra op}, we consider a non-homogeneous linear RFDE
\begin{equation}\label{eq: non-homogeneous linear RFDE, continuous forcing term}
	\Dot{x}(t) = Lx_t + f(t)
	\mspace{25mu}
	(t \ge 0)
\end{equation}
for a continuous function $f \colon [0, \infty) \to \mathbb{K}^n$ to motivate the use of the convolution for locally Riemann integrable functions.
We show that the function $x(\argdot; f) \colon [-r, \infty) \to \mathbb{K}^n$ defined by $x(\argdot; f)_0 = 0$ and
\begin{equation}\label{eq: X^L * f}
	x(t; f) \coloneqq \int_0^t X^L(t - u)f(u) \mspace{2mu} \mathrm{d}u
\end{equation}
for $t \ge 0$ is a solution to Eq.~\eqref{eq: non-homogeneous linear RFDE, continuous forcing term} after developing the results of convolution for locally Riemann integrable functions.
See Subsection~\ref{subsec: convolution and RS convolution} for the developments.

In Section~\ref{sec: VOC formula}, we study the non-homogeneous linear RFDE~\eqref{eq: non-homogeneous linear RFDE} under the initial condition~\eqref{eq: initial condition} and find a variation of constants formula expressed by the principal fundamental matrix solution $X^L$.
For this purpose, we indeed consider the integral equation~\eqref{eq: integral eq with G, t_0 = 0} for some continuous function $G \colon [0, \infty) \to \mathbb{K}^n$ with $G(0) = 0$.
One of the main results of this paper is that the solution $x^L(\argdot; \phi, G)$ of \eqref{eq: integral eq with G, t_0 = 0} under the initial condition~\eqref{eq: initial condition} satisfies
\begin{align}\label{eq: VOC formula, integral eq}
	&x^L(t; \phi, G) \notag \\
	&= X^L(t)\phi(0) + \bigl[ G^L(t; \phi) + G(t) \bigr] + \int_0^t \Dot{X}^L(t - u) \bigl[ G^L(u; \phi) + G(u) \bigr] \mspace{2mu} \mathrm{d}u
\end{align}
for all $t \ge 0$.
Here $\Dot{X}^L(t)$ denotes the derivative of the locally absolutely continuous function $X^L|_{[0, \infty)}$ at $t \ge 0$ (when it exists), and $G^L(\argdot; \phi) \colon [0, \infty) \to \mathbb{K}^n$ is a function determined by the initial history function $\phi$.
See Subsection~\ref{subsec: derivation of a general forcing term} for the detail of the derivation of the function $G^L(\argdot; \phi)$.
We note that before we obtain the variation of constants formula~\eqref{eq: VOC formula, integral eq}, we show that
\begin{equation}\label{eq: VOC formula, phi = 0}
	x^L(t; 0, G) = G(t) + \int_0^t \dot{X}^L(t - u)G(u) \mspace{2mu} \mathrm{d}u
\end{equation}
holds for all $t \ge 0$.
Then the derivation of \eqref{eq: VOC formula, integral eq} is performed by defining a function $z^L(\argdot; \phi) \colon [-r, \infty) \to \mathbb{K}^n$ by $z^L(\argdot; \phi)_0 = 0$ and
\begin{equation}\label{eq: z^L(argdot; phi)}
	z^L(t; \phi)
	\coloneqq x^L(t; \phi, 0) - X^L(t)\phi(0)
\end{equation}
for $t \ge 0$ and showing that $z \coloneqq z^L(\argdot; \phi)$ satisfies an integral equation
\begin{equation}\label{eq: integral eq with G^L(argdot; phi)}
	z(t)
	= L\int_0^t z_s \mspace{2mu} \mathrm{d}s + G^L(t; \phi)
	\mspace{25mu}
	(t \ge 0),
\end{equation}
because \eqref{eq: integral eq with G^L(argdot; phi)} shows that
\begin{equation*}
	z^L(\argdot; \phi) = x^L(\argdot; 0, G^L(\argdot; \phi))
\end{equation*}
holds.
Here we need to know the regularity of the function $G^L(\argdot; \phi)$, which is discussed in Subsection~\ref{subsec: regularity of general forcing term}.

In Section~\ref{sec: exponential stability}, we discuss the exponential stability of the principal fundamental matrix solution $X^L$ of the linear RFDE~\eqref{eq: linear RFDE} and the uniform exponential stability of the $C_0$-semigroup $\left( T^L(t) \right)_{t \ge 0}$ on the Banach space $C([-r, 0], \mathbb{K}^n)$ defined by
\begin{equation}\label{eq: C_0-semigroup}
	T^L(t)\phi \coloneqq x^L(\argdot; \phi, 0)_t
\end{equation}
for $(t, \phi) \in [0, \infty) \times C([-r, 0], \mathbb{K}^n)$.
We show that $X^L$ is $\alpha$-exponentially stable if and only if $\left( T^L(t) \right)_{t \ge 0}$ is uniformly $\alpha$-exponentially stable.
See Theorems~\ref{thm: exponential stability of T^L} and \ref{thm: exponential stability of X^L} for the details.

In Section~\ref{sec: linearized stability}, we apply the obtained variation of constants formulas to a proof of the stability part of the principle of linearized stability and Poincar\'{e}-Lyapunov theorem for RFDEs.
This is indeed an appropriate modification of the proof for ODEs.
However, the given proof makes clear the importance of the principal fundamental matrix solution.
In the statement, we do not need to assume the uniqueness of a solution.
Therefore, this should be compared with the proof relying on the nonlinear semigroup theory.

We have five appendices.
In Appendix~\ref{sec: RS integrals wrt matrix-valued functions}, we collect results on Riemann-Stieltjes integrals for matrix-valued functions that are needed for this paper.
In Appendix~\ref{sec: Riesz rep thm}, we give a proof of the representability of $L$ by a Riemann-Stieltjes integral~\eqref{eq: L as RS integral} because there does not seem to be any proof of the representability in the literature.
In Appendix~\ref{sec: variants of Gronwall's ineq}, we discuss Gronwall's inequality and its variants used in the context of RFDEs.
In Appendix~\ref{sec: lemmas on fixed pt argument}, we give lemmas that are used in the fixed point argument in this paper.
In Appendix~\ref{sec: convolution continued}, we continue to discuss the convolution.
The contents of this appendix will not be used in this paper, but it will be useful to share the proofs of results on the convolution for matrix-valued locally Lebesgue integrable functions in the literature of RFDEs.

\subsection*{Notations}

Throughout this paper, the following notations will be used.
\begin{itemize}
\item Let $E = (E, \norm{\argdot})$ be a Banach space.
For each subset $I \subset \mathbb{R}$, let $C(I, E)$ denote the linear space of all continuous functions from $I$ to $E$.
When the subset $I$ is a closed and bounded interval, the linear space $C(I, E)$ is considered as the Banach space of continuous functions endowed with the supremum norm $\norm{\argdot}$ given by
\begin{equation*}
	\norm{f} \coloneqq \sup_{x \in I} \norm{f(x)}
\end{equation*}
for $f \in C(I, E)$.
\item For each pair of Banach spaces $E = (E, \norm{\argdot})$ and $F = (F, \norm{\argdot})$, let $\mathcal{B}(E, F)$ denote the linear space of all continuous linear maps (i.e., all bounded linear operators) from $E$ to $F$.
For each $T \in \mathcal{B}(E, F)$, its operator norm is denoted by $\norm{T}$.
Then $\mathcal{B}(E, F)$ is considered as the Banach space of continuous linear maps endowed with the operator norm.
When $F = E$, $\mathcal{B}(E, F)$ is also denoted by $\mathcal{B}(E)$.
\item An $n \times n$ matrix $A \in M_n(\mathbb{K})$ is considered as a continuous linear map on the Banach space $\mathbb{K}^n$ endowed with the given norm $\abs{\argdot}$.
The operator norm of $A$ is denoted by $\abs{A}$.
The linear space $M_n(\mathbb{K})$ of all $n \times n$ matrices is considered as the Banach space of matrices endowed with the operator norm.
\item Let $d \ge 1$ be an integer, $X$ be a measurable set of $\mathbb{R}^d$, and $Y = \mathbb{K}^n$ or $M_n(\mathbb{K})$.
	\begin{itemize}[leftmargin=*]
	\item We say that a function $f \colon X \supset \dom(f) \to Y$ is a \textit{Lebesgue integrable function defined almost everywhere} if (i) $\dom(f)$ is measurable, (ii) $X \setminus \dom(f)$ has measure $0$, and (iii) $f|_{\dom(f)} \colon \dom(f) \to Y$ is Lebesgue integrable, i.e., it is measurable and
	\begin{equation*}
		\norm{f}_1
		\coloneqq \int_X \abs{f(x)} \mspace{2mu} \mathrm{d}x
		\coloneqq \int_{\dom(f)} \abs{f(x)} \mspace{2mu} \mathrm{d}x
	\end{equation*}
	is finite.
	We note that the function $\dom(f) \ni x \mapsto \abs{f(x)} \in [0, \infty)$ is also measurable by the continuity of the norm $\abs{\argdot}$, and the above integral is the unsigned Lebesgue integral.
	\item Let $\mathcal{L}^1(X, Y)$ be the set of all Lebesgue integrable functions from $X$ to $Y$ defined almost everywhere.
	For $f \in \mathcal{L}^1(X, Y)$, let
	\begin{equation*}
		\int_X f(x) \mspace{2mu} \mathrm{d}x
		\coloneqq \int_{\dom(f)} f(x) \mspace{2mu} \mathrm{d}x.
	\end{equation*}
	Then one can prove that
	\begin{equation*}
		\abs{\int_X f(x) \mspace{2mu} \mathrm{d}x}
		\le \int_{\dom(f)} \abs{f(x)} \mspace{2mu} \mathrm{d}x
		= \norm{f}_1
	\end{equation*}
	holds.
	\item For $f, g \in \mathcal{L}^1(X, Y)$, the addition $f + g \colon X \supset \dom(f) \cap \dom(g) \to Y$ is defined by
	\begin{equation*}
		(f + g)(x) \coloneqq f(x) + g(x)
	\end{equation*}
	for $x \in \dom(f) \cap \dom(g)$.
	Then $f + g \in \mathcal{L}^1(X, Y)$.
	The scalar multiplication $\alpha f$ for $\alpha \in \mathbb{K}$ is also defined, and it holds that $\alpha f \in \mathcal{L}^1(X, Y)$.
	\end{itemize}
\item Let $X$ be an interval of $\mathbb{R}$ and $Y = \mathbb{K}^n$ or $M_n(\mathbb{K})$.
Let $\mathcal{L}^1_\mathrm{loc}(X, Y)$ be the set of all functions $f \colon X \supset \dom(f) \to Y$ satisfying (i) $\dom(f)$ is measurable, (ii) $X \setminus \dom(f)$ has measure $0$, and (iii) for each closed and bounded interval $I$ contained in $X$, the restriction $f|_I \colon I \supset \dom(f) \cap I \to Y$ belongs to $\mathcal{L}^1(I, Y)$.
\end{itemize}

\section{Mild solutions and fundamental matrix solutions}\label{sec: mild sol and fundamental matrix sol}

\subsection{Definitions}

\subsubsection{History segments and memory space}

We first make clear the notion of history segments in our setting.

\begin{definition}\label{dfn: history segment}
Let $x \colon [-r, \infty) \supset \dom(x) \to \mathbb{K}^n$ be a function.
For each $t \ge 0$, we define a function $x_t \colon [-r, 0] \supset \dom(x_t) \to \mathbb{K}^n$ by
\begin{align*}
	\dom(x_t)
	&\coloneqq \Set{\theta \in [-r, 0]}{t + \theta \in \dom(x)}, \\
	x_t(\theta)
	&\coloneqq x(t + \theta)
	\mspace{25mu}
	(\theta \in \dom(x_t)).
\end{align*}
We call $x_t$ the \textit{history segment} of $x$ at $t$.
\end{definition}

We note that $\dom(x_t)$ is expressed by
\begin{equation*}
	\dom(x_t) = (\dom(x) - t) \cap [-r, 0],
\end{equation*}
where $\dom(x)$ is not necessarily equal to $[-r, \infty)$.

In this paper, we need discontinuous initial history functions.
For this purpose, we adopt the following space of history functions.

\begin{definition}[cf.\ \cite{Delfour--Mitter_1972_hereditary}]\label{dfn: memory space}
We define a linear subspace $\mathcal{M}^1([-r, 0], \mathbb{K}^n)$ of $\mathcal{L}^1([-r, 0], \mathbb{K}^n)$ by
\begin{equation*}
	\mathcal{M}^1([-r, 0], \mathbb{K}^n)
	\coloneqq \Set{\phi \in \mathcal{L}^1([-r, 0], \mathbb{K}^n)}{0 \in \dom(\phi)}
\end{equation*}
and call it the \textit{memory space} of $\mathcal{L}^1$-type.
We consider $\mathcal{M}^1([-r, 0], \mathbb{K}^n)$ as a seminormed space endowed with the seminorm $\norm{\argdot}_{\mathcal{M}^1} \colon \mathcal{M}^1([-r, 0], \mathbb{K}^n) \to [0, \infty)$ defined by
\begin{equation*}
	\norm{\phi}_{\mathcal{M}^1} \coloneqq \norm{\phi}_1 + \abs{\phi(0)}.
\end{equation*}
\end{definition}

\begin{remark}
Let $1 \le p < \infty$ and $E$ be a Banach space.
The memory space of $\mathcal{L}^1$-type should be compared with a Banach space $M^p([-r, 0], E)$ introduced by Delfour and Mitter~\cite{Delfour--Mitter_1972_hereditary}.
It is isomorphic to the product Banach space
\begin{equation*}
	L^p([-r, 0], E) \oplus E.
\end{equation*}
See also \cite{Bernier--Manitius_1978}, \cite{Delfour_1980}, and references therein for the use of the product space.
\end{remark}

\begin{definition}\label{dfn: continuous prolongation}
For each $\phi \in \mathcal{M}^1([-r, 0], \mathbb{K}^n)$, we will call a function $x \colon [-r, \infty) \supset \dom(x) \to \mathbb{K}^n$ a \textit{continuous prolongation} of $\phi$ if it satisfies the following properties: (i) $x_0 = \phi$, (ii) $[0, \infty) \subset \dom(x)$, and (iii) $x|_{[0, \infty)}$ is continuous.
\end{definition}

For a continuous prolongation $x$ of $\phi$,
\begin{equation*}
	\dom(x) = \dom(\phi) \cup [0, \infty)
\end{equation*}
holds.

\subsubsection{Mild solutions}

The following is the notion of a mild solution, whose introduction is a one of the contribution of this paper.
We use the expression of $L$ by the Riemann-Stieltjes integral~\eqref{eq: L as RS integral}
\begin{equation*}
	L\psi = \int_{-r}^0 \mathrm{d}\eta(\theta) \mspace{2mu} \psi(\theta)
\end{equation*}
for $\psi \in C([-r, 0], \mathbb{K}^n)$.

\begin{definition}[cf.\ \cite{Webb_1976}]\label{dfn: mild sol}
Let $\phi \in \mathcal{M}^1([-r, 0], \mathbb{K}^n)$ be given.
We say that a function $x \colon [-r, \infty) \supset \dom(x) \to \mathbb{K}^n$ is a \textit{mild solution} of the linear RFDE~\eqref{eq: linear RFDE} under the initial condition $x_0 = \phi$ if the following conditions are satisfied:
(i) $x$ is a continuous prolongation of $\phi$ and (ii) for all $t \ge 0$,
\begin{equation}\label{eq: mild sol}
	x(t)
	= \phi(0)
	+ \int_{-r}^0 \mathrm{d}\eta(\theta) \mspace{2mu}
		\biggl( \int_0^t x(s + \theta) \mspace{2mu} \mathrm{d}s \biggr)
\end{equation}
holds.
Here $\int_0^t x(s + \theta) \mspace{2mu} \mathrm{d}s$ is a Lebesgue integral.
\end{definition}

Since
\begin{equation*}
	\int_0^t x(s + \theta) \mspace{2mu} \mathrm{d}s
	= \int_{\theta}^{t + \theta} x(s) \mspace{2mu} \mathrm{d}s,
\end{equation*}
the integrand in Eq.~\eqref{eq: mild sol} is continuous with respect to $\theta \in [-r, 0]$.
Therefore, the integral in \eqref{eq: mild sol} is meaningful as a Riemann-Stieltjes integral.
Eq.~\eqref{eq: mild sol} is also expressed by
\begin{equation}\label{eq: mild sol'}
	x(t)
	= \phi(0) + \int_{-r}^0 \mathrm{d}\eta(\theta) \mspace{2mu}
		\biggl( \int_\theta^0 \phi(s) \mspace{2mu} \mathrm{d}s \biggr)
	+ \int_{-r}^0 \mathrm{d}\eta(\theta) \mspace{2mu}
		\biggl( \int_0^{t + \theta} x(s) \mspace{2mu} \mathrm{d}s \biggr),
\end{equation}
where the third term of the right-hand side may depend on $\phi$.

\begin{remark}
Eq.~\eqref{eq: mild sol} appeared at \cite[(5.19) in Corollary~5.13]{Webb_1976} after developing a nonlinear semigroup theory for some class of RFDEs.
Compared with this approach, the method of this paper is considered to be taking the notion of mild solutions as a starting point.
\end{remark}

\subsubsection{Notation \texorpdfstring{$\int_0^t x_s \mspace{2mu} \mathrm{d}s$}{int0t xs ds}}

For ease of notation, we introduce the following.

\begin{definition}\label{dfn: int_0^t x_s ds}
Let $x \in \mathcal{L}^1_\mathrm{loc}([-r, \infty), \mathbb{K}^n)$ be given.
For each $t \ge 0$, we define $\int_0^t x_s \mspace{2mu} \mathrm{d}s \in C([-r, 0], \mathbb{K}^n)$ by
\begin{equation*}
	{\biggl( \int_0^t x_s \mspace{2mu} \mathrm{d}s \biggr)}(\theta)
	\coloneqq \int_0^t x_s(\theta) \mspace{2mu} \mathrm{d}s
	= \int_\theta^{t + \theta} x(s) \mspace{2mu} \mathrm{d}s
\end{equation*}
for $\theta \in [-r, 0]$.
\end{definition}

We note that $\int_0^t x_s \mspace{2mu} \mathrm{d}s \in C([-r, 0], \mathbb{K}^n)$ introduced above is not an integral of a vector-valued function
\begin{equation*}
	[0, t] \ni s \mapsto x_s \in \mathcal{X}
\end{equation*}
for some function space $\mathcal{X}$.

\subsection{\texorpdfstring{$\int_0^t x_s \mspace{2mu} \mathrm{d}s$}{int0t xs ds} and its properties}

We have the following lemma.

\begin{lemma}\label{lem: continuity of int_0^t x_s ds}
If $x \in \mathcal{L}^1_\mathrm{loc}([-r, \infty), \mathbb{K}^n)$, then
\begin{equation}\label{eq: integrated history curve}
	[0, \infty) \ni t \mapsto \int_0^t x_s \mspace{2mu} \mathrm{d}s \in C([-r, 0], \mathbb{K}^n)
\end{equation}
is continuous.
\end{lemma}

\begin{proof}
We define a function $y \colon [-r, \infty) \to \mathbb{K}^n$ by
\begin{equation*}
	y(t) = \int_{-r}^t x(s) \mspace{2mu} \mathrm{d}s
\end{equation*}
for $t \ge -r$.
Then $y$ is continuous, and
\begin{equation*}
	y(t + \theta)
	= \int_\theta^{t + \theta} x(s) \mspace{2mu} \mathrm{d}s + \int_{-r}^\theta x(s) \mspace{2mu} \mathrm{d}s
\end{equation*}
holds for all $t \ge 0$ and all $\theta \in [-r, 0]$.
This shows that the function~\eqref{eq: integrated history curve} is continuous if and only if
\begin{equation*}
	[0, \infty) \ni t \mapsto y_t \in C([-r, 0], \mathbb{K}^n)
\end{equation*}
is continuous.
Since the continuity of this function is ensured by the uniform continuity of $y$ on any closed and bounded interval, the conclusion is obtained.
\end{proof}

When $x \in C([-r, \infty), \mathbb{K}^n)$, the Riemann integral
\begin{equation*}
	\mathrm{(R)}\int_0^t x_s \mspace{2mu} \mathrm{d}s \in C([-r, 0], \mathbb{K}^n)
\end{equation*}
of the continuous function
\begin{equation*}
	[0, t] \ni s \mapsto x_s \in C([-r, 0], \mathbb{K}^n)
\end{equation*}
exists.
See Graves~\cite[Section~2]{Graves_1927} for the definition of the Riemann integrability of functions on closed and bounded intervals taking values in normed spaces.
We now show that when $x \in C([-r, \infty), \mathbb{K}^n)$, the Riemann integral $\mathrm{(R)}\int_0^t x_s \mspace{2mu} \mathrm{d}s$ coincides with $\int_0^t x_s \mspace{2mu} \mathrm{d}s$ introduced in Definition~\ref{dfn: int_0^t x_s ds}.
More generally, one can prove the following result.

\begin{lemma}\label{lem: Riemann integral of two variables function}
Let $E$ be a Banach space, $[a, b]$ and $[c, d]$ be closed and bounded intervals of $\mathbb{R}$, and $f \colon [a, b] \times [c, d] \to E$ be a continuous function.
For each $y \in [c, d]$, let $f(\argdot, y) \in C([a, b], E)$ be defined by
\begin{equation*}
	f(\argdot, y)(x) \coloneqq f(x, y)
\end{equation*}
for $x \in [a, b]$.
Then
\begin{equation*}
	{\biggl( \int_c^d f(\argdot, y) \mspace{2mu} \mathrm{d}y \biggr)}(x)
	= \int_c^d f(x, y) \mspace{2mu} \mathrm{d}y
\end{equation*}
holds for all $x \in [a, b]$.
Here $\int_c^d f(\argdot, y) \mspace{2mu} \mathrm{d}y$ is the Riemann integral of the continuous function $[c, d] \ni y \mapsto f(\argdot, y) \in C([a, b], E)$.
\end{lemma}

We note that the continuity of $[c, d] \ni y \mapsto f(\argdot, y) \in C([a, b], E)$ is a consequence of the uniform continuity of $f$.

\begin{proof}[Proof of Lemma~\ref{lem: Riemann integral of two variables function}]
We fix $x \in [a, b]$.
Let $T \colon C([a, b], E) \to E$ be the evaluation map defined by
\begin{equation*}
	Tg \coloneqq g(x)
\end{equation*}
for $g \in C([a, b], E)$.
Since $T$ is a bounded linear operator, we have
\begin{equation*}
	{\biggl( \int_c^d f(\argdot, y) \mspace{2mu} \mathrm{d}y \biggr)}(x)
	= T\int_c^d f(\argdot, y) \mspace{2mu} \mathrm{d}y
	= \int_c^d Tf(\argdot, y) \mspace{2mu} \mathrm{d}y,
\end{equation*}
where the last term is equal to $\int_c^d f(x, y) \mspace{2mu} \mathrm{d}y$.
This completes the proof.
\end{proof}

As an application of Lemma~\ref{lem: Riemann integral of two variables function}, the following result can be obtained.

\begin{theorem}\label{thm: int_0^t x_s ds for continuous x}
If $x \in C([-r, \infty), \mathbb{K}^n)$, then
\begin{equation*}
	\mathrm{(R)}\int_0^t x_s \mspace{2mu} \mathrm{d}s
	= \int_0^t x_s \mspace{2mu} \mathrm{d}s
\end{equation*}
holds for all $t \ge 0$.
\end{theorem}

\begin{proof}
Let $t > 0$ be given.
We consider a function $f \colon [-r, 0] \times [0, t] \to \mathbb{K}^n$ defined by
\begin{equation*}
	f(\theta, s) \coloneqq x(s + \theta).
\end{equation*}
Then the function $f(\argdot, s)$ is equal to $x_s$.
By applying Lemma~\ref{lem: Riemann integral of two variables function} with this $f$,
\begin{equation*}
	{\biggl[ \mathrm{(R)}\int_0^t x_s \mspace{2mu} \mathrm{d}s \biggr]}(\theta)
	= \int_0^t x(s + \theta) \mspace{2mu} \mathrm{d}s
\end{equation*}
holds for all $\theta \in [-r, 0]$.
Since the right-hand side is equal to $\bigl( \int_0^t x_s \mspace{2mu} \mathrm{d}s \bigr)(\theta)$, this shows the conclusion.
\end{proof}

\begin{remark}
When $x \in C([-r, \infty), \mathbb{K}^n)$, Theorem~\ref{thm: int_0^t x_s ds for continuous x} yields that
\begin{equation*}
	\dv{t} \int_0^t x_s \mspace{2mu} \mathrm{d}s
	= x_t
	\in C([-r, 0], \mathbb{K}^n)
\end{equation*}
holds by the fundamental theorem of calculus for vector-valued functions.
\end{remark}

We have the following corollary.

\begin{corollary}\label{cor: T*int_0^t x_s ds for continuous x}
Let $F$ be a Banach space over $\mathbb{K}$ and $T \colon C([-r, 0], \mathbb{K}^n) \to F$ be a bounded linear operator.
If $x \in C([-r, \infty), \mathbb{K}^n)$, then
\begin{equation}\label{eq: T*int_0^t u_s ds}
	T\int_0^t x_s \mspace{2mu} \mathrm{d}s = \int_0^t Tx_s \mspace{2mu} \mathrm{d}s
\end{equation}
holds for all $t \ge 0$.
Here the right-hand side is the Riemann integral of the continuous function $[0, t] \ni s \mapsto Tx_s \in F$.
\end{corollary}

\begin{proof}
From Theorem~\ref{thm: int_0^t x_s ds for continuous x},
\begin{equation*}
	T\int_0^t x_s \mspace{2mu} \mathrm{d}s
	= T {\biggl[ \mathrm{(R)}\int_0^t x_s \mspace{2mu} \mathrm{d}s \biggr]}
	= \int_0^t Tx_s \mspace{2mu} \mathrm{d}s
\end{equation*}
holds since $T$ is a bounded linear operator.
\end{proof}

\begin{remark}
Corollary~\ref{cor: T*int_0^t x_s ds for continuous x} yields the following:
Let $x \colon [-r, \infty) \to \mathbb{K}^n$ be a continuous function satisfying $x_0 = \phi \in C([-r, 0], \mathbb{K}^n)$.
Since $L \colon C([-r, 0], \mathbb{K}^n) \to \mathbb{K}^n$ is a bounded linear operator, $x$ is a mild solution of the linear RFDE~\eqref{eq: linear RFDE} with the initial history function $\phi$ if and only if it satisfies
\begin{equation*}
	x(t) = \phi(0) + \int_0^t Lx_s \mspace{2mu} \mathrm{d}s
\end{equation*}
for all $t \ge 0$.
This shows that a mild solution coincides with a solution in the usual sense when the initial history function $\phi$ is continuous.
\end{remark}

\subsection{Existence and uniqueness of a mild solution}\label{subsec: existence and uniqueness of mild sol}

By using the contraction mapping principle with an \textit{a priori} estimate, we will prove the unique existence of a mild solution of the linear RFDE~\eqref{eq: linear RFDE} under an initial condition~\eqref{eq: initial condition}
\begin{equation*}
	x_0 = \phi \in \mathcal{M}^1([-r, 0], \mathbb{K}^n).
\end{equation*}
We will use the following notation.

\begin{notation}\label{notation: constant prolongation}
For each $\phi \in \mathcal{M}^1([-r, 0], \mathbb{K}^n)$, let $\Bar{\phi} \colon \dom(\phi) \cup [0, \infty) \to \mathbb{K}^n$ be the function defined by
\begin{equation}\label{eq: static prolongation}
	\Bar{\phi}(t) \coloneqq
	\begin{cases}
		\phi(t) & (t \in \dom(\phi)), \\
		\phi(0) & (t \ge 0).
	\end{cases}
\end{equation}
$\Bar{\phi}$ is a constant prolongation of $\phi$.
\end{notation}

\begin{theorem}\label{thm: unique existence of mild sol}
For any $\phi \in \mathcal{M}^1([-r, 0], \mathbb{K}^n)$, the linear RFDE~\eqref{eq: linear RFDE} has a unique mild solution under the initial condition $x_0 = \phi$.
\end{theorem}

In the following, we give a proof based on an \textit{a priori} estimate.
See Chicone~\cite[Subsection 2.1]{Chicone_2003} for a similar argument.

\begin{proof}[Proof of Theorem~\ref{thm: unique existence of mild sol}]
We divide the proof into the following steps.

\vspace{0.5\baselineskip}
\noindent
\textbf{Step 1: Reduction to a continuous unknown function and derivation of an \textit{a priori} estimate.}
For a continuous prolongation $x \colon [-r, \infty) \supset \dom(x) \to \mathbb{K}^n$ of $\phi$, we consider the function $y \colon [-r, \infty) \to \mathbb{K}^n$ defined by
\begin{equation*}
	y(t) \coloneqq
	\begin{cases}
		x(t) - \Bar{\phi}(t) & (t \in \dom(x)), \\
		0 & (t \not\in \dom(x)).
	\end{cases}
\end{equation*}
Then $y$ is a continuous function satisfying $y_0 = 0$.
The problem of finding a mild solution $x \colon [-r, \infty) \supset \dom(x) \to \mathbb{K}^n$ of the linear RFDE~\eqref{eq: linear RFDE} under the initial condition $x_0 = \phi$ is reduced to find a continuous function $y \colon [-r, \infty) \to \mathbb{K}^n$ satisfying $y_0 = 0$ and
\begin{equation}\label{eq: mild solution, continuous unknown}
	y(t)
	= L\int_0^t (y + \Bar{\phi})_s \mspace{2mu} \mathrm{d}s
	= \int_0^t Ly_s \mspace{2mu} \mathrm{d}s + L\int_0^t \Bar{\phi}_s \mspace{2mu} \mathrm{d}s
\end{equation}
for all $t \ge 0$.
Here Corollary~\ref{cor: T*int_0^t x_s ds for continuous x} is used.
By noticing the following estimate from above
\begin{equation*}
	\abs{\int_0^t \Bar{\phi}(s + \theta) \mspace{2mu} \mathrm{d}s}
	\le \norm{\phi}_1 + t\abs{\phi(0)}
	\mspace{25mu}
	(t \ge 0),
\end{equation*}
a continuous function $y \colon [-r, \infty) \to \mathbb{K}^n$ satisfying $y_0 = 0$ and Eq.~\eqref{eq: mild solution, continuous unknown} must satisfy
\begin{equation*}
	\abs{y(t)}
	\le \norm{L}(\norm{\phi}_1 + t\abs{\phi(0)}) + \int_0^t \norm{L}\norm{y_s} \mspace{2mu} \mathrm{d}s
\end{equation*}
for all $t \ge 0$.
By applying Lemma~\ref{lem: generalized Gronwall's inequality and RFDEs},
\begin{equation*}
	\norm{y_t}
	\le \norm{L}(\norm{\phi}_1 + t\abs{\phi(0)})\mathrm{e}^{\norm{L}t}
\end{equation*}
holds for all $t \ge 0$.

\vspace{0.5\baselineskip}
\noindent
\textbf{Step 2: Setting of function space.}
For each $\gamma > \norm{L}$, Step 1 indicates that for a continuous function $y \colon [-r, \infty) \to \mathbb{K}^n$ satisfying $y_0 = 0$ and Eq.~\eqref{eq: mild solution, continuous unknown}, we have
\begin{equation*}
	\mathrm{e}^{-\gamma t}\norm{y_t}
	\le \norm{L}(\norm{\phi}_1 + t\abs{\phi(0)})\mathrm{e}^{(\norm{L} - \gamma)t}.
\end{equation*}
Here the right-hand side converges to $0$ as $t \to \infty$.
Therefore,
\begin{equation*}
	\norm{y}_\gamma
	\coloneqq \sup_{t \ge 0} \mspace{2mu} (\mathrm{e}^{-\gamma t}\norm{y_t})
	= \sup_{t \ge 0} \mspace{2mu} (\mathrm{e}^{-\gamma t}\abs{y(t)})
	< \infty
\end{equation*}
holds (see Lemma~\ref{lem: norm{argdot}_gamma} for the detail).
For each $\gamma > \norm{L}$, let $Y_\gamma$ be the linear subspace of $C([-r, \infty), \mathbb{K}^n)$ given by
\begin{equation*}
	Y_\gamma
	\coloneqq \Set{y \in C([-r, \infty), \mathbb{K}^n)}{\text{$y_0 = 0$, $\norm{y}_\gamma < \infty$}},
\end{equation*}
which is considered as a normed space endowed with the norm $\norm{\argdot}_\gamma$.
Then $Y_\gamma$ is a Banach space (see Lemma~\ref{lem: completeness of Y_gamma}).
We fix $\gamma > \norm{L}$ arbitrarily, and let $Y \coloneqq Y_\gamma$ and $\norm{\argdot}_Y \coloneqq \norm{\argdot}_\gamma$.

\vspace{0.5\baselineskip}
\noindent
\textbf{Step 3: Reduction to fixed point problem.}
We define a transformation $T \colon Y \to C([-r, \infty), \mathbb{K}^n)$ by $(Ty)_0 = 0$ and
\begin{equation*}
	(Ty)(t)
	\coloneqq \int_0^t Ly_s \mspace{2mu} \mathrm{d}s + L\int_0^t \Bar{\phi}_s \mspace{2mu} \mathrm{d}s
	\mspace{25mu}
	(t \ge 0).
\end{equation*}
We now claim that $T(Y) \subset Y$ holds.
Let $y \in Y$ be given.
In the same way as in Step 1,
\begin{equation*}
	\abs{Ty(t)}
	\le \norm{L}(\norm{\phi}_1 + t\abs{\phi(0)}) + \norm{L}\int_0^t \norm{y_s} \mspace{2mu} \mathrm{d}s
\end{equation*}
holds for all $t \ge 0$.
Since $\mathrm{e}^{-\gamma t}\norm{L}(\norm{\phi}_1 + t\abs{\phi(0)}) \to 0$ as $t \to \infty$, we only need to show
\begin{equation*}
	\sup_{t \ge 0} \mspace{2mu} \mathrm{e}^{-\gamma t}\int_0^t \norm{y_s} \mspace{2mu} \mathrm{d}s < \infty
\end{equation*}
in order to obtain $Ty \in Y$.
By the assumption of $y \in Y$, $\norm{y_t} \le \norm{y}_Y\mathrm{e}^{\gamma t}$ holds for all $t \ge 0$.
Therefore, we have
\begin{align*}
	\int_0^t \norm{y_s} \mspace{2mu} \mathrm{d}s
	\le \norm{y}_Y \int_0^t \mathrm{e}^{\gamma s} \mspace{2mu} \mathrm{d}s
	\le \frac{\norm{y}_Y}{\gamma} \mathrm{e}^{\gamma t}
	\mspace{25mu}
	(t \ge 0),
\end{align*}
which implies $\sup_{t \ge 0} \mspace{2mu} \mathrm{e}^{-\gamma t}\int_0^t \norm{y_s} \mspace{2mu} \mathrm{d}s < \infty$.
Thus, $Ty \in Y$ is concluded.

\vspace{0.5\baselineskip}
\noindent
\textbf{Step 4: Application of contraction mapping principle.}
We now claim that the mapping $T \colon Y \to Y$ is a contraction.
For any $y^1, y^2 \in Y$,
\begin{equation*}
	\mathrm{e}^{-\gamma t}\abs{Ty^1(t) - Ty^2(t)}
	\le \mathrm{e}^{-\gamma t}\norm{L} \int_0^t \norm{y^1_s - y^2_s} \mspace{2mu} \mathrm{d}s
\end{equation*}
holds.
Since we have
\begin{align*}
	\norm{y^1_s - y^2_s}
	&= \mathrm{e}^{\gamma s} \cdot \mathrm{e}^{-\gamma s}\norm{(y^1 - y^2)_s} \\
	&\le \mathrm{e}^{\gamma s}\norm{y^1 - y^2}_Y
\end{align*}
for the integrand in the right-hand side,
\begin{align*}
	\mathrm{e}^{-\gamma t}\norm{L} \int_0^t \norm{y^1_s - y^2_s} \mspace{2mu} \mathrm{d}s
	&\le \frac{\norm{L}}{\gamma}(1 - \mathrm{e}^{-\gamma t}) \norm{y^1 - y^2}_Y \\
	&\le \frac{\norm{L}}{\gamma} \norm{y^1 - y^2}_Y
\end{align*}
is concluded.
Therefore, $T \colon Y \to Y$ is a contraction.
By applying the contraction mapping principle, there exists a unique $y_* \in Y$ such that
\begin{equation*}
	Ty_* = y_*.
\end{equation*}
The function $x_* \colon [-r, \infty) \supset \dom(\phi) \cup [0, \infty) \to \mathbb{K}^n$ defined by
\begin{equation*}
	x_*(t) \coloneqq y(t) + \Bar{\phi}(t)
	\mspace{25mu}
	(t \in \dom(\phi) \cup [0, \infty))
\end{equation*}
is a mild solution of the linear RFDE~\eqref{eq: linear RFDE} under the initial condition $x_0 = \phi$.
The uniqueness follows by the above discussion.
\end{proof}

We hereafter use the following notation.

\begin{notation}
For each $\phi \in \mathcal{M}^1([-r, 0], \mathbb{K}^n)$, we denote the unique mild solution of the linear RFDE~\eqref{eq: linear RFDE} under the initial condition $x_0 = \phi$ by $x^L(\argdot; \phi) \colon \dom(\phi) \cup [0, \infty) \to \mathbb{K}^n$.
\end{notation}

We have the following corollary.

\begin{corollary}\label{cor: linearity wrt initial history}
Let $\alpha, \beta \in \mathbb{K}$ and $\phi, \psi \in \mathcal{M}^1([-r, 0], \mathbb{K}^n)$ be given.
Then for all $t \ge 0$,
\begin{equation}\label{eq: linearity of mild sol wrt initial history}
	x^L(t; \alpha\phi + \beta\psi)
	= \alpha x^L(t; \phi) + \beta x^L(t; \psi)
\end{equation}
holds.
\end{corollary}

\begin{proof}
Let $\chi \coloneqq \alpha\phi + \beta\psi \in \mathcal{M}^1([-r, 0], \mathbb{K}^n)$ and $x \colon [-r, \infty) \supset \dom(x) \to \mathbb{K}^n$ be the function defined by
\begin{equation*}
	\dom(x) \coloneqq \dom(\chi) \cup [0, \infty),
	\mspace{15mu}
	x(t) \coloneqq \alpha x^L(t; \phi) + \beta x^L(t; \psi).
\end{equation*}
Since the map $L$ and the Lebesgue integration are linear, $x$ is a mild solution of the linear RFDE~\eqref{eq: linear RFDE} under the initial condition $x_0 = \chi$ by the definition of mild solutions (see Definition~\ref{dfn: mild sol}).
Therefore, \eqref{eq: linearity of mild sol wrt initial history} is a consequence of Theorem~\ref{thm: unique existence of mild sol}.
\end{proof}

\subsection{Fundamental matrix solutions}

Since ODEs are special DDEs, it is natural to expect that the notions of fundamental systems of solutions and fundamental matrix solutions for linear ODEs are meaningful for DDEs in some way.
However, the solution space of the linear RFDE~\eqref{eq: linear RFDE} is infinite-dimensional.
Therefore, it is impossible to define these notions to \eqref{eq: linear RFDE} as a simple generalization.

A key to this consideration is to focus on a ``finite-dimensionality''.
For this purpose, we consider an ``instantaneous input'' as an initial history function.
We will use the following notation.

\begin{definition}\label{dfn: instantaneous input}
For each $\xi \in \mathbb{K}^n$, we define a function $\Hat{\xi} \colon [-r, 0] \to \mathbb{K}^n$ by
\begin{equation*}
	\Hat{\xi}(\theta) \coloneqq
	\begin{cases}
		0 & (\theta \in [-r, 0)), \\
		\xi & (\theta = 0).
	\end{cases}
\end{equation*}
$\Hat{0}$ is the constant function whose value is identically equal to the zero vector $0 \in \mathbb{K}^n$.
\end{definition}

Since $\Hat{\xi} \in \mathcal{M}^1([-r, 0], \mathbb{K}^n)$ for each $\xi \in \mathbb{K}^n$, one can consider the mild solution
\begin{equation*}
	x^L \bigl( \argdot; \Hat{\xi} \bigr) \colon [-r, \infty) \to \mathbb{K}^n
\end{equation*}
of the linear RFDE~\eqref{eq: linear RFDE} under the initial condition $x_0 = \Hat{\xi}$ from Theorem~\ref{thm: unique existence of mild sol}.
Then Corollary~\ref{cor: linearity wrt initial history} yields that the subset $\mathcal{S}$ given by
\begin{equation*}
	\mathcal{S}
	\coloneqq \Set{x^L \bigl( \argdot; \Hat{\xi} \bigr) \colon [-r, \infty) \to \mathbb{K}^n}{\xi \in \mathbb{K}^n}
\end{equation*}
forms a linear space.
We have the following lemma.

\begin{lemma}\label{lem: linear independence of mild sol}
Let $\xi_1, \dots, \xi_m \in \mathbb{K}^n$ be vectors and let $x_j \coloneqq x^L \bigl( \argdot; \Hat{\xi}_j \bigr) \colon [-r, \infty) \to \mathbb{K}^n$ for each $j \in \{1, \dots, m\}$.
Then the following properties are equivalent:
\begin{enumerate}[label=(\alph*)]
\item The system of vectors $\xi_1, \dots, \xi_m$ is linearly independent.
\item The system of functions $x_1, \dots, x_m$ is linearly independent.
\end{enumerate}
\end{lemma}

Here the system of functions $x_1, \dots, x_m$ is said to be \textit{linearly independent} if for any scalars $\alpha_1, \dots, \alpha_m$, $\alpha_1x_1 + \dots + \alpha_mx_m = 0$ implies $\alpha_1 = \dots = \alpha_m = 0$.

\begin{proof}[Proof of Lemma~\ref{lem: linear independence of mild sol}]
(a) $\Rightarrow$ (b): Since $\alpha_1x_1 + \dots + \alpha_mx_m = 0$ implies
\begin{equation*}
	\alpha_1\xi_1 + \dots + \alpha_m\xi_m
	= (\alpha_1x_1 + \dots + \alpha_mx_m)(0)
	= 0,
\end{equation*}
this part follows by the definition of linear independence for functions.

(b) $\Rightarrow$ (a): We suppose $\alpha_1\xi_1 + \dots + \alpha_m\xi_m = 0$ for $\alpha_1, \dots, \alpha_m \in \mathbb{K}$.
Since this implies
\begin{equation*}
	\alpha_1\Hat{\xi}_1 + \dots + \alpha_m\Hat{\xi}_m = 0,
\end{equation*}
\eqref{eq: linearity of mild sol wrt initial history} yields
\begin{equation*}
	\alpha_1x_1 + \dots + \alpha_mx_m = 0.
\end{equation*}
Therefore, we have $\alpha_1 = \dots = \alpha_m = 0$ by the assumption (b).

This completes the proof.
\end{proof}

\begin{theorem}\label{thm: S is n-dimensional}
The linear space $\mathcal{S}$ is $n$-dimensional.
\end{theorem}

\begin{proof}
Let $\bm{b}_1, \dots, \bm{b}_n$ be a basis of $\mathbb{K}^n$.
From Lemma~\ref{lem: linear independence of mild sol}, the system of functions
\begin{equation*}
	x^L \bigl( \argdot; \Hat{\bm{b}}_1 \bigr), \dots, x^L \bigl( \argdot; \Hat{\bm{b}}_n \bigr) \in \mathcal{S}
\end{equation*}
is linearly independent.
Furthermore, for any $x_1, \dots, x_{n + 1} \in \mathcal{S}$, the system of functions is linearly dependent from Lemma~\ref{lem: linear independence of mild sol} because the system $x_1(0), \dots, x_{n + 1}(0) \in \mathbb{K}^n$ of vectors is linearly dependent.
Therefore, the statement holds.
\end{proof}

Theorem~\ref{thm: S is n-dimensional} naturally leads us to the following definition.

\begin{definition}[cf.\ \cite{Hale_1971_Springer}, \cite{Hale_1977}]\label{dfn: fundamental system of solutions and fundamental matrix sol}
We call a basis of the $n$-dimensional linear space $\mathcal{S}$ a \textit{fundamental system of solutions} to the linear RFDE~\eqref{eq: linear RFDE}.
Equivalently, a fundamental system of solutions is the linear independent system
\begin{equation*}
	x^L \bigl( \argdot; \Hat{\bm{b}}_1 \bigr), \dots, x^L \bigl( \argdot; \Hat{\bm{b}}_n \bigr) \colon [-r, \infty) \to \mathbb{K}^n
\end{equation*}
for some basis $\bm{b}_1, \dots, \bm{b}_n$ of $\mathbb{K}^n$.
We call a matrix-valued function having a fundamental system of solutions as its column vectors a \textit{fundamental matrix solution}.
In particular, we call the fundamental matrix solution
\begin{equation*}
	X \colon [-r, \infty) \to M_n(\mathbb{K})
\end{equation*}
satisfying $X(0) = I$ the \textit{principal fundamental matrix solution}.
Here $I$ denotes the identity matrix.
\end{definition}

The above definition is considered as a natural generalization of the corresponding definition for linear ODEs (see \cite[Definition~2.12 in Section~2.1 of Chapter~2]{Chicone_2006}).
See also \cite[Definition~5.10]{Walther_2020_preprint} for a related definition.

We hereafter use the following notation.

\begin{notation}
Let $X^L \colon [-r, \infty) \to M_n(\mathbb{K})$ denote the principal fundamental matrix solution of the linear RFDE~\eqref{eq: linear RFDE}.
By the above definition,
\begin{equation}\label{eq: X^L}
	X^L(\argdot) = \left( x^L(\argdot; \Hat{\bm{e}}_1) \mspace{5mu} \cdots \mspace{5mu} x^L(\argdot; \Hat{\bm{e}}_n) \right)
\end{equation}
holds.
Here $(\bm{e}_1, \dots, \bm{e}_n)$ denotes the standard basis of $\mathbb{K}^n$.
\end{notation}

\begin{remark}
Let $\xi = (\xi_1, \dots, \xi_n) \in \mathbb{K}^n$ be given.
From \eqref{eq: X^L}, we have
\begin{equation*}
	X^L(\argdot)\xi
	= \xi_1x^L(\argdot; \Hat{\bm{e}}_1) + \dots + \xi_nx^L(\argdot; \Hat{\bm{e}}_n).
\end{equation*}
Here the right-hand side is equal to $x^L(\argdot; \xi_1\Hat{\bm{e}}_1 + \dots + \xi_n\Hat{\bm{e}}_n)$ from \eqref{eq: linearity of mild sol wrt initial history}.
Therefore,
\begin{equation*}
	X^L(\argdot)\xi = x^L \bigl( \argdot; \Hat{\xi} \bigr)
\end{equation*}
holds.
\end{remark}

\begin{remark}
We consider an autonomous linear ODE
\begin{equation}\label{eq: linear ODE}
	\Dot{x} = Ax
\end{equation}
for some $A \in M_n(\mathbb{K})$.
For a system of global solutions $y_1, \dots, y_m \colon \mathbb{R} \to \mathbb{K}^n$ to the linear ODE~\eqref{eq: linear ODE}, the following statements are equivalent:
\begin{enumerate}[label=(\alph*)]
\item For any $t \in \mathbb{R}$, $y_1(t), \dots, y_m(t) \in \mathbb{K}^n$ is linearly independent.
\item For some $t_0 \in \mathbb{R}$, $y_1(t_0), \dots, y_m(t_0) \in \mathbb{K}^n$ is linearly independent.
\item The system of functions $y_1, \dots, y_m$ is linearly independent.
\end{enumerate}
The nontrivial part is (c) $\Rightarrow$ (a), which is proved by the principle of superposition and by the unique existence of a solution of \eqref{eq: linear ODE} under an initial condition
\begin{equation*}
	x(t_0) = \xi \in \mathbb{K}^n.
\end{equation*}
Compared with this situation, the linear independence of vectors $x_1(t_0), \dots, x_m(t_0) \in \mathbb{K}^n$ for each $t_0 > 0$ is not necessarily guaranteed for the functions $x_1, \dots, x_m$ in Lemma~\ref{lem: linear independence of mild sol} under the assumption that (a) or (b) in Lemma~\ref{lem: linear independence of mild sol} holds.
This should be compared with an example given by Popov~\cite{Popov_1972}, which is a three dimensional system of linear DDEs whose solution values are contained in a hyperplane of $\mathbb{R}^3$ after a certain amount of time has elapsed.
See also \cite[Section~3.5]{Hale_1977} and \cite[Section~3.5]{Hale--VerduynLunel_1993}.
\end{remark}

\subsection{Remarks}

\subsubsection{Consideration by Delfour}

The definition of a mild solution in Definition~\ref{dfn: mild sol} is also related to the consideration by Delfour~\cite{Delfour_1980}.
In that paper, the author considered a continuous linear map
\begin{equation*}
	L \colon W^{1, p}((-r, 0), \mathbb{R}^n) \to \mathbb{R}^n
\end{equation*}
for some $p \in [1, \infty)$.
Here $W^{1, p}((-r, 0), \mathbb{R}^n)$ is the Sobolev space (e.g., see Brezis~\cite[Section 8.2]{Brezis_2011}).
The author used the integral representation of $L$ given by
\begin{equation}\label{eq: continuous linear map on W^{1, p}}
	L\phi
	\coloneqq \int_{-r}^0 [A_1(\theta)\phi(\theta) + A_2(\theta)\phi'(\theta)] \mspace{2mu} \mathrm{d}\theta,
\end{equation}
where $A_1, A_2 \colon (-r, 0) \to M_n(\mathbb{R})$ are $n \times n$ real matrix-valued $q$-integrable functions with $(1/p) + (1/q) = 1$.
For the first term of the right-hand side of \eqref{eq: continuous linear map on W^{1, p}}, we have
\begin{equation*}
	\int_0^t
		\biggl( \int_{-r}^0 A_1(\theta)x(s + \theta) \mspace{2mu} \mathrm{d}\theta \biggr)
	\mspace{2mu} \mathrm{d}s
	= \int_{-r}^0 A_1(\theta)
		\biggl( \int_0^t x(s + \theta) \mspace{2mu} \mathrm{d}s \biggr)
	\mspace{2mu} \mathrm{d}\theta
\end{equation*}
under the exchange of order of integration.
Here we have replaced $\phi$ with $x_s$ and have integrated from $0$ to $t$ with respect to $s$.
In view of the above equality, it can be said that the concept of mild solutions in Definition~\ref{dfn: mild sol} is also hidden in \cite{Delfour_1980}.
Theorem~\ref{thm: unique existence of mild sol} and its proof should be compared with the existence and uniqueness result in \cite{Delfour_1980}.

\subsubsection{Mild solutions for linear differential difference equations}

We consider an autonomous linear \textit{differential difference equation}
\begin{equation}\label{eq: linear differential difference eq, multiple delays}
	\Dot{x}(t) = Ax(t) + \sum_{k = 1}^m B_kx(t - \tau_k)
	\mspace{25mu}
	(t \ge 0)
\end{equation}
for $n \times n$ matrices $A, B_1, \dots, B_m \in M_n(\mathbb{K})$ and $\tau_1, \dots, \tau_m \in (0, r]$.
We refer the reader to \cite{Bellman--Cooke_1963} as a general reference of the theory of differential difference equations.

The linear DDE~\eqref{eq: linear differential difference eq, multiple delays} can be expressed in the form of the linear RFDE~\eqref{eq: linear RFDE} by defining a continuous linear map $L \colon C([-r, 0], \mathbb{K}^n) \to \mathbb{K}^n$ by
\begin{equation}\label{eq: L, linear differential difference eq, multiple delays}
	L\psi = A\psi(0) + \sum_{k = 1}^m B_k\psi(-\tau_k)
\end{equation}
for $\psi \in C([-r, 0], \mathbb{K}^n)$.
Let $\phi \in \mathcal{M}^1([-r, 0], \mathbb{K}^n)$ be given and $x \coloneqq x^L(\argdot; \phi)$ for the above continuous linear map $L$.
By the definition of mild solutions (see Definitions~\ref{dfn: mild sol} and \ref{dfn: int_0^t x_s ds}), $x$ satisfies
\begin{align*}
	x(t)
	&= \phi(0) + L\int_0^t x_s \mspace{2mu} \mathrm{d}s \\
	&= \phi(0) + A\int_0^t x(s) \mspace{2mu} \mathrm{d}s + \sum_{k = 1}^m B_k\int_{-\tau_k}^{t - \tau_k} x(s) \mspace{2mu} \mathrm{d}s
\end{align*}
for all $t \ge 0$.
Since the last term is equal to
\begin{equation*}
	\phi(0) + \int_0^t Ax(s) \mspace{2mu} \mathrm{d}s + \sum_{k = 1}^m \int_{-\tau_k}^{t - \tau_k} B_kx(s) \mspace{2mu} \mathrm{d}s,
\end{equation*}
$x$ also satisfies
\begin{equation*}
	\Dot{x}(t) = Ax(t) + \sum_{k = 1}^m B_kx(t - \tau_k)
	\mspace{25mu}
	(\text{a.e.\ $t \ge 0$})
\end{equation*}
by the Lebesgue differentiation theorem (see Subsection~\ref{subsec: motivation, Volterra operator and RS convolution}).

\section{Differential equation satisfied by principal fundamental matrix solution}\label{sec: DE satisfied by X^L}

In this section, we consider the linear RFDE~\eqref{eq: linear RFDE}
\begin{equation*}
	\Dot{x}(t) = Lx_t
	\mspace{25mu}
	(t \ge 0)
\end{equation*}
for a continuous linear map $L \colon C([-r, 0], \mathbb{K}^n) \to \mathbb{K}^n$.
We choose a matrix-valued function $\eta \colon [-r, 0] \to M_n(\mathbb{K})$ of bounded variation so that $L$ is represented as the Riemann-Stieltjes integral~\eqref{eq: L as RS integral}
\begin{equation*}
	L\psi = \int_{-r}^0 \mathrm{d}\eta(\theta) \mspace{2mu} \psi(\theta)
\end{equation*}
for $\psi \in C([-r, 0], \mathbb{K}^n)$.
We recall that the domain of definition of $\eta$ is extended to $(-\infty, 0]$ by letting $\eta(\theta) \coloneqq \eta(-r)$ for $\theta \in (-\infty, -r]$.
We will use the following notation.

\begin{notation}
Let $\Check{\eta} \colon [0, \infty) \to M_n(\mathbb{K})$ be the function given by
\begin{equation*}
	\Check{\eta}(u) \coloneqq -\eta(-u)
\end{equation*}
for $u \in [0, \infty)$.
\end{notation}

In this paper, a function defined on $[0, \infty)$ is said to be of \textit{locally bounded variation} if it is of bounded variation on any closed and bounded interval of $[0, \infty)$.
A function of locally bounded variation is also called a \textit{locally BV function}.
Then the above function $\Check{\eta}$ is a function of locally bounded variation whose value is constant on $[r, \infty)$.
It is related to the reversal formula for Riemann-Stieltjes integrals (see Theorem~\ref{thm: reversal formula for RS integrals}).

It will be turned out that the notions of Volterra operator and Riemann-Stieltjes convolution are useful to deduce a differential equation that is satisfied by the principal fundamental matrix solution $X^L \colon [-r, \infty) \to M_n(\mathbb{K})$ of the linear RFDE~\eqref{eq: linear RFDE}.

\subsection{Definitions}\label{subsec: definitions of Volterra op and RS convolution}

\begin{definition}\label{dfn: Volterra operator}
For each $f \in \mathcal{L}^1_\mathrm{loc}([0, \infty), M_n(\mathbb{K}))$, let $Vf \colon [0, \infty) \to M_n(\mathbb{K})$ be the function defined by
\begin{equation}\label{eq: Volterra operator}
	(Vf)(t) = \int_0^t f(s) \mspace{2mu} \mathrm{d}s.
\end{equation}
Here the right-hand side is a Lebesgue integral.
We call $V$ the \textit{Volterra operator}.
\end{definition}

For details related to the Volterra operator as a linear operator on $C([0, T], \mathbb{K})$ for each $T > 0$, see \cite{Shapiro_2018}.
By using the Lebesgue differentiation theorem (e.g., see \cite[Theorem 1.3 in Section 1 of Chapter 3]{Stein--Shakarchi_2005}) component-wise, it holds that $Vf$ is \textit{locally absolutely continuous} (i.e., locally absolutely continuous on any closed and bounded interval of $[0, \infty)$), differentiable almost everywhere on $[0, \infty)$, and
\begin{equation*}
	(Vf)'(t) = f(t)
\end{equation*}
holds for almost all $t \in [0, \infty)$.

\begin{definition}\label{dfn: RS convolution}
For each function $\alpha \colon [0, \infty) \to M_n(\mathbb{K})$ of locally bounded variation and for each continuous function $f \colon [0, \infty) \to M_n(\mathbb{K})$, we define a function $\mathrm{d}\alpha * f \colon [0, \infty) \to M_n(\mathbb{K})$ by
\begin{equation*}
	(\mathrm{d}\alpha * f)(t)
	\coloneqq \int_0^t \mathrm{d}\alpha(u) \mspace{2mu} f(t - u).
\end{equation*}
Here the right-hand side is a Riemann-Stieltjes integral.
This function is called a \textit{Riemann-Stieltjes convolution}.
\end{definition}

See \cite[Definition 10.3 in Section 10.3]{Shapiro_2018} for the scalar-valued case.
The above definition should be compared with the treatment in \cite[Eq.~(2.13) in Chapter 2]{VerduynLunel_1989} and \cite[Corollary~2.5 in Section~I.2 of Appendix~I]{Diekmann--vanGils--Lunel--Walther_1995}, where an appearing integral is not a Riemann-Stieltjes integral but a Lebesgue-Stieltjes integral.

\subsection{Motivation}\label{subsec: motivation, Volterra operator and RS convolution}

The following lemma motivates the use of Volterra operator and Riemann-Stieltjes convolution.

\begin{lemma}\label{lem: L*int_0^t x_s ds, instantaneous input}
If $x \in \mathcal{L}^1_\mathrm{loc}([-r, \infty), \mathbb{K}^n)$ satisfies $x_0 = \Hat{\xi}$ for some $\xi \in \mathbb{K}^n$, then
\begin{equation*}
	L\int_0^t x_s \mspace{2mu} \mathrm{d}s
	= \int_{-t}^0 \mathrm{d}\eta(\theta) \mspace{2mu}
		\biggl( \int_0^{t + \theta} x(s) \mspace{2mu} \mathrm{d}s \biggr)
	= \int_0^t \mathrm{d}\Check{\eta}(u) \mspace{2mu}
		\biggl( \int_0^{t - u} x(s) \mspace{2mu} \mathrm{d}s \biggr)
\end{equation*}
holds for all $t \ge 0$.
\end{lemma}

\begin{proof}
Let $t > 0$ be fixed.
Since $\int_\theta^0 x(s) \mspace{2mu} \mathrm{d}s = 0$ by the assumption, we have
\begin{equation*}
	L\int_0^t x_s \mspace{2mu} \mathrm{d}s
	= \int_{-r}^0 \mathrm{d}\eta(\theta) \mspace{2mu}
		\biggl( \int_0^{t + \theta} x(s) \mspace{2mu} \mathrm{d}s \biggr).
\end{equation*}
We examine the right-hand side by dividing the consideration into the following cases:
\begin{itemize}
\item Case: $t \in [0, r)$.
In this case, $t + \theta \ge 0$ is equivalent to $\theta \ge -t$ for each $\theta \in [-r, 0]$.
Since $\int_0^{t + \theta} x(s) \mspace{2mu} \mathrm{d}s = 0$ for $\theta \in [-r, -t)$,
\begin{equation*}
	L\int_0^t x_s \mspace{2mu} \mathrm{d}s
	= \int_{-t}^0 \mathrm{d}\eta(\theta) \mspace{2mu}
		\biggl( \int_0^{t + \theta} x(s) \mspace{2mu} \mathrm{d}s \biggr)
\end{equation*}
holds by the additivity of Riemann-Stieltjes integrals on sub-intervals.
\item Case: $t \ge r$.
In this case, $t + \theta \ge 0$ holds for all $\theta \in [-r, 0]$.
Since $\eta$ is constant on $[-t, -r]$,
\begin{equation*}
	L\int_0^t x_s \mspace{2mu} \mathrm{d}s
	= \int_{-t}^0 \mathrm{d}\eta(\theta) \mspace{2mu}
		\biggl( \int_0^{t + \theta} x(s) \mspace{2mu} \mathrm{d}s \biggr)
\end{equation*}
holds.
\end{itemize}
Therefore, the expressions of $L\int_0^t x_s \mspace{2mu} \mathrm{d}s$ are obtained in combination with the reversal formula for Riemann-Stieltjes integrals (see Theorem~\ref{thm: reversal formula for RS integrals}).
\end{proof}

\subsection{Properties of Volterra operator and Riemann-Stieltjes convolution}\label{subsec: properties of Volterra op and RS convolution}

Throughout this subsection, let $\alpha \colon [0, \infty) \to M_n(\mathbb{K})$ be a function of locally bounded variation and $f \colon [0, \infty) \to M_n(\mathbb{K})$ be a continuous function.

\subsubsection{Continuity and local integrability}

The following is a simple result about the continuity of Riemann-Stieltjes convolution.

\begin{lemma}\label{lem: continuity of RS convolution}
If $f(0) = O$, then $\mathrm{d}\alpha * f$ is continuous.
\end{lemma}

\begin{proof}
We extend the domain of definition of $f$ to $\mathbb{R}$ by defining $f(t) \coloneqq f(0) = O$ for $t \le 0$.
Then the obtained function $f \colon \mathbb{R} \to M_n(\mathbb{K})$ is continuous.
Let $s, t \in [0, \infty)$ be given so that $s < t$.
By the additivity of Riemann-Stieltjes integrals on sub-intervals,
\begin{align*}
	(\mathrm{d}\alpha * f)(s)
	&= \int_0^s \mathrm{d}\alpha(u) \mspace{2mu} f(s - u) \\
	&= \int_0^t \mathrm{d}\alpha(u) \mspace{2mu} f(s - u) - \int_s^t \mathrm{d}\alpha(u) \mspace{2mu} f(s - u)
\end{align*}
holds.
Since
\begin{equation*}
	\int_s^t \mathrm{d}\alpha(u) \mspace{2mu} f(s - u)
	= [\alpha(t) - \alpha(s)]f(0)
	= O,
\end{equation*}
we have
\begin{equation*}
	(\mathrm{d}\alpha * f)(t) - (\mathrm{d}\alpha * f)(s)
	= \int_0^t \mathrm{d}\alpha(u) \mspace{2mu} [f(t - u) - f(s - u)].
\end{equation*}
By combining this and the uniform continuity of $f$ on closed and bounded intervals, the continuity of $\mathrm{d}\alpha * f$ is obtained.
\end{proof}

See \cite[Lemma~10.4 in Section~10.3]{Shapiro_2018} for the corresponding result for scalar-valued functions.
In this paper, we say that a function is \textit{locally Riemann integrable} if it is Riemann integrable on any closed and bounded interval.

\begin{theorem}\label{thm: local integrability of RS convolution}
$\mathrm{d}\alpha * f$ is a sum of a continuous function and a function of locally bounded variation.
Consequently, $\mathrm{d}\alpha * f$ is locally Riemann integrable.
\end{theorem}

\begin{proof}
By using $f = (f - f(0)) + f(0)$, we have
\begin{equation}\label{eq: decomposition of RS convolution}
	\mathrm{d}\alpha * f
	= \mathrm{d}\alpha * (f - f(0)) + \mathrm{d}\alpha * f(0).
\end{equation}
The first term in the right-hand side is continuous from Lemma~\ref{lem: continuity of RS convolution}.
The second term is of locally bounded variation since
\begin{equation*}
	(\mathrm{d}\alpha * f(0))(t)
	= [\alpha(t) - \alpha(0)]f(0)
\end{equation*}
holds for all $t \ge 0$.
Therefore, the conclusion holds.
\end{proof}

\begin{remark}\label{rmk: continuity of RS convolution}
Theorem~\ref{thm: local integrability of RS convolution} yields that $V(\mathrm{d}\alpha * f)$ makes sense.
Furthermore if $\alpha$ is continuous, then \eqref{eq: decomposition of RS convolution} shows that $\mathrm{d}\alpha * f$ is also continuous.
\end{remark}

\subsubsection{Riemann-Stieltjes convolution under Volterra operator}

The Riemann-Stieltjes convolution and Volterra operator are related in the following way.

\begin{theorem}\label{thm: RS convolution under Volterra operator}
The equality
\begin{equation}\label{eq: RS convolution under Volterra operator}
	V(\mathrm{d}\alpha * f) = \mathrm{d}\alpha * Vf
\end{equation}
holds.
Consequently, $\mathrm{d}\alpha * Vf$ is locally absolutely continuous, differentiable almost everywhere, and satisfies
\begin{equation*}
	(\mathrm{d}\alpha * Vf)'(t) = (\mathrm{d}\alpha * f)(t)
\end{equation*}
holds for almost all $t \ge 0$.
\end{theorem}

For the proof, we need the following theorem.
It contains the result on iterated Riemann integrals for continuous functions on rectangles as a special case.

\begin{theorem}\label{thm: iterated integrals for RS integral}
Let $[a, b]$ and $[c, d]$ be closed and bounded intervals of $\mathbb{R}$, $f \colon [a, b] \times [c, d] \to M_n(\mathbb{K})$ be a continuous function, and $\alpha \colon [a, b] \to M_n(\mathbb{K})$ be a function of bounded variation.
Then
\begin{equation}\label{eq: iterated RS integrals}
	\int_a^b \mathrm{d}\alpha(x) \mspace{2mu}
		\biggl( \int_c^d f(x, y) \mspace{2mu} \mathrm{d}y \biggr)
	= \int_c^d
		\biggl( \int_a^b \mathrm{d}\alpha(x) \mspace{2mu} f(x, y) \biggr)
	\mspace{2mu} \mathrm{d}y
\end{equation}
holds.
\end{theorem}

See also \cite[Theorem~15a in Section~15 of Chapter~I]{Widder_1941}.
We will give the proof in Appendix~\ref{subsec: proof of thm on iterated integrals}.

\begin{proof}[Proof of Theorem~\ref{thm: RS convolution under Volterra operator}]
We extend the domain of definition of $f$ to $\mathbb{R}$ by defining $f(t) \coloneqq f(0)$ for $t \le 0$.
By the proof of Lemma~\ref{lem: continuity of RS convolution}, we have
\begin{equation*}
	V(\mathrm{d}\alpha * f)(t)
	= \int_0^t
		\biggl( \int_0^t \mathrm{d}\alpha(u) \mspace{2mu} f(s - u) \biggr)
	\mspace{2mu} \mathrm{d}s
	- \int_0^t [\alpha(t) - \alpha(s)]f(0) \mspace{2mu} \mathrm{d}s,
\end{equation*}
where
\begin{align*}
	\int_0^t
		\biggl( \int_0^t \mathrm{d}\alpha(u) \mspace{2mu} f(s - u) \biggr)
	\mspace{2mu} \mathrm{d}s
	= \int_0^t \mathrm{d}\alpha(u)
		\biggl( \int_0^t f(s - u) \mspace{2mu} \mathrm{d}s \biggr)
\end{align*}
holds from Theorem~\ref{thm: iterated integrals for RS integral}.
The last term is expressed by
\begin{equation*}
	(\mathrm{d}\alpha * Vf)(t) + \int_0^t \mathrm{d}\alpha(u) \int_{-u}^0 f(s) \mspace{2mu} \mathrm{d}s
\end{equation*}
by using the Volterra operator and the Riemann-Stieltjes convolution.
Since $\int_{-u}^0 f(s) \mspace{2mu} \mathrm{d}s = uf(0)$ for $u \in [0, t]$, the proof is complete by showing
\begin{equation*}
	\int_0^t [\alpha(t) - \alpha(s)] \mspace{2mu} \mathrm{d}s
	= \int_0^t u \mspace{2mu} \mathrm{d}\alpha(u).
\end{equation*}
This is indeed true because
\begin{align*}
	\int_0^t u \mspace{2mu} \mathrm{d}\alpha(u)
	&= [u\alpha(u)]_{u = 0}^t - \int_0^t \alpha(u) \mspace{2mu} \mathrm{d}u \\
	&= t\alpha(t) - \int_0^t \alpha(u) \mspace{2mu} \mathrm{d}u
\end{align*}
holds by the integration by parts formula for Riemann-Stieltjes integrals.
\end{proof}

The following is a corollary of Theorem~\ref{thm: RS convolution under Volterra operator}.
It will not be used in the sequel.

\begin{corollary}\label{cor: loc BV of RS convolution}
Furthermore, if $f$ is continuously differentiable, then $\mathrm{d}\alpha * f$ is expressed by
\begin{equation*}
	\mathrm{d}\alpha * f
	= (\alpha - \alpha(0))f(0) + V \bigl( \mathrm{d}\alpha * f' \bigr).
\end{equation*}
Consequently, $\mathrm{d}\alpha * f$ is of locally bounded variation, differentiable almost everywhere, and satisfies
\begin{equation*}
	(\mathrm{d}\alpha * f)'(t)
	= \alpha'(t)f(0) + (\mathrm{d}\alpha * f')(t)
\end{equation*}
for almost all $t \ge 0$.
\end{corollary}

\begin{proof}
By the fundamental theorem of calculus, $f = f(0) + Vf'$ holds.
By combining this and \eqref{eq: RS convolution under Volterra operator}, the expression of $\mathrm{d}\alpha * f$ is obtained.
Since $V(\mathrm{d}\alpha * f')$ is locally absolutely continuous, it is also of locally bounded variation.
Therefore, the expression of $\mathrm{d}\alpha * f$ yields that $\mathrm{d}\alpha * f$ is of locally bounded variation.
The remaining properties are consequences of the fact that matrix-valued functions of bounded variation are differentiable almost everywhere.
This is obtained by applying the corresponding result\footnote{See \cite[Theorem 3.4 in Subsection 3.1 of Chapter 3]{Stein--Shakarchi_2005}, for example.} for real-valued functions component-wise.
\end{proof}

\subsection{Differential equation and principal fundamental matrix solution}

As an application of Theorem~\ref{thm: RS convolution under Volterra operator}, one can derive a differential equation that is satisfied by $x \bigl( \argdot; \Hat{\xi} \bigr)$ for each $\xi \in \mathbb{K}^n$.

\begin{theorem}\label{thm: mild sol with instantaneous input}
Let $x \coloneqq x^L \bigl( \argdot; \Hat{\xi} \bigr)$ for some $\xi \in \mathbb{K}^n$.
Then
\begin{equation}\label{eq: mild sol for instantaneous input, RS convolution}
	x(t)
	= \xi + \int_{-t}^0 \mathrm{d}\eta(\theta) \mspace{2mu}
		\biggl( \int_0^{t + \theta} x(s) \mspace{2mu} \mathrm{d}s \biggr)
	= \xi + \bigl( \mathrm{d}\Check{\eta} * Vx|_{[0, \infty)} \bigr)(t)
\end{equation}
holds for all $t \ge 0$.
Furthermore, $x|_{[0, \infty)}$ is locally absolutely continuous, differentiable almost everywhere, and satisfies
\begin{equation}\label{eq: DE of mild sol for instantaneous input}
	\Dot{x}(t)
	= \int_{-t}^0 \mathrm{d}\eta(\theta) \mspace{2mu} x(t + \theta)
	= \bigl( \mathrm{d}\Check{\eta} * x|_{[0, \infty)} \bigr)(t)
\end{equation}
for almost all $t \in [0, \infty)$.
\end{theorem}

\begin{proof}
Eq.~\eqref{eq: mild sol for instantaneous input, RS convolution} is a consequence of the definition of mild solutions and Lemma~\ref{lem: L*int_0^t x_s ds, instantaneous input}.
Since $x|_{[0, \infty)}$ is continuous, Theorem~\ref{thm: RS convolution under Volterra operator}  and Eq.~\eqref{eq: mild sol for instantaneous input, RS convolution} yield that
\begin{equation*}
	x(t) = \xi + V \bigl( \mathrm{d}\Check{\eta} * x|_{[0, \infty)} \bigr)(t)
\end{equation*}
holds for all $t \ge 0$.
In combination with Theorem~\ref{thm: local integrability of RS convolution}, it holds that $x|_{[0, \infty)}$ is locally absolutely continuous, differentiable almost everywhere, and satisfies
\begin{equation*}
	\Dot{x}(t) = \bigl( \mathrm{d}\Check{\eta} * x|_{[0, \infty)} \bigr)(t)
\end{equation*}
for almost all $t \ge 0$.
The remaining expression in Eq.~\eqref{eq: DE of mild sol for instantaneous input} is a consequence of the reversal formula for Riemann-Stieltjes integrals.
\end{proof}

We obtain the following result as a direct consequence of Theorem~\ref{thm: mild sol with instantaneous input} and \eqref{eq: X^L}.
We omit the proof.

\begin{theorem}[cf.\ \cite{VerduynLunel_1989}]\label{thm: DE of X^L}
The principal fundamental matrix solution $X^L \colon [-r, \infty) \to M_n(\mathbb{K})$ of the linear RFDE~\eqref{eq: linear RFDE} satisfies
\begin{equation}\label{eq: Eq for X^L, RS convolution}
	X^L(t)
	= I + \int_{-t}^0 \mathrm{d}\eta(\theta) \mspace{2mu}
		\biggl( \int_0^{t + \theta} X^L(s) \mspace{2mu} \mathrm{d}s \biggr)
	= I + \bigl( \mathrm{d}\Check{\eta} * VX^L|_{[0, \infty)} \bigr)(t)
\end{equation}
for all $t \ge 0$.
Furthermore, $X^L|_{[0, \infty)}$ is locally absolutely continuous, differentiable almost everywhere, and satisfies
\begin{equation}\label{eq: DE for X^L, RS convolution}
	\Dot{X}^L(t)
	= \int_{-t}^0 \mathrm{d}\eta(\theta) \mspace{2mu} X^L(t + \theta)
	= \bigl( \mathrm{d}\Check{\eta} * X^L|_{[0, \infty)} \bigr)(t)
\end{equation}
for almost all $t \in [0, \infty)$.
\end{theorem}

Roughly speaking, Eq.~\eqref{eq: DE for X^L, RS convolution} is used as the defining equation of the fundamental matrix solution in \cite[(9.1) in Chapter~9]{VerduynLunel_1989}.

\section{Non-homogeneous linear RFDEs}\label{sec: non-homogeneous liner RFDEs}

In this section, we study a non-homogeneous linear RFDE~\eqref{eq: non-homogeneous linear RFDE}
\begin{equation*}
	\Dot{x}(t) = Lx_t + g(t)
	\mspace{25mu}
	(\text{a.e.\ $t \ge 0$})
\end{equation*}
for a continuous linear map $L \colon C([-r, 0], \mathbb{K}^n) \to \mathbb{K}^n$ and some $g \in \mathcal{L}^1_\mathrm{loc}([0, \infty), \mathbb{K}^n)$.

\subsection{Non-homogeneous linear RFDE and mild solutions}

It is natural to define the notion of mild solutions to Eq.~\eqref{eq: non-homogeneous linear RFDE} in the following way.

\begin{definition}\label{dfn: mild sol for non-homogeneous eq}
Let $t_0 \ge 0$ and $\phi \in \mathcal{M}^1([-r, 0], \mathbb{K}^n)$ be given.
We say that a function $x \colon [t_0 - r, \infty) \supset \dom(x) \to \mathbb{K}^n$ is a \textit{mild solution} of Eq.~\eqref{eq: non-homogeneous linear RFDE} under the initial condition $x_{t_0} = \phi$ if the following conditions are satisfied:
(i) $x_{t_0} = \phi$, (ii) $[t_0, \infty) \subset \dom(x)$, (iii) $x|_{[t_0, \infty)}$ is continuous, and (iv) for all $t \ge t_0$,
\begin{equation*}
	x(t) = \phi(0) + L\int_{t_0}^t x_s \mspace{2mu} \mathrm{d}s + \int_{t_0}^t g(s) \mspace{2mu} \mathrm{d}s
\end{equation*}
holds.
\end{definition}

We note that $\int_{t_0}^t x_s \mspace{2mu} \mathrm{d}s \in C([-r, 0], \mathbb{K}^n)$ is defined by
\begin{equation*}
	\biggl( \int_{t_0}^t x_s \mspace{2mu} \mathrm{d}s \biggr)(\theta)
	\coloneqq \int_{t_0}^t x(s + \theta) \mspace{2mu} \mathrm{d}s
	= \int_{t_0 + \theta}^{t + \theta} x(s) \mspace{2mu} \mathrm{d}s
\end{equation*}
for $\theta \in [-r, 0]$, and
\begin{equation*}
	\dom(x)
	= (t_0 + \dom(\phi)) \cup [t_0, \infty)
	= t_0 + (\dom(\phi) \cup [0, \infty))
\end{equation*}
holds for a mild solution of Eq.~\eqref{eq: non-homogeneous linear RFDE} under the initial condition $x_{t_0} = \phi$.

\begin{lemma}\label{lem: Caratheodory sol and non-homogeneous eq}
Let $t_0 \ge 0$ and $\phi \in C([-r, 0], \mathbb{K}^n)$ be given.
If $x \colon [t_0 - r, \infty) \to \mathbb{K}^n$ is a mild solution of Eq.~\eqref{eq: non-homogeneous linear RFDE} under the initial condition $x_{t_0} = \phi$, then $x$ satisfies
\begin{equation*}
	\Dot{x}(t) = Lx_t + g(t)
\end{equation*}
for almost all $t \ge t_0$.
\end{lemma}

\begin{proof}
By the translation, we may assume $t_0 = 0$.
Since $L \colon C([-r, 0], \mathbb{K}^n) \to \mathbb{K}^n$ is a bounded linear operator,
\begin{equation*}
	x(t) = \phi(0) + \int_0^t Lx_s \mspace{2mu} \mathrm{d}s + \int_0^t g(s) \mspace{2mu} \mathrm{d}s
\end{equation*}
holds for all $t \ge 0$ from Corollary~\ref{cor: T*int_0^t x_s ds for continuous x}.
Then the fundamental theorem of calculus and the Lebesgue differentiation theorem yield that $x|_{[0, \infty)}$ is differentiable almost everywhere and
\begin{equation*}
	\Dot{x}(t) = Lx_t + g(t)
\end{equation*}
holds for almost all $t \ge 0$.
\end{proof}

\begin{remark}\label{rmk: Caratheodory sol and mild sol}
Let $\mathbb{K} = \mathbb{R}$.
We assume that $\dom(g) = [0, \infty)$ and consider the function $F \colon [0, \infty) \times C([-r, 0], \mathbb{R}^n) \to \mathbb{R}^n$ defined by
\begin{equation*}
	F(t, \phi) \coloneqq L\phi + g(t).
\end{equation*}
Then $F$ satisfies the Carath\'{e}odory condition.
See \cite[Section~2.6 of Chapter~2]{Hale_1977} and \cite[Section~2.6 of Chapter~2]{Hale--VerduynLunel_1993} for the detail of the Carath\'{e}odory condition for RFDEs.
Lemma~\ref{lem: Caratheodory sol and non-homogeneous eq} shows that a mild solution $x \colon [t_0 - r, \infty) \to \mathbb{R}^n$ of Eq.~\eqref{eq: non-homogeneous linear RFDE} under the initial condition $x_{t_0} = \phi \in C([-r, 0], \mathbb{R}^n)$ is a solution (in the Carath\'{e}odory sense).
\end{remark}

\subsection{Integral equation with a general forcing term}

More generally, for a given $t_0 \ge 0$ and a given continuous function $G \colon [t_0, \infty) \to \mathbb{K}^n$ with $G(t_0) = 0$, we can discuss a solution of the following integral equation
\begin{equation}\label{eq: integral eq with G}
	x(t) = \phi(0) + L \int_{t_0}^t x_s \mspace{2mu} \mathrm{d}s + G(t)
	\mspace{25mu}
	(t \ge t_0)
\end{equation}
under an initial condition $x_{t_0} = \phi \in \mathcal{M}^1([-r, 0], \mathbb{K}^n)$.
Here the assumption $G(t_0) = 0$ is natural because the right-hand side of \eqref{eq: integral eq with G} is equal to
\begin{equation*}
	\phi(0) + G(t_0)
\end{equation*}
at $t = t_0$.
The notion of a solution of \eqref{eq: integral eq with G} can be defined in the similar way as in Definition~\ref{dfn: mild sol for non-homogeneous eq}.
The following theorem holds.

\begin{theorem}\label{thm: unique existence, integral eq with G}
Let $t_0 \ge 0$ be given.
Suppose that $G \colon [t_0, \infty) \to \mathbb{K}^n$ is a continuous function with $G(t_0) = 0$.
Then for any $\phi \in \mathcal{M}^1([-r, 0], \mathbb{K}^n)$, Eq.~\eqref{eq: integral eq with G} has a unique solution under the initial condition $x_{t_0} = \phi$.
\end{theorem}

The following proof should be compared with the proof of Theorem~\ref{thm: unique existence of mild sol}.

\begin{proof}[Proof of Theorem~\ref{thm: unique existence, integral eq with G}]
By the translation, it is sufficient to consider the case $t_0 = 0$.
We will solve the integral equation locally and will connect the obtained local solutions.
For this purpose, we need to consider an integral equation under the initial condition $x_\sigma = \psi$ for each $\sigma \ge 0$ and each $\psi \in \mathcal{M}^1([-r, 0], \mathbb{K}^n)$.
Here an appropriate forcing term is given by
\begin{equation*}
	G(t, \sigma) \coloneqq G(t) - G(\sigma)
\end{equation*}
for $t \ge \sigma$.
Then we are going to consider an integral equation
\begin{equation}\label{eq: integral eq with G, initial time sigma}
	x(t) = \psi(0) + L\int_\sigma^t x_s \mspace{2mu} \mathrm{d}s + G(t, \sigma)
	\mspace{25mu}
	(t \ge \sigma)
\end{equation}
under the initial condition $x_\sigma = \psi$.
The remainder of the proof is divided into the following steps.

\vspace{0.5\baselineskip}
\noindent
\textbf{Step 1: Existence and uniqueness of a local solution.}
We fix the above $\sigma$ and $\psi$.
By defining a continuous function $y \colon [-r, \infty) \to \mathbb{K}^n$ by
\begin{equation*}
	y(s) \coloneqq
	\begin{cases}
		x(\sigma + s) - \Bar{\psi}(s) & (\sigma + s \in \dom(x)), \\
		0 & (\sigma + s \not\in \dom(x)),
	\end{cases}
\end{equation*}
Eq.~\eqref{eq: integral eq with G, initial time sigma} is transformed into
\begin{equation*}
	y(s) = \int_0^s Ly_u \mspace{2mu} \mathrm{d}u + L\int_0^s \bar{\psi}_u \mspace{2mu} \mathrm{d}u + G(\sigma + s, \sigma)
	\mspace{25mu}
	(s \ge 0),
\end{equation*}
which is an integral equation under the initial condition $y_0 = 0$.
We choose a constant $a > 0$ so that
\begin{equation*}
	\norm{L}a < 1
\end{equation*}
and consider a closed subset $Y$ of the Banach space $C([-r, a], \mathbb{K}^n)$ given by
\begin{equation*}
	Y \coloneqq \Set{y \in C([-r, a], \mathbb{K}^n)}{y_0 = 0}.
\end{equation*}
Furthermore, we define a transformation $T \colon Y \to Y$ by $(Ty)_0 = 0$ and
\begin{equation*}
	(Ty)(s)
	\coloneqq \int_0^s Ly_u \mspace{2mu} \mathrm{d}u + L\int_0^s \bar{\psi}_u \mspace{2mu} \mathrm{d}u + G(\sigma + s, \sigma)
	\mspace{25mu}
	(s \ge 0).
\end{equation*}
Then it holds that $T$ is contractive, and the application of the contraction mapping principle yields the unique existence of a fixed point $y_*$ of $T$.
By defining a function $x_* \colon [\sigma - r, \sigma + a] \to \mathbb{K}^n$ by
\begin{equation*}
	x_*(\sigma + s) \coloneqq y_*(s) + \Bar{\psi}(s)
	\mspace{25mu}
	(s \in \dom(\psi) \cup [0, \infty)),
\end{equation*}
it is concluded that $x_*$ is a solution of Eq.~\eqref{eq: integral eq with G, initial time sigma}.
We note that such a local solution is unique by the choice of the above $a$.

\vspace{0.5\baselineskip}
\noindent
\textbf{Step 2: Existence and uniqueness of a (global) solution.}
We note that the time $a > 0$ of existence of a local solution to Eq.~\eqref{eq: integral eq with G, initial time sigma} in Step 1 does not depend on the considered integral equation~\eqref{eq: integral eq with G, initial time sigma} and the specified initial condition $x_\sigma = \psi$.
In this step, we will show that by connecting these local solutions, we obtain a global solution.
For this purpose, for each $\sigma \ge 0$ and each $\psi \in \mathcal{M}^1([-r, 0], \mathbb{K}^n)$, let
\begin{equation*}
	x(\argdot; \sigma, \psi) \colon [\sigma - r, \sigma + a] \to \mathbb{K}^n
\end{equation*}
be the obtained unique solution of Eq.~\eqref{eq: integral eq with G, initial time sigma} under an initial condition $x_\sigma = \psi$.
We fix $\sigma$ and $\psi$.
Let
\begin{equation*}
	x \coloneqq x(\argdot; \sigma, \psi)
	\amd
	y \coloneqq x(\argdot; \sigma + a, x_{\sigma + a}).
\end{equation*}

We now claim that the function $z \colon [\sigma - r, \sigma + 2a] \to \mathbb{K}^n$ defined by
\begin{equation*}
	z(t) \coloneqq
	\begin{cases}
		x(t) & (t \in [\sigma - r, \sigma + a]) \\
		y(t) & (t \in [\sigma + a - r, \sigma + 2a])
	\end{cases}
\end{equation*}
is a solution to Eq.~\eqref{eq: integral eq with G, initial time sigma}.
We note that this definition makes sense because $y_{\sigma + a} = x_{\sigma + a}$.
To show the claim, it is sufficient to consider the case $t \in [\sigma + a, \sigma + 2a]$.
In this case, we have
\begin{equation*}
	z(t) = y(t)
	= x_{\sigma + a}(0) + L\int_{\sigma + a}^t y_s \mspace{2mu} \mathrm{d}s + G(t, \sigma + a),
\end{equation*}
where
\begin{equation*}
	x_{\sigma + a}(0)
	= x(\sigma + a)
	= \phi(0) + L\int_\sigma^{\sigma + a} x_s \mspace{2mu} \mathrm{d}s + G(\sigma + a, \sigma).
\end{equation*}
In the above equations, one can replace $y_s$ and $x_s$ with $z_s$.
Therefore, in view of
\begin{equation*}
	G(t, \sigma + a) + G(\sigma + a, \sigma) = G(t, \sigma),
\end{equation*}
it holds that $z$ is a solution of Eq.~\eqref{eq: integral eq with G, initial time sigma} under the initial condition $x_{t_0} = \phi$.

By repeating the above procedure, a global solution of the original integral equation~\eqref{eq: integral eq with G} is obtained.
By the uniqueness of each local solution, such a global solution is unique.
\end{proof}

\begin{remark}
Let $\mathbb{K} = \mathbb{R}$.
In \cite{Hale--Meyer_1967}, Hale and Meyer studied the following equation
\begin{equation*}
	x(t) = \phi(0) + g(t, x_t) - g(t_0, \phi) + \int_{t_0}^t f(s, x_s) \mspace{2mu} \mathrm{d}s + \int_{t_0}^t h(s) \mspace{2mu} \mathrm{d}s
\end{equation*}
under an initial condition $x_{t_0} = \phi \in C([-r, 0], \mathbb{R}^n)$ for each $t_0 \in \mathbb{R}$.
Here $f, g \colon \mathbb{R} \times C([-r, 0], \mathbb{R}^n) \to \mathbb{R}^n$ are continuous maps with the properties that
\begin{equation*}
	C([-r, 0], \mathbb{R}^n) \ni \phi \mapsto f(t, \phi) \in \mathbb{R}^n
	\amd
	C([-r, 0], \mathbb{R}^n) \ni \phi \mapsto g(t, \phi) \in \mathbb{R}^n
\end{equation*}
are linear for each $t \in \mathbb{R}$, and $h \colon \mathbb{R} \to \mathbb{R}^n$ is a locally  Lebesgue integrable function.
In \cite[Theorem 1 in Chapter II]{Hale--Meyer_1967}, it is shown that the above problem has a unique solution under an additional assumption of the non-atomicity of $g$ at $0$.
See \cite[Chapter I]{Hale--Meyer_1967} for the detail of this condition.
The proof of Theorem~\ref{thm: unique existence, integral eq with G} should be compared with \cite[Proof of Theorem 1 in Chapter II]{Hale--Meyer_1967}.
\end{remark}

We hereafter use the following notation.

\begin{notation}
Let $G \colon [0, \infty) \to \mathbb{K}^n$ be a continuous function with $G(0) = 0$ and $\phi \in \mathcal{M}^1([-r, 0], \mathbb{K}^n)$ be given.
The unique solution of Eq.~\eqref{eq: integral eq with G, t_0 = 0}
\begin{equation*}
	x(t) = \phi(0) + L \int_0^t x_s \mspace{2mu} \mathrm{d}s + G(t)
	\mspace{25mu}
	(t \ge 0)
\end{equation*}
is denoted by $x^L(\argdot; \phi, G) \colon [-r, \infty) \to \mathbb{K}^n$.
Then $x^L(\argdot; \phi, 0) = x^L(\argdot; \phi)$.
\end{notation}

We obtain the following corollary.
It will be a basics to consider a variation of constants formula for Eq.~\eqref{eq: integral eq with G, t_0 = 0}.

\begin{corollary}\label{cor: decomposition of sol to integral eq with G}
For any $\phi \in \mathcal{M}^1([-r, 0], \mathbb{K}^n)$ and any continuous function $G \colon [0, \infty) \to \mathbb{K}^n$ with $G(0) = 0$,
\begin{equation*}
	x^L(\argdot; \phi, G)
	= x^L(\argdot; \phi, 0) + x^L(\argdot; 0, G)
\end{equation*}
holds.
\end{corollary}

\begin{proof}
Let $x \coloneqq x^L(\argdot; \phi, 0) + x^L(\argdot; 0, G)$.
Then $x$ satisfies $x_0 = \phi$.
Furthermore, we have
\begin{align*}
	x(t)
	&= x^L(t; \phi, 0) + x^L(t; 0, G) \\
	&= \phi(0) + L\int_0^t x^L(\argdot; \phi, 0)_s \mspace{2mu} \mathrm{d}s + L\int_0^t x^L(\argdot; 0, G)_s \mspace{2mu} \mathrm{d}s + G(t)
\end{align*}
for all $t \ge 0$.
Since the last term is equal to
\begin{equation*}
	\phi(0) + L\int_0^t x_s \mspace{2mu} \mathrm{d}s + G(t)
\end{equation*}
by the linearity of $L$, Theorem~\ref{thm: unique existence, integral eq with G} yields $x = x^L(\argdot; \phi, G)$.
\end{proof}

In the same way as in Theorem~\ref{thm: mild sol with instantaneous input}, we obtain the following theorem.
The proof can be omitted.

\begin{theorem}\label{thm: sol to integral eq with G and instantaneous input}
Let $G \colon [0, \infty) \to \mathbb{K}^n$ be a continuous function with $G(0) = 0$ and $x \coloneqq x^L \bigl( \argdot; \Hat{\xi}, G \bigr)$ for some $\xi \in \mathbb{K}^n$.
Then $x$ satisfies
\begin{equation*}
	x(t)
	= \xi
	+ \int_{-t}^0 \mathrm{d}\eta(\theta) \mspace{2mu}
		\biggl( \int_0^{t + \theta} x(s) \mspace{2mu} \mathrm{d}s \biggr)
	+ G(t)
	= \xi + \bigl( \mathrm{d}\Check{\eta} * Vx|_{[0, \infty)} \bigr)(t) + G(t)
\end{equation*}
for all $t \ge 0$.
\end{theorem}

\section{Convolution and Volterra operator}\label{sec: convolution and Volterra op}

\subsection{A motivation to introduce convolution}\label{subsec: motivation of convolution}

\subsubsection{Variation of constants formula for non-homogeneous linear ODEs}

As a motivation to introduce convolution for locally Riemann integrable functions on $[0, \infty)$, we first recall the variation of constants formula for a non-homogeneous linear ODE
\begin{equation}\label{eq: non-homogeneous linear ODE}
	\Dot{x} = Ax + f(t)
\end{equation}
for an $n \times n$ matrix $A \in M_n(\mathbb{K})$ and a continuous function $f \colon \mathbb{R} \to \mathbb{K}^n$.
The unique global solution $x^A(\argdot; t_0, \xi, f) \colon \mathbb{R} \to \mathbb{K}^n$ of Eq.~\eqref{eq: non-homogeneous linear ODE} satisfying an initial condition $x(t_0) = \xi \in \mathbb{K}^n$ is expressed by
\begin{equation}\label{eq: VOC formula for non-homogeneous linear ODE}
	x^A(t; t_0, \xi, f)
	= \mathrm{e}^{tA}
		{\biggl[ \mathrm{e}^{-t_0 A}\xi + \int_{t_0}^t \mathrm{e}^{-uA}f(u) \mspace{2mu} \mathrm{d}u \biggr]}
	\mspace{25mu}
	(t \in \mathbb{R})
\end{equation}
with the matrix exponential.
This is the \textit{variation of constants formula} for \eqref{eq: non-homogeneous linear ODE}, which is obtained by finding an equation of $y = y(t)$ under the change of variable $x(t) = \mathrm{e}^{tA}y(t)$.
Indeed, the function $y$ must satisfy an initial condition $y(t_0) = \mathrm{e}^{-t_0 A}\xi$ and
\begin{equation*}
	\Dot{y}(t) = \mathrm{e}^{-tA}f(t)
	\mspace{25mu}
	(t \in \mathbb{R}).
\end{equation*}
This procedure to derive the formula~\eqref{eq: VOC formula for non-homogeneous linear ODE} corresponds to replacing a constant vector $v \in \mathbb{K}^n$ in the general solution
\begin{equation*}
	x(t) = \mathrm{e}^{tA}v
\end{equation*}
for the linear ODE~\eqref{eq: linear ODE} with a vector-valued function $y = y(t)$.
This is the reason for the terminology of the variation of constants formula.

The above method to derive \eqref{eq: VOC formula for non-homogeneous linear ODE} should be called the \textit{method of variation of constants}.
Unfortunately, this method does not exist for a non-homogeneous linear RFDE~\eqref{eq: non-homogeneous linear RFDE} because the solution space of the linear RFDE~\eqref{eq: linear RFDE} is infinite-dimensional and \eqref{eq: linear RFDE} does not have a general solution.
Even if the method itself does not exist for \eqref{eq: non-homogeneous linear RFDE}, a formula similar to \eqref{eq: VOC formula for non-homogeneous linear ODE} if it exists will be useful to analyze the dynamics of RFDEs near equilibria.
For this purpose, a form
\begin{equation}\label{eq: VOC formula' for non-homogeneous linear ODE}
	x^A(t; \xi, f) = \mathrm{e}^{tA}\xi + \int_0^t \mathrm{e}^{(t - u)A}f(u) \mspace{2mu} \mathrm{d}u,
\end{equation}
which is equivalent to \eqref{eq: VOC formula for non-homogeneous linear ODE} is helpful.
Here the initial time $t_0$ is set to $0$, and it has been omitted in $x^A(t; \xi, f)$.
The first term of the right-hand side of \eqref{eq: VOC formula' for non-homogeneous linear ODE} is the solution of the linear ODE~\eqref{eq: linear ODE} under the initial condition $x(0) = \xi$.
Therefore, the second term of the right-hand side of \eqref{eq: VOC formula' for non-homogeneous linear ODE} is the solution of \eqref{eq: non-homogeneous linear ODE} under the initial condition $x(0) = 0$.
This can be checked directly by differentiating the second term as
\begin{align*}
	\dv{t} \int_0^t \mathrm{e}^{(t - u)A}f(u) \mspace{2mu} \mathrm{d}u
	&= \dv{t} \mspace{2mu}
		{\biggl[ \mathrm{e}^{tA}\int_0^t \mathrm{e}^{-uA}f(u) \mspace{2mu} \mathrm{d}u \biggr]} \\
	&= f(t) + A\mathrm{e}^{tA}\int_0^t \mathrm{e}^{-uA}f(u) \mspace{2mu} \mathrm{d}u.
\end{align*}
We note that this gives another proof of \eqref{eq: VOC formula' for non-homogeneous linear ODE}.

\subsubsection{Convolution and non-homogeneous linear RFDEs}

For a continuous linear map $L \colon C([-r, 0], \mathbb{K}^n) \to \mathbb{K}^n$ and a continuous function $f \colon [0, \infty) \to \mathbb{K}^n$, we consider the non-homogeneous linear RFDE~\eqref{eq: non-homogeneous linear RFDE, continuous forcing term}
\begin{equation*}
	\Dot{x}(t) = Lx_t + f(t)
	\mspace{25mu}
	(t \ge 0).
\end{equation*}
Since $\mathbb{R} \ni t \mapsto \mathrm{e}^{tA} \in M_n(\mathbb{K})$ is the \textit{principal fundamental matrix solution} of the linear ODE~\eqref{eq: linear ODE} in the sense that it is a matrix solution to \eqref{eq: linear ODE} and $\mathrm{e}^{tA}|_{t = 0}$ is the identity matrix, it is natural to ask whether the function $x(\argdot; f) \colon [-r, \infty) \to \mathbb{K}^n$ defined by $x(\argdot; f)_0 = 0$ and \eqref{eq: X^L * f}
\begin{equation*}
	x(t; f) \coloneqq \int_0^t X^L(t - u)f(u) \mspace{2mu} \mathrm{d}u
\end{equation*}
for $t \ge 0$ is a solution to Eq.~\eqref{eq: non-homogeneous linear RFDE, continuous forcing term}.
Here $X^L \colon [-r, \infty) \to M_n(\mathbb{K})$ is the principal fundamental matrix solution of the linear RFDE~\eqref{eq: linear RFDE}
\begin{equation*}
	\Dot{x}(t) = Lx_t
	\mspace{25mu}
	(t \ge 0).
\end{equation*}

In Theorem~\ref{thm: DE of X^L}, we obtained the differential equation that is satisfied by $X^L$.
However, it is not direct to prove that the function $x(\argdot; f)$ is a solution to \eqref{eq: non-homogeneous linear RFDE, continuous forcing term} by differentiating the right-hand side of \eqref{eq: X^L * f} as in the case of the non-homogeneous linear ODE~\eqref{eq: non-homogeneous linear ODE} because one cannot take the term $X^L(t)$ out of the integral.
This comes from the property that initial value problems of RFDEs cannot be solved backward in general.
Therefore, one needs to treat the integral of the right-hand side of \eqref{eq: X^L * f} as it is.

Such an integral is a convolution for locally (Riemann) integrable functions, which should be distinguished from the convolution for integrable functions.
The convolution for locally integrable functions has been used in the literature of DDEs.
For example, see \cite[Chapter 1]{Bellman--Cooke_1963} with the context of the Laplace transform.
The convolution is also used in \cite{VerduynLunel_1989} and \cite{Diekmann--vanGils--Lunel--Walther_1995}, however, the detail has been omitted there.

\subsection{Convolution and Riemann-Stieltjes convolution}\label{subsec: convolution and RS convolution}

In this subsection, we study a convolution of the following type.

\begin{definition}\label{dfn: convolution for locally Riemann integrable functions}
For each pair of locally Riemann integrable functions $f, g \colon [0, \infty) \to M_n(\mathbb{K})$, we define a function $g * f \colon [0, \infty) \to M_n(\mathbb{K})$ by
\begin{equation*}
	(g * f)(t)
	\coloneqq \int_0^t g(t - u)f(u) \mspace{2mu} \mathrm{d}u
	= \int_0^t g(u)f(t - u) \mspace{2mu} \mathrm{d}u
\end{equation*}
for $t \ge 0$.
Here the above integrals are Riemann integrals.
We call the function $g * f$ the \textit{convolution} of $g$ and $f$.
\end{definition}

See \cite[Section 5.3]{Shapiro_2018} for the convolution of continuous functions.
We note that when $f$ is a constant function, then
\begin{equation}\label{eq: g * f(0)}
	(g * f)(t)
	= \int_0^t g(u)f(0) \mspace{2mu} \mathrm{d}u
	= (Vg)(t)f(0)
\end{equation}
holds for all $t \ge 0$.
In the same way, $g * f = g(0)Vf$ holds when $g$ is constant.

\begin{lemma}[cf.\ \cite{Shapiro_2018}]\label{lem: continuity of convolution, loc Riemann integrable function and continuous function}
Let $f, g \colon [0, \infty) \to M_n(\mathbb{K})$ be locally Riemann integrable functions.
If $f$ is continuous, then $g * f$ is a sum of a continuous function and a locally absolutely continuous function.
\end{lemma}

\begin{proof}
By using \eqref{eq: g * f(0)},
\begin{equation*}
	g * f = g * (f - f(0)) + (Vg)f(0)
\end{equation*}
holds.
Therefore, the conclusion is obtained by showing that $g * f$ is continuous when $f(0) = O$.
We extend the domain of definition of $f$ to $\mathbb{R}$ by defining $f(t) \coloneqq f(0) = O$ for $t \le 0$.
Let $s, t \in [0, \infty)$ be given so that $s < t$.
By the same reasoning as in the proof of Lemma~\ref{lem: continuity of RS convolution}, we have
\begin{equation*}
	(g * f)(t) - (g * f)(s)
	= \int_0^t g(u)[f(t - u) - f(s - u)] \mspace{2mu} \mathrm{d}u.
\end{equation*}
By combining this and the uniform continuity of $f$ on closed and bounded intervals, the continuity of $g * f$ is obtained.
\end{proof}

\subsubsection{Convolution of locally BV functions and continuous functions}

By using Theorem~\ref{thm: RS convolution under Volterra operator}, one can obtain the following result on the regularity of convolution.

\begin{theorem}[cf.\ \cite{Thieme_2008}]\label{thm: convolution of loc BV function and continuous function}
If $f \colon [0, \infty) \to M_n(\mathbb{K})$ is continuous and $g \colon [0, \infty) \to M_n(\mathbb{K})$ is of locally bounded variation, then
\begin{equation}\label{eq: convolution under Volterra operator, RS convolution}
	g * f
	= g(0)Vf + \mathrm{d}g * (Vf)
	= V(g(0)f + \mathrm{d}g * f)
\end{equation}
holds.
Consequently, the convolution $g * f$ is locally absolutely continuous, differentiable almost everywhere, and satisfies
\begin{equation*}
	(g * f)'(t) = g(0)f(t) + (\mathrm{d}g * f)(t)
\end{equation*}
for almost all $t \ge 0$.
\end{theorem}

The above result is considered as the finite-dimensional version of \cite[Theorem 3.2]{Thieme_2008} (i.e., the case that the Banach space $X$ in \cite[Theorem 3.2]{Thieme_2008} is finite-dimensional) except the equality
\begin{equation*}
	g * f = g(0)Vf + \mathrm{d}g * (Vf).
\end{equation*}
In the following, we give a simpler proof of Theorem~\ref{thm: convolution of loc BV function and continuous function} based on Theorem~\ref{thm: RS convolution under Volterra operator}.

\begin{proof}[Proof of Theorem~\ref{thm: convolution of loc BV function and continuous function}]
Since $Vf$ is continuously differentiable and $(Vf)(0) = O$,
\begin{align*}
	[\mathrm{d}g * (Vf)](t)
	&= [g(u)(Vf)(t - u)]_{u = 0}^t + \int_0^t g(u)f(t - u) \mspace{2mu} \mathrm{d}u \\
	&= -g(0)(Vf)(t) + (g * f)(t)
\end{align*}
holds for all $t \ge 0$ by the integration by parts formula for Riemann-Stieltjes integrals and from Theorem~\ref{thm: RS integral wrt C^1-function}.
By combining the obtained equality
\begin{equation*}
	g * f = g(0)(Vf) + \mathrm{d}g * (Vf)
\end{equation*}
and Theorem~\ref{thm: RS convolution under Volterra operator}, the equality~\eqref{eq: convolution under Volterra operator, RS convolution} is obtained.
\end{proof}

\begin{remark}
Let $f \colon [0, \infty) \to M_n(\mathbb{K})$ be a continuous function and $g \colon [0, \infty) \to M_n(\mathbb{K})$ be a function of locally bounded variation.
By defining a function $V(\mathrm{d}g) \colon [0, \infty) \to M_n(\mathbb{K})$ by
\begin{equation*}
	V(\mathrm{d}g)(t) \coloneqq g(t) - g(0)
\end{equation*}
for $t \ge 0$, we have
\begin{equation*}
	V(\mathrm{d}g * f) = \mathrm{d}g * (Vf) = V(\mathrm{d}g) * f
\end{equation*}
from Theorems~\ref{thm: RS convolution under Volterra operator} and \ref{thm: convolution of loc BV function and continuous function}.
This formula is easy to remember.
We note that the above definition of $V(\mathrm{d}g)$ is reasonable because
\begin{equation*}
	\int_0^t \mathrm{d}g(u) = g(t) - g(0)
\end{equation*}
holds for all $t \ge 0$.
\end{remark}

We have the following corollaries.

\begin{corollary}\label{cor: convolution of loc BV function and continuous function under Volterra operator}
If $f \colon [0, \infty) \to M_n(\mathbb{K})$ is continuous and $g \colon [0, \infty) \to M_n(\mathbb{K})$ is of locally bounded variation, then
\begin{equation*}
	V(g * f) = g * (Vf)
\end{equation*}
holds.
\end{corollary}

\begin{proof}
From Theorems~\ref{thm: convolution of loc BV function and continuous function} and \ref{thm: RS convolution under Volterra operator}, we have
\begin{equation*}
	V(g * f)
	= g(0) \bigl( V^2f \bigr) + \mathrm{d}g * \bigl( V^2f \bigr),
\end{equation*}
where $V^2f \coloneqq V(Vf)$.
Since the right-hand side is equal to $g * (Vf)$ from Theorem~\ref{thm: convolution of loc BV function and continuous function}, the equality is obtained.
\end{proof}

\begin{corollary}\label{cor: convolution of loc AC function and continuous function}
Let $f \colon [0, \infty) \to M_n(\mathbb{K})$ be a continuous function and $g \colon [0, \infty) \to M_n(\mathbb{K})$ be a function of locally bounded variation.
Then the following statements hold:
\begin{enumerate}
\item If $g$ is continuous or $f(0) = O$, then $g * f$ is continuously differentiable and 
\begin{equation*}
	(g * f)' = g(0)f + \mathrm{d}g * f
\end{equation*}
holds.
\item If $g$ is locally absolutely continuous, then $g * f$ is continuously differentiable and 
\begin{equation*}
	(g * f)' = g(0)f + g' * f
\end{equation*}
holds.
Here $g' * f \colon [0, \infty) \to M_n(\mathbb{K})$ be the function defined by
\begin{equation*}
	(g' * f)(t)
	\coloneqq \int_0^t g'(t - u)f(u) \mspace{2mu} \mathrm{d}u
	= \int_0^t g'(u)f(t - u) \mspace{2mu} \mathrm{d}u
\end{equation*}
for $t \ge 0$, where the integrals are Lebesgue integrals.
\end{enumerate}
\end{corollary}

\begin{proof}
1. Under the assumption, $\mathrm{d}g * f$ is continuous from Lemma~\ref{lem: continuity of RS convolution} and Remark~\ref{rmk: continuity of RS convolution}.
Therefore, the conclusion follows by the formula~\eqref{eq: convolution under Volterra operator, RS convolution}.

\vspace{0.3\baselineskip}
2. The continuous differentiability of $g * f$ follows by the statement 1.
When $g$ is locally absolutely continuous,
\begin{equation*}
	\mathrm{d}g * f = g' * f
\end{equation*}
holds from Theorem~\ref{thm: RS integral of continuous function wrt AC function}.
\end{proof}

\subsubsection{Associativity of Riemann-Stieltjes convolution}

For the proof of Theorem~\ref{thm: VOC formula, phi = 0, G = Vf} below, we need the following result.

\begin{theorem}[refs.\ \cite{Gripenberg--Londen--Staffans_1990}, \cite{Shapiro_2018}]\label{thm: associativity for RS convolution}
Let $\alpha \colon [0, \infty) \to M_n(\mathbb{K})$ be a function of locally bounded variation.
Then for any continuous functions $f, g \colon [0, \infty) \to M_n(\mathbb{K})$,
\begin{equation}\label{eq: associativity for RS convolution}
	\mathrm{d}\alpha * (g * f) = (\mathrm{d}\alpha * g) * f
\end{equation}
holds.
\end{theorem}

\begin{remark}
Both sides of Eq.~\eqref{eq: associativity for RS convolution} are meaningful from Lemma~\ref{lem: continuity of convolution, loc Riemann integrable function and continuous function} and Theorem~\ref{thm: local integrability of RS convolution}.
Eq.~\eqref{eq: RS convolution under Volterra operator} is a special case of \eqref{eq: associativity for RS convolution} since we have
\begin{equation*}
	Vf = f * \mathcal{I} = \mathcal{I} * f
\end{equation*}
for any $f \in \mathcal{L}^1_\mathrm{loc}([0, \infty), M_n(\mathbb{K}))$.
\end{remark}

The above is a result on the associativity for Riemann-Stieltjes convolutions.
The corresponding statements in a more general setting are given in \cite[Section 6 in Chapter 3]{Gripenberg--Londen--Staffans_1990}.
See also \cite[Proposition D.9 in Appendix D]{Shapiro_2018} for a similar result to Theorem~\ref{thm: associativity for RS convolution}.
One can prove Theorem~\ref{thm: associativity for RS convolution} by the same reasoning in the proof of Theorem~\ref{thm: RS convolution under Volterra operator}, and therefore, we omit the proof.

\subsection{A formula for non-homogeneous equation with trivial initial history}

Let $L \colon C([-r, 0], \mathbb{K}^n) \to \mathbb{K}^n$ be a continuous linear map.
We recall that for a continuous map $G \colon [0, \infty) \to \mathbb{K}^n$ with $G(0) = 0$, the function $x^L(\argdot; 0, G) \colon [-r, \infty) \to \mathbb{K}^n$ denotes the unique solution of an integral equation
\begin{equation}\label{eq: integral eq with G, t_0 = 0, phi = 0}
	x(t) = L \int_0^t x_s \mspace{2mu} \mathrm{d}s + G(t)
	\mspace{25mu}
	(t \ge 0)
\end{equation}
under the initial condition $x_0 = 0$.

In this subsection, as an application of the results in Subsection~\ref{subsec: convolution and RS convolution}, we show that the function $x(\argdot; f) \colon [-r, \infty) \to \mathbb{K}^n$ defined by $x(\argdot; f)_0 = 0$ and \eqref{eq: X^L * f} is a solution to the non-homogeneous linear RFDE~\eqref{eq: non-homogeneous linear RFDE, continuous forcing term}.

\begin{theorem}[cf.\ \cite{Walther_1975}]\label{thm: VOC formula, phi = 0, G = Vf}
Let $f \colon [0, \infty) \to \mathbb{K}^n$ be a continuous function.
Then 
\begin{equation}\label{eq: VOC formula, phi = 0, G = Vf}
	x^L(t; 0, Vf) = \int_0^t X^L(t - u)f(u) \mspace{2mu} \mathrm{d}u
\end{equation}
holds for all $t \ge 0$.
\end{theorem}

We note that $x^L(\argdot; 0, Vf)$ is a solution to Eq.~\eqref{eq: non-homogeneous linear RFDE, continuous forcing term} (see Lemma~\ref{lem: Caratheodory sol and non-homogeneous eq}).

\begin{proof}[Proof of Theorem~\ref{thm: VOC formula, phi = 0, G = Vf}]
Let $x \coloneqq x(\argdot; f)|_{[0, \infty)}$ and $X \coloneqq X^L|_{[0, \infty)}$.
Since $X$ is locally absolutely continuous (see Theorem~\ref{thm: DE of X^L}) and $X(0) = I$,
\begin{equation*}
	x
	= X * f
	= Vf + \Dot{X} * (Vf)
\end{equation*}
holds from Theorems~\ref{thm: convolution of loc BV function and continuous function} and \ref{thm: RS integral of continuous function wrt AC function}.
For the term $\Dot{X} * (Vf)$, we have
\begin{align*}
	\Dot{X} * (Vf)
	&= (\mathrm{d}\Check{\eta} * X) * (Vf) \\
	&= \mathrm{d}\Check{\eta} * [X * (Vf)] \\
	&= \mathrm{d}\Check{\eta} * V(X * f)
\end{align*}
from Theorems~\ref{thm: DE of X^L}, \ref{thm: associativity for RS convolution}, and Corollary~\ref{cor: convolution of loc BV function and continuous function under Volterra operator}.
Therefore, the equality~\eqref{eq: VOC formula, phi = 0, G = Vf} is obtained from Theorems~\ref{thm: sol to integral eq with G and instantaneous input} and \ref{thm: unique existence, integral eq with G}.
\end{proof}

The above proof of Theorem~\ref{thm: VOC formula, phi = 0, G = Vf} is different from the proofs in the literature (e.g., see \cite[Section~4]{Walther_1975}).

\section{Variation of constants formula}\label{sec: VOC formula}

Let $L \colon C([-r, 0], \mathbb{K}^n) \to \mathbb{K}^n$ be a continuous linear map and $X^L \colon [-r, \infty) \to M_n(\mathbb{K})$ be the principal fundamental matrix solution of the linear RFDE~\eqref{eq: linear RFDE}
\begin{equation*}
	\Dot{x}(t) = Lx_t
	\mspace{25mu}
	(t \ge 0).
\end{equation*}
In this section, we obtain a ``variation of constants formula'' for the non-homogeneous linear RFDE~\eqref{eq: non-homogeneous linear RFDE}
\begin{equation*}
	\dot{x}(t) = Lx_t + g(t)
	\mspace{25mu}
	(\text{a.e.\ $t \ge 0$})
\end{equation*}
for some $g \in \mathcal{L}^1_\mathrm{loc}([0, \infty), \mathbb{K}^n)$ expressed by $X^L$.
In view of Corollary~\ref{cor: decomposition of sol to integral eq with G}, we will divide our consideration into the following steps:
\begin{itemize}
\item Step 1: To find a formula for the mild solution $x^L(\argdot; 0, Vg)$ of Eq.~\eqref{eq: non-homogeneous linear RFDE} under the initial condition $x_0 = 0$.
\item Step 2: To find a formula for the mild solution $x^L(\argdot; \phi, 0)$ of Eq.~\eqref{eq: linear RFDE} under the initial condition $x_0 = \phi \in \mathcal{M}^1([-r, 0], \mathbb{K}^n)$.
\end{itemize}
Then the full formula for the mild solution of \eqref{eq: non-homogeneous linear RFDE} under the initial condition $x_0 = \phi \in \mathcal{M}^1([-r, 0], \mathbb{K}^n)$ is obtained by combining the above formulas.
In Step 1, for a given continuous function $G \colon [0, \infty) \to \mathbb{K}^n$ with $G(0) = 0$, we indeed consider the integral equation~\eqref{eq: integral eq with G, t_0 = 0, phi = 0}
\begin{equation*}
	x(t) = L \int_0^t x_s \mspace{2mu} \mathrm{d}s + G(t)
	\mspace{25mu}
	(t \ge 0)
\end{equation*}
under the initial condition $x_0 = 0$ and try to find a formula for the solution $x^L(\argdot; 0, G)$ expressed by $X^L$.

\begin{remark}\label{rmk: integral eq with G, t_0 = 0, phi = 0}
Since $x_0 = 0 = \Hat{0}$, Eq.~\eqref{eq: integral eq with G, t_0 = 0, phi = 0} is equivalent to
\begin{equation*}
	x(t)
	= \int_{-t}^0 \mathrm{d}\eta(\theta) \mspace{2mu}
		\biggl( \int_0^{t + \theta} x(s) \mspace{2mu} \mathrm{d}s \biggr)
	+ G(t)
	\mspace{25mu}
	(t \ge 0)
\end{equation*}
from Lemma~\ref{lem: L*int_0^t x_s ds, instantaneous input}.
\end{remark}

The following is the main result of this section.

\begin{theorem}\label{thm: VOC formula, integral eq}
Let $G \colon [0, \infty) \to \mathbb{K}^n$ be a continuous function with $G(0) = 0$ and $\phi \in \mathcal{M}^1([-r, 0], \mathbb{K}^n)$ be given.
Then the mild solution $x^L(\argdot; \phi, G)$ of the integral equation~\eqref{eq: integral eq with G, t_0 = 0}
\begin{equation*}
	x(t) = \phi(0) + L \int_0^t x_s \mspace{2mu} \mathrm{d}s + G(t)
	\mspace{25mu}
	(t \ge 0)
\end{equation*}
under the initial condition $x_0 = \phi$ satisfies \eqref{eq: VOC formula, integral eq}
\begin{align*}
	&x^L(t; \phi, G) \notag \\
	&= X^L(t)\phi(0) + \bigl[ G^L(t; \phi) + G(t) \bigr] + \int_0^t \Dot{X}^L(t - u) \bigl[ G^L(u; \phi) + G(u) \bigr] \mspace{2mu} \mathrm{d}u
\end{align*}
for all $t \ge 0$.
\end{theorem}

We will call the formula~\eqref{eq: VOC formula, integral eq} the \textit{variation of constants formula} for Eq.~\eqref{eq: integral eq with G, t_0 = 0}.
The definition of the function $G^L(\argdot; \phi) \colon [0, \infty) \to \mathbb{K}^n$ for $\phi \in \mathcal{M}^1([-r, 0], \mathbb{K}^n)$ will be given later.
For this definition, the expression of $L$ by the Riemann-Stieltjes integral~\eqref{eq: L as RS integral}
\begin{equation*}
	L\psi = \int_{-r}^0 \mathrm{d}\eta(\theta) \mspace{2mu} \psi(\theta)
\end{equation*}
for $\psi \in C([-r, 0], \mathbb{K}^n)$ is a key tool.

\subsection{Motivation: Naito's consideration}\label{subsec: Naito's consideration}

We first concentrate our consideration to the case that $g \in C([0, \infty), \mathbb{K}^n)$ and $\phi \in C([-r, 0], \mathbb{K}^n)$.
From Theorem~\ref{thm: VOC formula, phi = 0, G = Vf}, we only need to find a formula for $x^L(\argdot; \phi, 0)$ in this case.

Naito~\cite[Theorem 6.5]{Naito_1979} has discussed an expression of the form
\begin{equation*}
	x(t)
	= \phi(0) + \int_0^t X(t - u)L\Bar{\phi}_u \mspace{2mu} \mathrm{d}u
	\mspace{25mu}
	(t \ge 0).
\end{equation*}
In the above formula, $x \colon [-r, \infty) \to \mathbb{K}^n$ is the solution of the linear RFDE~\eqref{eq: linear RFDE} under the initial condition $x_0 = \phi \in C([-r, 0], \mathbb{K}^n)$, and $\Bar{\phi} \colon [-r, \infty) \to \mathbb{K}^n$ is the function defined by
\begin{equation*}
	\Bar{\phi}(t) \coloneqq
	\begin{cases}
		\phi(t) & (t \in [-r, 0]), \\
		\phi(0) & (t \ge 0).
	\end{cases}
\end{equation*}
See also Notation~\ref{notation: constant prolongation}.
Although the study of \cite{Naito_1979} is in the setting of infinite retardation, we are now interpreting this in the setting of finite retardation (i.e., the history function space is $C([-r, 0], \mathbb{K}^n)$).
We note that the matrix-valued function $X \colon [0, \infty) \to M_n(\mathbb{K})$ is defined by using the inverse Laplace transform.
See \cite{Naito_1979} for the detail.
See also \cite{Naito_1980}, where an interpretation of the matrix-valued function $X$ is given.

In our setting, a formula expressed by the principal fundamental matrix solution $X^L$
\begin{equation}\label{eq: formula by Naito}
	x^L(t; \phi, 0)
	= \phi(0) + \int_0^t X^L(t - u)L\Bar{\phi}_u \mspace{2mu} \mathrm{d}u
	\mspace{25mu}
	(t \ge 0)
\end{equation}
is true.
To see this, let $y(t) \coloneqq x^L(t; \phi, 0) - \Bar{\phi}(t)$ for $t \in [-r, \infty)$.
Then the function $y \colon [-r, \infty) \to \mathbb{K}^n$ satisfies $y_0 = 0$ and
\begin{equation*}
	\Dot{y}(t) = Ly_t + L\Bar{\phi}_t
	\mspace{25mu}
	(t \ge 0).
\end{equation*}
See also the proof of Theorem~\ref{thm: unique existence of mild sol}.
Since the function $[0, \infty) \ni t \mapsto L\Bar{\phi}_t \in \mathbb{K}^n$ is continuous, we obtain
\begin{equation*}
	y(t) = \int_0^t X^L(t - u)L\Bar{\phi}_u \mspace{2mu} \mathrm{d}u
	\mspace{25mu}
	(t \ge 0)
\end{equation*}
by applying Theorem~\ref{thm: VOC formula, phi = 0, G = Vf}.

\subsection{Derivation of a general forcing term}\label{subsec: derivation of a general forcing term}

The formula~\eqref{eq: formula by Naito} is not sufficient for the application to the linearized stability.
See Section~\ref{sec: linearized stability} for the detail of the application of the variation of constants formula to the linearized stability.
We now introduce the following function.

\begin{notation}
For each $\phi \in \mathcal{M}^1([-r, 0], \mathbb{K}^n)$, we define a function $z^L(\argdot; \phi) \colon [-r, \infty) \to \mathbb{K}^n$ by $z^L(\argdot; \phi)_0 = 0$ and \eqref{eq: z^L(argdot; phi)}
\begin{equation*}
	z^L(t; \phi)
	\coloneqq x^L(t; \phi, 0) - X^L(t)\phi(0)
\end{equation*}
for $t \ge 0$.
\end{notation}

\begin{remark}
Since
\begin{equation*}
	z^L(0; \phi) = \phi(0) - X^L(0)\phi(0) = 0,
\end{equation*}
the function $z^L(\argdot; \phi)$ is continuous.
In view of $X^L(\argdot)\phi(0) = x^L \Bigl( \argdot; \widehat{\phi(0)}, 0 \Bigr)$, we also have
\begin{equation*}
	z^L(t; \phi) = x^L \Bigl(t; \phi - \widehat{\phi(0)}, 0 \Bigr)
	\mspace{25mu}
	(t \ge 0)
\end{equation*}
from Corollary~\ref{cor: linearity wrt initial history}.
We note that this equality is not valid for $t \in [-r, 0)$ because $z^L(\argdot; \phi)_0 = 0$.
\end{remark}

From the expression~\eqref{eq: mild sol'} for a mild solution, the function $z^L(\argdot; \phi)$ satisfies
\begin{equation*}
	z^L(t; \phi)
	= \int_{-r}^0 \mathrm{d}\eta(\theta) \mspace{2mu}
		\biggl( \int_\theta^0 \phi(s) \mspace{2mu} \mathrm{d}s \biggr)
	+ \int_{-r}^0 \mathrm{d}\eta(\theta) \mspace{2mu}
		\biggl( \int_0^{t + \theta} x^L \Bigl(s; \phi - \widehat{\phi(0)}, 0 \Bigr) \mspace{2mu} \mathrm{d}s \biggr)
\end{equation*}
for all $t \ge 0$.
The second term of the right-hand side is further calculated as follows:
\begin{itemize}
\item When $t \in [0, r)$, $\theta \in [-r, 0]$ satisfies $t + \theta \ge 0$ if and only if $\theta \in [-t, 0]$.
Since
\begin{equation*}
	x^L \Bigl(s; \phi - \widehat{\phi(0)}, 0 \Bigr) = \phi(s)
\end{equation*}
for $s \in \dom(\phi) \setminus \{0\}$, the second term is decomposed by
\begin{equation*}
	\int_{-r}^{-t} \mathrm{d}\eta(\theta) \mspace{2mu}
		\biggl( \int_0^{t + \theta} \phi(s) \mspace{2mu} \mathrm{d}s \biggr)
	+ \int_{-t}^0 \mathrm{d}\eta(\theta) \mspace{2mu}
		\biggl( \int_0^{t + \theta} z^L(s; \phi) \mspace{2mu} \mathrm{d}s \biggr)
\end{equation*}
by the additivity of Riemann-Stieltjes integrals on sub-intervals.
\item When $t \ge r$, the second term is equal to $\int_{-r}^0 \mathrm{d}\eta(\theta) \mspace{2mu} \bigl( \int_0^{t + \theta} z^L(s; \phi) \mspace{2mu} \mathrm{d}s \bigr)$.
\end{itemize}

This leads to the following definition.

\begin{definition}\label{dfn: G^L(argdot; phi)}
For each $\phi \in \mathcal{M}^1([-r, 0], \mathbb{K}^n)$, we define a function $G^L(\argdot; \phi) \colon [0, \infty) \to \mathbb{K}^n$ by
\begin{equation*}
	G^L(t; \phi)
	\coloneqq \int_{-r}^0 \mathrm{d}\eta(\theta) \mspace{2mu}
		\biggl( \int_\theta^0 \phi(s) \mspace{2mu} \mathrm{d}s \biggr)
	+ \int_{-r}^{-t} \mathrm{d}\eta(\theta) \mspace{2mu}
		\biggl(\int_0^{t + \theta} \phi(s) \mspace{2mu} \mathrm{d}s \biggr)
\end{equation*}
for $t \in [0, r)$ and
\begin{equation*}
	G^L(t; \phi)
	\coloneqq \int_{-r}^0 \mathrm{d}\eta(\theta) \mspace{2mu}
		\biggl( \int_\theta^0 \phi(s) \mspace{2mu} \mathrm{d}s \biggr)
\end{equation*}
for $t \in [r, \infty)$.
\end{definition}

By definition, $G^L(0; \phi) = 0$ holds.
Summarizing the above discussion, we obtain the following lemma.

\begin{lemma}\label{lem: integral eq for z^L(argdot; phi)}
For each $\phi \in \mathcal{M}^1([-r, 0], \mathbb{K}^n)$, the function $z \coloneqq z^L(\argdot; \phi)$ is a solution of an integral equation~\eqref{eq: integral eq with G^L(argdot; phi)}
\begin{equation*}
	z(t)
	= L\int_0^t z_s \mspace{2mu} \mathrm{d}s + G^L(t; \phi)
	\mspace{25mu}
	(t \ge 0)
\end{equation*}
under the initial condition $z_0 = 0$.
\end{lemma}

\subsection{Regularity of the general forcing term}\label{subsec: regularity of general forcing term}

To study Eq.~\eqref{eq: integral eq with G^L(argdot; phi)}, it is important to reveal the regularity of the function $G^L(\argdot; \phi)$ for each $\phi \in \mathcal{M}^1([-r, 0], \mathbb{K}^n)$.

\subsubsection{Forcing terms for continuous initial histories}

Before we tackle this problem, we find a differential equation satisfied by $z \coloneqq z^L(\argdot; \phi)$ for $\phi \in C([-r, 0], \mathbb{K}^n)$.
It should be noted that this is not straightforward because \eqref{eq: z^L(argdot; phi)} is only valid for $t \ge 0$.

Let $x \coloneqq x^L(\argdot; \phi, 0)$ and $\Tilde{x} \coloneqq x^L \Bigl( \argdot; \widehat{\phi(0)}, 0 \Bigr)$.
In view of
\begin{equation*}
	Lz_t
	= \int_{-r}^0 \mathrm{d}\eta(\theta) \mspace{2mu} z(t + \theta)
	= \int_{-t}^0 \mathrm{d}\eta(\theta) \mspace{2mu} z(t + \theta)
\end{equation*}
for each $t \ge 0$, we express the linear RFDE~\eqref{eq: linear RFDE} as
\begin{equation*}
	\Dot{x}(t)
	= \int_{-t}^0 \mathrm{d}\eta(\theta) \mspace{2mu} x(t + \theta)
	+ \int_{-r}^{-t} \mathrm{d}\eta(\theta) \mspace{2mu} \phi(t + \theta)
\end{equation*}
by using the additivity of Riemann-Stieltjes integrals on sub-intervals.
Here we are interpreting that the second term of the right-hand side is equal to $0$ when $t \ge r$.
More precisely, we introduce the following.

\begin{definition}[cf.\ \cite{Bernier--Manitius_1978}, \cite{Delfour--Manitius_1980}, \cite{Kappel_2006}]\label{dfn: g^L(argdot; phi)}
For each $\phi \in C([-r, 0], \mathbb{K}^n)$, we define a function $g^L(\argdot; \phi) \colon [0, \infty) \to \mathbb{K}^n$ by
\begin{equation*}
	g^L(t; \phi)
	\coloneqq \int_{-r}^{-t} \mathrm{d}\eta(\theta) \mspace{2mu} \phi(t + \theta)
\end{equation*}
for $t \in [0, r)$ and $g^L(t; \phi) = 0$ for $t \ge r$.
Here the right-hand side is a Riemann-Stieltjes integral.
\end{definition}

We note that similar concepts have appeared in the literature.
See \cite[(3.1) and (3.2)]{Bernier--Manitius_1978}, \cite[(2.7) and (2.13)]{Delfour--Manitius_1980}, and \cite[Lemma~1.10]{Kappel_2006}, for example.

From Theorem~\ref{thm: mild sol with instantaneous input}, $\Tilde{x}$ satisfies
\begin{equation*}
	\Dot{\Tilde{x}}(t)
	= \int_{-t}^0 \mathrm{d}\eta(\theta) \mspace{2mu} \Tilde{x}(t + \theta)
\end{equation*}
for almost all $t \ge 0$.
In combination with the above consideration, $z$ satisfies
\begin{align*}
	\Dot{z}(t)
	&= \Dot{x}(t) - \Dot{\Tilde{x}}(t) \\
	&= \int_{-t}^0 \mathrm{d}\eta(\theta) \mspace{2mu} z(t + \theta) + g^L(t; \phi)
\end{align*}
for almost all $t \ge 0$.
Here the property that $t + \theta \ge 0$ for all $\theta \in [-t, 0]$ is used.

In summary, we have the following statement.

\begin{lemma}\label{lem: z^L(argdot; phi)}
For each $\phi \in C([-r, 0], \mathbb{K}^n)$, $z \coloneqq z^L(\argdot; \phi)$ is locally absolutely continuous, differentiable almost everywhere, and
\begin{equation}\label{eq: non-homogeneous linear RFDE with g^L(argdot; phi)}
	\Dot{z}(t) = Lz_t + g^L(t; \phi)
\end{equation}
holds for almost all $t \ge 0$.
\end{lemma}

We note that since $g^L(\argdot; \phi)$ is not necessarily continuous, Theorem~\ref{thm: VOC formula, phi = 0, G = Vf} is not sufficient to obtain an expression of $z = z^L(\argdot; \phi)$ by $X^L$.

\subsubsection{Relationship with the forcing terms}

Comparing \eqref{eq: integral eq with G^L(argdot; phi)} and \eqref{eq: non-homogeneous linear RFDE with g^L(argdot; phi)}, it is natural to expect that
\begin{equation}\label{eq: G^L(argdot; phi) and g^L(argdot; phi)}
	G^L(t; \phi) = \int_0^t g^L(s; \phi) \mspace{2mu} \mathrm{d}s
\end{equation}
holds for all $t \ge 0$ when $\phi \in C([-r, 0], \mathbb{K}^n)$.
We now justify this relationship.

\begin{lemma}\label{lem: expression of g^L(argdot; phi)}
Suppose $\phi \in C([-r, 0], \mathbb{K}^n)$.
Then
\begin{equation}\label{eq: expression of g^L(argdot; phi)}
	g^L(t; \phi)
	= L\bar{\phi}_t - [\eta(0) - \eta(-t)]\phi(0)
\end{equation}
holds for all $t \ge 0$.
Consequently, $g^L(\argdot; \phi)$ is a locally Riemann integrable function vanishing at $[r, \infty)$.
\end{lemma}

\begin{proof}
When $t \ge r$,
\begin{equation*}
	L\Bar{\phi}_t
	= \int_{-r}^0 \mathrm{d}\eta(\theta) \mspace{2mu} \phi(0)
	= [\eta(0) - \eta(-r)]\phi(0)
\end{equation*}
holds.
Therefore, the right-hand side of \eqref{eq: expression of g^L(argdot; phi)} is equal to $0$ for all $t \ge r$.
We next consider the case $t \in [0, r)$.
In this case, we have
\begin{align*}
	g^L(t; \phi)
	&= \int_{-r}^{-t} \mathrm{d}\eta(\theta) \mspace{2mu} \Bar{\phi}(t + \theta) \\
	&= \int_{-r}^0 \mathrm{d}\eta(\theta) \mspace{2mu} \Bar{\phi}(t + \theta)
	- \int_{-t}^0 \mathrm{d}\eta(\theta) \mspace{2mu} \Bar{\phi}(t + \theta)
\end{align*}
by the additivity of Riemann-Stieltjes integrals on sub-intervals.
Since
\begin{equation*}
	\int_{-t}^0 \mathrm{d}\eta(\theta) \mspace{2mu} \Bar{\phi}(t + \theta)
	= \int_{-t}^0 \mathrm{d}\eta(\theta) \mspace{2mu} \phi(0)
	= [\eta(0) - \eta(-t)]\phi(0),
\end{equation*}
the expression~\eqref{eq: expression of g^L(argdot; phi)} is obtained.
Since $[0, \infty) \ni t \mapsto L\Bar{\phi}_t \in \mathbb{K}^n$ is continuous and $[0, \infty) \ni t \mapsto \eta(-t)\phi(0)$ is of locally bounded variation, the local Riemann integrability of $g^L(\argdot; \phi)$ follows by the expression~\eqref{eq: expression of g^L(argdot; phi)}.
\end{proof}

\begin{remark}
The expression~\eqref{eq: expression of g^L(argdot; phi)} also shows that $g^L(\argdot; \phi)$ is continuous if $\phi(0) = 0$.
This should be compared with \cite[Theorem~2.1(ii) and Remark~2.1]{Delfour--Manitius_1980}.
\end{remark}

The following theorem reveals a connection between $G^L(\argdot; \phi)$ and $g^L(\argdot; \phi)$.

\begin{theorem}\label{thm: expression of G^L(argdot; phi)}
Let $\phi \in \mathcal{M}^1([-r, 0], \mathbb{K}^n)$ be given.
Then for all $t \ge 0$,
\begin{equation}\label{eq: expression of G^L(argdot; phi)}
	G^L(t; \phi)
	= L\int_0^t \Bar{\phi}_s \mspace{2mu} \mathrm{d}s - \int_0^t [\eta(0) - \eta(-s)]\phi(0) \mspace{2mu} \mathrm{d}s
\end{equation}
holds.
\end{theorem}

\begin{proof}
For the first term of the definition of $G^L(t; \phi)$, we have
\begin{align*}
	\int_{-r}^0 \mathrm{d}\eta(\theta) \mspace{2mu}
		\biggl( \int_\theta^0 \phi(s) \mspace{2mu} \mathrm{d}s \biggr)
	&= \int_{-r}^0 \mathrm{d}\eta(\theta) \mspace{2mu}
		\biggl( \int_\theta^0 \Bar{\phi}(s) \mspace{2mu} \mathrm{d}s \biggr) \\
	&= \int_{-r}^0 \mathrm{d}\eta(\theta) \mspace{2mu}
		\biggl( \int_\theta^{t + \theta} \Bar{\phi}(s) \mspace{2mu} \mathrm{d}s \biggr)
	- \int_{-r}^0 \mathrm{d}\eta(\theta) \mspace{2mu}
		\biggl( \int_0^{t + \theta} \Bar{\phi}(s) \mspace{2mu} \mathrm{d}s \biggr),
\end{align*}
where the first term of the last equation is equal to $L\int_0^t \Bar{\phi}_s \mspace{2mu} \mathrm{d}s$.
The remainder of the proof is divided into the cases $t \in [0, r]$ and $t \in (r, \infty)$ in order to study the term $\int_{-r}^0 \mathrm{d}\eta(\theta) \mspace{2mu} \bigl( \int_0^{t + \theta} \Bar{\phi}(s) \mspace{2mu} \mathrm{d}s \bigr)$.

\vspace{0.5\baselineskip}
\noindent
\textbf{Case 1: $t \in [0, r]$.}
When $t = r$, we have
\begin{equation*}
	\int_{-r}^0 \mathrm{d}\eta(\theta) \mspace{2mu}
		\biggl( \int_0^{t + \theta} \Bar{\phi}(s) \mspace{2mu} \mathrm{d}s \biggr)
	= \int_{-r}^0 \mathrm{d}\eta(\theta) \mspace{2mu} (r + \theta)\phi(0)
\end{equation*}
because $r + \theta \ge 0$ for all $\theta \in [-r, 0]$.
We next consider the case $t \in [0, r)$.
In this case,
\begin{align*}
	&\int_{-r}^0 \mathrm{d}\eta(\theta) \mspace{2mu}
		\biggl( \int_0^{t + \theta} \Bar{\phi}(s) \mspace{2mu} \mathrm{d}s \biggr) \\
	&= \int_{-r}^{-t} \mathrm{d}\eta(\theta) \mspace{2mu}
		\biggl(\int_0^{t + \theta} \phi(s) \mspace{2mu} \mathrm{d}s \biggr)
	+ \int_{-t}^0 \mathrm{d}\eta(\theta) \mspace{2mu}
		\biggl( \int_0^{t + \theta} \Bar{\phi}(s) \mspace{2mu} \mathrm{d}s \biggr)
\end{align*}
holds by the additivity of Riemann-Stieltjes integrals on sub-intervals and by the property that $t + \theta \le 0$ for all $\theta \in [-r, -t]$.
Here the second term of the right-hand side is equal to
\begin{equation*}
	\int_{-t}^0 \mathrm{d}\eta(\theta) \mspace{2mu} (t + \theta)\phi(0).
\end{equation*}
Therefore, the definition of $G^L(t; \phi)$ yields
\begin{align*}
	G^L(t; \phi)
	&= L\int_0^t \Bar{\phi}_s \mspace{2mu} \mathrm{d}s
	- \int_{-t}^0 \mathrm{d}\eta(\theta) \mspace{2mu}
		\biggl( \int_0^{t + \theta} \Bar{\phi}(s) \mspace{2mu} \mathrm{d}s \biggr) \\
	&= L\int_0^t \Bar{\phi}_s \mspace{2mu} \mathrm{d}s
	- \int_{-t}^0 \mathrm{d}\eta(\theta) \mspace{2mu} (t + \theta) \phi(0)
\end{align*}
including the case $t = r$.
The proof is complete in view of
\begin{align*}
	\int_{-t}^0 (t + \theta) \mspace{2mu} \mathrm{d}\eta(\theta)
	&= [(t + \theta)\eta(\theta)]_{\theta = -t}^0 - \int_{-t}^0 \eta(\theta) \mspace{2mu} \mathrm{d}\theta \\
	&= t\eta(0) - \int_0^t \eta(-s) \mspace{2mu} \mathrm{d}s,
\end{align*}
where the integration by parts formula for Riemann-Stieltjes integrals is used.

\vspace{0.5\baselineskip}
\noindent
\textbf{Case 2: $t \in (r, \infty)$.}
Since we have shown that \eqref{eq: expression of G^L(argdot; phi)} holds for $t = r$,
\begin{equation*}
	G^L(t; \phi)
	= L\int_0^r \Bar{\phi}_s \mspace{2mu} \mathrm{d}s - \int_0^r [\eta(0) - \eta(-s)]\phi(0) \mspace{2mu} \mathrm{d}s
\end{equation*}
holds for all $t \ge r$.
Here the property that $G^L(\argdot; \phi)$ is constant on $[r, \infty)$ is used.
Then the proof is complete by showing that the right-hand side of \eqref{eq: expression of G^L(argdot; phi)} is constant on $[r, \infty)$.
For this purpose, we calculate
\begin{equation*}
	L\int_0^t \Bar{\phi}_s \mspace{2mu} \mathrm{d}s - L\int_0^r \Bar{\phi}_s \mspace{2mu} \mathrm{d}s.
\end{equation*}
By the linearity of $L$, it is calculated as
\begin{align*}
	\int_{-r}^0 \mathrm{d}\eta(\theta) \mspace{2mu}
		\biggl( \int_{r + \theta}^{t + \theta} \Bar{\phi}(s) \mspace{2mu} \mathrm{d}s \biggr)
	&= \int_{-r}^0 \mathrm{d}\eta(\theta) \mspace{2mu} (t - r)\phi(0) \\
	&= (t - r)[\eta(0) - \eta(-r)]\phi(0).
\end{align*}
Since $\eta$ is constant on $(-\infty, -r]$, the last value is expressed as
\begin{equation*}
	\int_r^t [\eta(0) - \eta(-s)]\phi(0) \mspace{2mu} \mathrm{d}s.
\end{equation*}
This shows that
\begin{equation*}
	L\int_0^t \Bar{\phi}_s \mspace{2mu} \mathrm{d}s
	= L\int_0^r \Bar{\phi}_s \mspace{2mu} \mathrm{d}s + \int_r^t [\eta(0) - \eta(-s)]\phi(0) \mspace{2mu} \mathrm{d}s,
\end{equation*}
which also implies that the right-hand side of \eqref{eq: expression of G^L(argdot; phi)} is equal to
\begin{equation*}
	L\int_0^r \Bar{\phi}_s \mspace{2mu} \mathrm{d}s - \int_0^r [\eta(0) - \eta(-s)]\phi(0) \mspace{2mu} \mathrm{d}s
\end{equation*}
for all $t \ge r$.
\end{proof}

\begin{remark}
$G^L(t; \phi)$ is also expressed as
\begin{equation*}
	G^L(t; \phi)
	= \int_{-r}^0 [\eta(\theta) - \eta(\theta - t)]\phi(\theta) \mspace{2mu} \mathrm{d}\theta.
\end{equation*}
See \cite[Section I.2 of Chapter I]{Diekmann--vanGils--Lunel--Walther_1995} for the detail.
See also \cite[Remark~2.10(iii) in Chapter~2]{VerduynLunel_1989}.
In this paper, we do not need the above expression.
\end{remark}

By combining the obtained results, we obtain the following result on the regularity of $G^L(\argdot; \phi)$.
See also \cite[Remark~2.10(ii) in Chapter~2]{VerduynLunel_1989}.

\begin{theorem}\label{thm: regularity of G^L(argdot; phi)}
For any $\phi \in \mathcal{M}^1([-r, 0], \mathbb{K}^n)$, the function $G^L(\argdot; \phi)$ is continuous with $G^L(0; \phi) = 0$.
Furthermore, if $\phi \in C([-r, 0], \mathbb{K}^n)$, then it is locally absolutely continuous, differentiable almost everywhere, and
\begin{equation*}
	\Dot{G}^L(t; \phi) = g^L(t; \phi)
\end{equation*}
holds for almost all $t \in [0, \infty)$.
Here $\dot{G}^L(t; \phi)$ denotes the derivative of $G^L(\argdot; \phi)$ at $t$.
\end{theorem}

\begin{proof}
Let $\phi \in \mathcal{M}^1([-r, 0], \mathbb{K}^n)$ be given.
Then \eqref{eq: expression of G^L(argdot; phi)} yields the continuity of $G^L(\argdot; \phi)$.
The property $G^L(0; \phi) = 0$ follows by definition.
We next assume $\phi \in C([-r, 0], \mathbb{K}^n)$.
Since $\Bar{\phi} \colon [-r, \infty) \to \mathbb{K}^n$ is continuous, Theorem~\ref{thm: expression of G^L(argdot; phi)}, Corollary~\ref{cor: T*int_0^t x_s ds for continuous x}, and Lemma~\ref{lem: expression of g^L(argdot; phi)} show that \eqref{eq: G^L(argdot; phi) and g^L(argdot; phi)}
\begin{equation*}
	G^L(t; \phi) = \bigl( Vg^L(\argdot; \phi) \bigr)(t)
\end{equation*}
holds for all $t \ge 0$.
This yields the properties of $G^L(\argdot; \phi)$.
\end{proof}

\subsection{Derivation of the variation of constants formula}\label{subsec: derivation of VOC formula}

\subsubsection{Formulas for trivial initial histories}

Since $X^L|_{[0, \infty)}$ is locally absolutely continuous (see Theorem~\ref{thm: DE of X^L}), by the integration by parts formula for matrix-valued absolutely continuous functions\footnote{It can be obtained by the corresponding result for scalar-valued functions in combination with the linearity of Lebesgue integration. We note that the result for scalar-valued functions is mentioned in \cite[Exercise 14 of Chapter 7]{Rudin_1987}. One can also give a direct proof based on the matrix product.},
\begin{align*}
	\int_0^t X^L(t - u)g(u) \mspace{2mu} \mathrm{d}u
	&= [X^L(t - u)(Vg)(u)]_{u = 0}^t + \int_0^t \Dot{X}^L(t - u)(Vg)(u) \mspace{2mu} \mathrm{d}u \\
	&= (Vg)(t) + \int_0^t \Dot{X}^L(t - u)(Vg)(u) \mspace{2mu} \mathrm{d}u
\end{align*}
holds for any $g \in \mathcal{L}^1_\mathrm{loc}([0, \infty), \mathbb{K}^n)$.
Here $X^L(0) = I$ and $(Vg)(0) = 0$ are also used.
The following theorem is motivated by this.

\begin{theorem}\label{thm: VOC formula, phi = 0}
Let $G \colon [0, \infty) \to \mathbb{K}^n$ be a continuous function with $G(0) = 0$.
Then \eqref{eq: VOC formula, phi = 0}
\begin{equation*}
	x^L(t; 0, G) = G(t) + \int_0^t \dot{X}^L(t - u)G(u) \mspace{2mu} \mathrm{d}u
\end{equation*}
holds for all $t \ge 0$.
\end{theorem}

\begin{proof}
Let $X \coloneqq X^L|_{[0, \infty)}$.
We define a function $x \colon [-r, \infty) \to \mathbb{K}^n$ by $x_0 = 0$ and
\begin{equation*}
	x(t) \coloneqq G(t) + \bigl( \Dot{X} * G \bigr)(t)
\end{equation*}
for $t \ge 0$.
By applying Corollary~\ref{cor: convolution of loc AC function and continuous function}, we have $Vx|_{[0, \infty)} = X * G$ in combination with the fundamental theorem of calculus.
Furthermore, we have
\begin{equation*}
	x(t) = G(t) + [\mathrm{d}\Check{\eta} * (X * G)](t)
	\mspace{25mu}
	(t \ge 0)
\end{equation*}
from Theorems~\ref{thm: DE of X^L} and \ref{thm: associativity for RS convolution}.
Therefore, $x$ satisfies
\begin{equation*}
	x(t) = G(t) + \bigl( \mathrm{d}\Check{\eta} * Vx|_{[0, \infty)} \bigr)(t)
\end{equation*}
for all $t \ge 0$.
This implies that \eqref{eq: VOC formula, phi = 0} holds by applying Theorems~\ref{thm: sol to integral eq with G and instantaneous input} and \ref{thm: unique existence, integral eq with G}.
\end{proof}

The following is a corollary of Theorem~\ref{thm: VOC formula, phi = 0}.
The proof can be omitted.
It is an extension of Theorem~\ref{thm: VOC formula, phi = 0, G = Vf}.

\begin{corollary}[cf.\ \cite{Hale_1971_Springer}, \cite{Hale_1977}]\label{cor: VOC formula, phi = 0, G = Vg}
Let $g \in \mathcal{L}^1_\mathrm{loc}([0, \infty), \mathbb{K}^n)$.
Then 
\begin{equation}\label{eq: VOC formula, phi = 0, G = Vg}
	x^L(t; 0, Vg) = \int_0^t X^L(t - u)g(u) \mspace{2mu} \mathrm{d}u
\end{equation}
holds for all $t \ge 0$.
\end{corollary}

\subsubsection{Formulas for homogeneous equations}

We next find an expression of $x^L(\argdot; \phi, 0)$ by $X^L$ as an application of Theorem~\ref{thm: VOC formula, phi = 0}.

\begin{theorem}\label{thm: VOC formula, G = 0}
Let $\phi \in \mathcal{M}^1([-r, 0], \mathbb{K}^n)$.
Then
\begin{equation}\label{eq: VOC formula, G = 0}
	x^L(t; \phi, 0)
	= X^L(t)\phi(0) + G^L(t; \phi) + \int_0^t \Dot{X}^L(t - u)G^L(u; \phi) \mspace{2mu} \mathrm{d}u
\end{equation}
holds for all $t \ge 0$.
\end{theorem}

\begin{proof}
From Lemma~\ref{lem: integral eq for z^L(argdot; phi)} and Theorem~\ref{thm: VOC formula, phi = 0} together with Theorem~\ref{thm: regularity of G^L(argdot; phi)},
\begin{equation*}
	z^L(t; \phi)
	= G^L(t; \phi) + \int_0^t \Dot{X}^L(t - u)G^L(u; \phi) \mspace{2mu} \mathrm{d}u
\end{equation*}
holds for all $t \ge 0$.
Then the formula~\eqref{eq: VOC formula, G = 0} is obtained in view of
\begin{equation*}
	z^L(t; \phi) = x^L(t; \phi, 0) - X^L(t)\phi(0)
\end{equation*}
for $t \ge 0$.
\end{proof}

\begin{remark}
The above proof of Theorem~\ref{thm: VOC formula, G = 0} is considered to be a reorganization of \cite[Section~I.2 of Chapter~I]{Diekmann--vanGils--Lunel--Walther_1995}.
It leads us to the understanding of the variation of constants formula for non-homogeneous linear RFDEs that does not rely on the theory of Volterra convolution integral equations.
\end{remark}

We have the following corollary.

\begin{corollary}[cf.\ \cite{Kappel_2006}]\label{cor: VOC formula, continuous phi, G = 0}
Let $\phi \in C([-r, 0], \mathbb{K}^n)$.
Then
\begin{equation}\label{eq: VOC formula, continuous phi, G = 0}
	x^L(t; \phi, 0)
	= X^L(t)\phi(0) + \int_0^t X^L(t - u)g^L(u; \phi) \mspace{2mu} \mathrm{d}u
\end{equation}
holds for all $t \ge 0$.
\end{corollary}

\begin{proof}
From Theorem~\ref{thm: regularity of G^L(argdot; phi)},
\begin{equation*}
	G^L(\argdot; \phi) = V \bigl( g^L(\argdot; \phi) \bigr)
\end{equation*}
holds.
Therefore, the formula~\eqref{eq: VOC formula, continuous phi, G = 0} is obtained from \eqref{eq: VOC formula, G = 0} by using the integration by parts formula for matrix-valued absolutely continuous functions.
\end{proof}

Corollary~\ref{cor: VOC formula, continuous phi, G = 0} should be compared with \cite[Theorem~1.11]{Kappel_2006}, where the inverse Laplace transform is used to obtain a formula.

\subsubsection{Derivation of the main result of this section}

Theorem~\ref{thm: VOC formula, integral eq} is a combination of Theorems~\ref{thm: VOC formula, phi = 0} and \ref{thm: VOC formula, G = 0} in view of Corollary~\ref{cor: decomposition of sol to integral eq with G}.
Therefore, the proof can be omitted.

The following is a corollary of Theorem~\ref{thm: VOC formula, integral eq}, which is a combination of Corollaries~\ref{cor: VOC formula, phi = 0, G = Vg} and \ref{cor: VOC formula, continuous phi, G = 0} in view of Corollary~\ref{cor: decomposition of sol to integral eq with G}.
The proof can be omitted.

\begin{corollary}\label{cor: VOC formula, continuous phi, G = Vg}
If $\phi \in C([-r, 0], \mathbb{K}^n)$ and $G = Vg$ for some $g \in \mathcal{L}^1_\mathrm{loc}([0, \infty), \mathbb{K}^n)$, then
\begin{equation*}
	x^L(t; \phi, G)
	= X^L(t)\phi(0) + \int_0^t X^L(t - u) [g^L(u; \phi) + g(u)] \mspace{2mu} \mathrm{d}u
\end{equation*}
holds for all $t \ge 0$.
\end{corollary}

\subsection{Variation of constants formula for linear differential difference equations}

We apply Theorem~\ref{thm: VOC formula, G = 0} to an autonomous linear differential difference equation~\eqref{eq: linear differential difference eq, multiple delays}
\begin{equation*}
	\Dot{x}(t) = Ax(t) + \sum_{k = 1}^m B_kx(t - \tau_k)
	\mspace{25mu}
	(t \ge 0)
\end{equation*}
for $n \times n$ matrices $A, B_1, \dots, B_m \in M_n(\mathbb{K})$ and $\tau_1, \dots, \tau_m \in (0, r]$.
We recall that the linear DDE~\eqref{eq: linear differential difference eq, multiple delays} can be expressed in the form of the linear RFDE~\eqref{eq: linear RFDE} by defining a continuous linear map $L \colon C([-r, 0], \mathbb{K}^n) \to \mathbb{K}^n$ by \eqref{eq: L, linear differential difference eq, multiple delays}
\begin{equation*}
	L\psi = A\psi(0) + \sum_{k = 1}^m B_k\psi(-\tau_k)
\end{equation*}
for $\psi \in C([-r, 0], \mathbb{K}^n)$.

For the above mentioned application, we need to calculate the function $G^L(\argdot; \phi)$ for each $\phi \in \mathcal{M}^1([-r, 0], \mathbb{K}^n)$ based on Definition~\ref{dfn: G^L(argdot; phi)}.
By the linearity of $L \mapsto G^L(\argdot; \phi)$, this can be reduced to the calculation of $G^{L_k}(\argdot; \phi)$ for each $k \in \{0, \dots, m\}$, where $L_k \colon C([-r, 0], \mathbb{K}^n) \to \mathbb{K}^n$ is the continuous linear map given by
\begin{align*}
	L_0\psi &\coloneqq A\psi(0), \\
\intertext{and}
	L_k\psi &\coloneqq B_k\psi(-\tau_k)
\end{align*}
for $k \in \{1, \dots, m\}$.
We have the following lemma.

\begin{lemma}\label{lem: G^L(argdot; phi) for multiple delays}
Let $\phi \in \mathcal{M}^1([-r, 0], \mathbb{K}^n)$ be given.
Then the following statements hold:
\begin{enumerate}
\item $G^{L_0}(\argdot; \phi) = 0$.
\item For each $k \in \{1, \dots, m\}$, $G^{L_k}(0; \phi) = 0$ and
\begin{equation*}
	G^{L_k}(t; \phi) =
	\begin{cases}
		B_k\int_{-\tau_k}^{t - \tau_k} \phi(s) \mspace{2mu} \mathrm{d}s & (t \in (0, \tau_k]), \\
		B_k\int_{-\tau_k}^0 \phi(s) \mspace{2mu} \mathrm{d}s & (t \in (\tau_k, \infty)).
	\end{cases}
\end{equation*}
holds.
\end{enumerate}
\end{lemma}

\begin{proof}
1. Let $\eta_0 \colon [-r, 0] \to M_n(\mathbb{K})$ be the matrix-valued function given by
\begin{equation*}
	\eta_0(\theta) \coloneqq
	\begin{cases}
		O & (-r \le \theta < 0), \\
		A & (\theta = 0).
	\end{cases}
\end{equation*}
Then $L_0$ is expressed as
\begin{equation*}
	L_0\psi = \int_{-r}^0 \mathrm{d}\eta_0(\theta) \mspace{2mu} \psi(\theta)
\end{equation*}
for $\psi \in C([-r, 0], \mathbb{K}^n)$.
Therefore, the definition of $G^L(\argdot; \phi)$ yields the conclusion.

2. Let $k \in \{1, \dots, m\}$ be fixed and $\eta_k \colon [-r, 0] \to M_n(\mathbb{K})$ be the matrix-valued function given by
\begin{equation*}
	\eta_k(\theta) \coloneqq
	\begin{cases}
		O & (-r \le \theta \le -\tau_k), \\
		B_k & (-\tau_k < \theta \le 0).
	\end{cases}
\end{equation*}
Then $L_k$ is expressed as
\begin{equation*}
	L_k\psi = \int_{-r}^0 \mathrm{d}\eta_k(\theta) \mspace{2mu} \psi(\theta)
\end{equation*}
for $\psi \in C([-r, 0], \mathbb{K}^n)$.
By the definition of $L_k$, we have
\begin{equation*}
	\int_{-r}^0 \mathrm{d}\eta_k(\theta) \mspace{2mu}
		\biggl( \int_\theta^0 \phi(s) \mspace{2mu} \mathrm{d}s \biggr)
	= B_k\int_{-\tau_k}^0 \phi(s) \mspace{2mu} \mathrm{d}s.
\end{equation*}
Furthermore, the integral $\int_{-r}^{-t} \mathrm{d}\eta_k(\theta) \mspace{2mu} \bigl( \int_0^{t + \theta} \phi(s) \mspace{2mu} \mathrm{d}s \bigr)$ is calculated as
\begin{equation*}
	\int_{-r}^{-t} \mathrm{d}\eta_k(\theta) \mspace{2mu}
		\biggl(\int_0^{t + \theta} \phi(s) \mspace{2mu} \mathrm{d}s \biggr)
	=
	\begin{cases}
		B_k\int_0^{t - \tau_k} \phi(s) \mspace{2mu} \mathrm{d}s & (t \in [0, \tau_k]), \\
		0 & (t \in (\tau_k, \infty)).
	\end{cases}
\end{equation*}
By combining the above expressions, the conclusion is obtained.
\end{proof}

\begin{theorem}[cf.\ \cite{Hale_1977}, \cite{Hale--VerduynLunel_1993}]
Let $L \colon C([-r, 0], \mathbb{K}^n) \to \mathbb{K}^n$ be the continuous linear map given by \eqref{eq: L, linear differential difference eq, multiple delays}.
Then for any $\phi \in \mathcal{M}^1([-r, 0], \mathbb{K}^n)$,
\begin{equation}\label{eq: VOC formula II, differential difference eq}
	x^L(t; \phi, 0)
	= X^L(t)\phi(0) + \sum_{k = 1}^m \int_{-\tau_k}^0 X^L(t - \tau_k - \theta)B_k\phi(\theta) \mspace{2mu} \mathrm{d}\theta
\end{equation}
holds for all $t \ge 0$.
\end{theorem}

\begin{proof}
Let $\phi \in \mathcal{M}^1([-r, 0], \mathbb{K}^n)$ be given.
From Lemma~\ref{lem: G^L(argdot; phi) for multiple delays},
\begin{equation*}
	G^L(\argdot; \phi)
	= \sum_{k = 1}^m G^{L_k}(\argdot; \phi)
\end{equation*}
is locally absolutely continuous.
Therefore, Theorem~\ref{thm: VOC formula, G = 0} and the integration by parts formula for absolutely continuous functions yield that
\begin{equation*}
	x^L(t; \phi, 0)
	= X^L(t)\phi(0) + \sum_{k = 1}^m \int_0^t X^L(t - u)\Dot{G}^{L_k}(u; \phi) \mspace{2mu} \mathrm{d}u
\end{equation*}
holds for all $t \ge 0$.
We now fix $k \in \{1, \dots, m\}$ and find an expression of the integral
\begin{equation*}
	\int_0^t X^L(t - u)\Dot{G}^{L_k}(u; \phi) \mspace{2mu} \mathrm{d}u.
\end{equation*}
Lemma~\ref{lem: G^L(argdot; phi) for multiple delays} shows that $\Dot{G}^{L_k}(t; \phi) = B_k\phi(t - \tau_k)$ holds for almost all $t \in [0, \tau_k]$, and $\Dot{G}^{L_k}(t; \phi) = 0$ holds for all $t \in (\tau_k, \infty)$.
Then the integral is expressed as follows:
\begin{itemize}
\item When $t \in [0, \tau_k]$, the integral becomes
\begin{align*}
	\int_0^t X^L(t - u)B_k\phi(u - \tau_k) \mspace{2mu} \mathrm{d}u
	&= \int_{-\tau_k}^{t - \tau_k} X^L(t - \tau_k - \theta)B_k\phi(\theta) \mspace{2mu} \mathrm{d}\theta \\
	&= \int_{-\tau_k}^0 X^L(t - \tau_k - \theta)B_k\phi(\theta) \mspace{2mu} \mathrm{d}\theta
\end{align*}
because $t - \tau_k \le 0$.
\item When $t \in (\tau_k, \infty)$, the integral becomes
\begin{equation*}
	\int_0^{\tau_k} X^L(t - u)B_k\phi(u - \tau_k) \mspace{2mu} \mathrm{d}u
	= \int_{-\tau_k}^0 X^L(t - \tau_k - \theta)B_k\phi(\theta) \mspace{2mu} \mathrm{d}\theta.
\end{equation*}
\end{itemize}
This completes the proof.
\end{proof}

\begin{remark}
Suppose $\phi \in C([-r, 0], \mathbb{R})$ and $m = 1$.
In \cite[Theorem~6.1 in Section~1.6]{Hale_1977} and \cite[Theorem~6.1 in Section~1.6]{Hale--VerduynLunel_1993}, \eqref{eq: VOC formula II, differential difference eq} is obtained by using the Laplace transform.
See also \cite[Theorem~4.2]{Nakagiri_1981}, where \eqref{eq: VOC formula II, differential difference eq} is obtained under different assumptions for a linear evolution equation with commensurate delays.
\end{remark}

\subsection{Remarks on definitions of ``fundamental matrix''}

\subsubsection{Definition by Hale}

Let $\mathbb{K} = \mathbb{R}$.
In \cite[Theorem~16.3 and Corollary~16.1]{Hale_1971_Springer} and \cite[Theorem~2.1 and Corollary~2.1 in Chapter~6]{Hale_1977}, a matrix-valued function $X \colon [0, \infty) \to M_n(\mathbb{R})$ is defined by using the property that for every $t \ge 0$,
\begin{equation*}
	\mathcal{L}^1_\mathrm{loc}([0, \infty), \mathbb{R}^n) \ni g \mapsto x^L(t; 0, Vg) \in \mathbb{R}^n
\end{equation*}
is a bounded linear operator to show
\begin{equation*}
	x^L(t; 0, Vg) = \int_0^t X(t - u)g(u) \mspace{2mu} \mathrm{d}u
	\mspace{25mu}
	(t \ge 0).
\end{equation*}
Furthermore, by the formal exchange of order of integration, the function $X$ is interpreted as a ``matrix-valued solution'' to the linear RFDE~\eqref{eq: linear RFDE}.
Indeed, Hale argued that $X$ satisfies (i) $X_0 = \Hat{I}$, (ii) $X|_{[0, \infty)}$ is locally absolutely continuous, and (iii) $X$ satisfies
\begin{equation*}
	\Dot{X}(t)
	= \int_{-r}^0 \mathrm{d}\eta(\theta) \mspace{2mu} X(t + \theta)
\end{equation*}
for almost all $t \in [0, \infty)$.
Here $\Hat{I} \colon [-r, 0] \to M_n(\mathbb{R})$ is defined by \eqref{eq: Hat{I}}
\begin{equation*}
	\Hat{I}(\theta) \coloneqq
	\begin{cases}
		O & (\theta \in [-r, 0)), \\
		I & (\theta = 0).
	\end{cases}
\end{equation*}
However, the above integral does not make sense in general because $X$ is not continuous.

\subsubsection{Volterra convolution integral equations and fundamental matrix solutions}

Let $x \coloneqq x^L \bigl( \argdot; \Hat{\xi} \bigr)|_{[0, \infty)}$ for some $\xi \in \mathbb{K}^n$ and suppose $\eta(0) = O$.
By using the integration by parts formula for Riemann-Stieltjes integrals and Theorem~\ref{thm: RS integral wrt C^1-function} in \eqref{eq: mild sol for instantaneous input, RS convolution}
\begin{equation*}
	x = \xi + \mathrm{d}\Check{\eta} * Vx,
\end{equation*}
we have
\begin{align*}
	x(t)
	&= \xi + [\Check{\eta}(u)(Vx)(t - u)]_{u = 0}^t + \int_0^t \Check{\eta}(u)x(t - u) \mspace{2mu} \mathrm{d}u \\
	&= \xi + (\Check{\eta} * x)(t)
\end{align*}
for all $t \ge 0$.
Here $(Vx)(0) = 0$ is also used.
The above calculation shows that the function $x \colon [0, \infty) \to \mathbb{K}^n$ satisfies
\begin{equation*}
	x = \Check{\eta} * x + \xi,
\end{equation*}
which is a Volterra convolution integral equation with the kernel function $\Check{\eta}$ and with the constant forcing term $\xi$.
Therefore, $X \coloneqq X^L|_{[0, \infty)}$ satisfies
\begin{equation*}
	X = \Check{\eta} * X + I.
\end{equation*}
This means that the restriction $X = X^L|_{[0, \infty)}$ is the \textit{fundamental matrix solution} for the Volterra convolution integral equation with the kernel function $\Check{\eta}$ under the assumption that $\eta(0) = O$.
For an approach by the Volterra convolution integral equation, see \cite[Section I.2 of Chapter I]{Diekmann--vanGils--Lunel--Walther_1995}.

\section{Exponential stability of principal fundamental matrix solution}\label{sec: exponential stability}

For a continuous linear map $L \colon C([-r, 0], \mathbb{K}^n) \to \mathbb{K}^n$, we consider a linear RFDE~\eqref{eq: linear RFDE}
\begin{equation*}
	\Dot{x}(t) = Lx_t
	\mspace{25mu}
	(t \ge 0).
\end{equation*}
Let $X^L \colon [-r, \infty) \to M_n(\mathbb{K})$ be the principal fundamental matrix solution.
We use the following terminology.

\begin{definition}\label{dfn: exponential stability of X^L}
We say that the principal fundamental matrix solution $X^L$ is \textit{exponentially stable} if there exist constants $M \ge 1$ and $\alpha > 0$ such that
\begin{equation}\label{eq: exponential stability of X^L}
	\abs{X^L(t)} \le M\mathrm{e}^{-\alpha t}
\end{equation}
holds for all $t \ge 0$.
We also say that $X^L$ is \textit{$\alpha$-exponentially stable}.
\end{definition}

In the following calculations, it is useful to extend the domain of definition of $X^L$ to $\mathbb{R}$ by letting $X^L(t) \coloneqq O$ for $t \in (-\infty, -r)$.

\begin{lemma}\label{lem: estimate of X^L(t + theta)}
If $X^L$ is $\alpha$-exponentially stable for some $\alpha > 0$, then there exists a constant $M \ge 1$ such that
\begin{equation*}
	\sup_{\theta \in [-r, 0]} \abs{X^L(t + \theta)}
	\le M\mathrm{e}^{-\alpha t}
\end{equation*}
holds for all $t \in \mathbb{R}$.
\end{lemma}

\begin{proof}
By the assumption, one can choose a constant $M_0 \ge 1$ so that
\begin{equation*}
	\abs{X^L(t)} \le M_0\mathrm{e}^{-\alpha t}
\end{equation*}
holds for all $t \ge 0$.
Since the statement is trivial when $t \le 0$, we only have to consider the case $t > 0$.
Let $\theta \in [-r, 0]$.
When $t + \theta \ge 0$, we have
\begin{equation*}
	\abs{X^L(t + \theta)}
	\le M_0\mathrm{e}^{-\alpha(t + \theta)}
	\le M_0\mathrm{e}^{\alpha r}\mathrm{e}^{-\alpha t}.
\end{equation*}
The above estimate also holds when $t + \theta < 0$ because $X^L(t + \theta) = O$ in this case.
Therefore, the conclusion is obtained.
\end{proof}

\begin{theorem}[cf.\ \cite{Hale_1977}, \cite{Hale--VerduynLunel_1993}]\label{thm: exponential stability of T^L}
If $X^L$ is $\alpha$-exponentially stable for some $\alpha > 0$, then the $C_0$-semigroup $\left( T^L(t) \right)_{t \ge 0}$ on $C([-r, 0], \mathbb{K}^n)$ defined by \eqref{eq: C_0-semigroup}
\begin{equation*}
	T^L(t)\phi \coloneqq x^L(\argdot; \phi, 0)_t
\end{equation*}
for $(t, \phi) \in [0, \infty) \times C([-r, 0], \mathbb{K}^n)$ is uniformly $\alpha$-exponentially stable, i.e., there exists a constant $M \ge 1$ such that for all $t \ge 0$,
\begin{equation*}
	\norm{T^L(t)} \le M\mathrm{e}^{-\alpha t}
\end{equation*}
holds.
\end{theorem}

\begin{proof}
By applying Lemma~\ref{lem: estimate of X^L(t + theta)}, we choose a constant $M_0 \ge 1$ so that
\begin{equation*}
	\sup_{\theta \in [-r, 0]} \abs{X^L(t + \theta)} \le M_0\mathrm{e}^{-\alpha t}
\end{equation*}
holds for all $t \in \mathbb{R}$.
Since the statement is trivial when $t = 0$, we only have to consider the case $t > 0$.
Let $\theta \in [-r, 0]$ and $\phi \in C([-r, 0], \mathbb{K}^n)$ be given.
Then
\begin{equation*}
	{\left[ T^L(t)\phi \right]}(\theta) =
	\begin{cases}
		X^L(t + \theta)\phi(0)
		+ \int_0^{t + \theta} X^L(t + \theta - u)g^L(u; \phi) \mspace{2mu} \mathrm{d}u & (t + \theta \ge 0), \\
		\phi(t + \theta) & (t + \theta \in [-r, 0])
	\end{cases}
\end{equation*}
holds from Corollary~\ref{cor: VOC formula, continuous phi, G = 0} (see Definition~\ref{dfn: g^L(argdot; phi)} for the definition of $g^L(t; \phi)$).
We divide the consideration into the following cases.

\vspace{0.5\baselineskip}
\noindent
\textbf{Case 1: $t + \theta \ge 0$.}
For the first term of the right-hand side,
\begin{equation*}
	\abs{X^L(t + \theta)\phi(0)}
	\le M_0\mathrm{e}^{-\alpha t}\abs{\phi(0)}
	\le M_0\mathrm{e}^{-\alpha t}\norm{\phi}
\end{equation*}
holds.
For the second term,
\begin{align*}
	\abs{\int_0^{t + \theta} X^L(t + \theta - u)g^L(u; \phi) \mspace{2mu} \mathrm{d}u}
	&\le \int_0^{t + \theta} \abs{X^L(t - u + \theta)}\abs{g^L(u; \phi)} \mspace{2mu} \mathrm{d}u \\
	&\le \int_0^{t + \theta} M_0\mathrm{e}^{-\alpha(t - u)} \abs{g^L(u; \phi)} \mspace{2mu} \mathrm{d}u
\end{align*}
holds from Lemma~\ref{lem: estimate of X^L(t + theta)}.
Since
\begin{equation*}
	\abs{g^L(t; \phi)}
	= \abs{\int_{-r}^{-t} \mathrm{d}\eta(\theta) \mspace{2mu} \phi(t + \theta)}
	\le \Var(\eta)\norm{\phi}
\end{equation*}
holds for all $t \in [0, r)$ (see Lemma~\ref{eq: estimate of norm of RS integral}) and $g^L(t; \phi) = 0$ for all $t \ge r$, we have
\begin{align*}
	\int_0^{t + \theta} M_0\mathrm{e}^{-\alpha(t - u)} \abs{g^L(u; \phi)} \mspace{2mu} \mathrm{d}u
	&\le \int_0^r M_0\mathrm{e}^{-\alpha(t - u)} \Var(\eta)\norm{\phi} \mspace{2mu} \mathrm{d}u \\
	&= M_0 {\biggl( \int_0^r \mathrm{e}^{\alpha u} \mspace{2mu} \mathrm{d}u \biggr)} \Var(\eta) \mathrm{e}^{-\alpha t}\norm{\phi}.
\end{align*}
We note that
\begin{equation*}
	\int_0^r \mathrm{e}^{\alpha u} \mspace{2mu} \mathrm{d}u
	= \frac{1}{\alpha}(\mathrm{e}^{\alpha r} - 1)
\end{equation*}
holds.

\vspace{0.5\baselineskip}
\noindent
\textbf{Case 2: $t + \theta < 0$.}
In this case, we have
\begin{equation*}
	\abs{\phi(t + \theta)}
	\le \mathrm{e}^{-\alpha(t + \theta)}\abs{\phi(t + \theta)}
	\le \mathrm{e}^{\alpha r}\mathrm{e}^{-\alpha t}\norm{\phi}.
\end{equation*}

\vspace{0.5\baselineskip}
By combining the estimates obtained in Cases 1 and 2,
\begin{equation*}
	\sup_{\theta \in [-r, 0]} \abs{{\left[ T^L(t)\phi \right]}(\theta)}
	\le M\mathrm{e}^{-\alpha t}\norm{\phi}
\end{equation*}
holds for some $M \ge 1$.
Therefore, the conclusion is obtained.
\end{proof}

The converse of Theorem~\ref{thm: exponential stability of T^L} also holds.

\begin{theorem}[cf.\ \cite{Hale_1977}, \cite{Hale--VerduynLunel_1993}]\label{thm: exponential stability of X^L}
If $\left( T^L(t) \right)_{t \ge 0}$ is uniformly $\alpha$-exponentially stable for some $\alpha > 0$, then $X^L$ is $\alpha$-exponentially stable.
\end{theorem}

\begin{proof}
By the assumption, we choose a constant $M_0 \ge 1$ so that
\begin{equation*}
	\norm{T^L(t)} \le M_0\mathrm{e}^{-\alpha t}
\end{equation*}
holds for all $t \ge 0$.
We fix $\xi \in \mathbb{K}^n$ and let
\begin{equation*}
	\phi_\xi
	\coloneqq x^L \bigl( \argdot; \Hat{\xi} \bigr)_r
	\in C([-r, 0], \mathbb{K}^n).
\end{equation*}
Then the map $\mathbb{K}^n \ni \xi \mapsto \phi_\xi \in C([-r, 0], \mathbb{K}^n)$ is linear from Corollary~\ref{eq: linearity of mild sol wrt initial history}.
Since $X^L(\argdot)\xi = x^L \bigl( \argdot; \Hat{\xi} \bigr)$, we have
\begin{equation*}
	\norm{\phi_\xi}
	= \sup_{t \in [0, r]} \abs{x^L \bigl( t; \Hat{\xi} \bigr)}
	\le \Bigl( \textstyle\sup_{t \in [0, r]} \abs{X^L(t)} \Bigr) \cdot \abs{\xi}.
\end{equation*}
This yields that the linear operator $\mathbb{K}^n \ni \xi \mapsto \phi_\xi \in C([-r, 0], \mathbb{K}^n)$ is bounded.

We now show that $X^L$ is $\alpha$-exponentially stable by dividing the following cases.

\vspace{0.5\baselineskip}
\noindent
\textbf{Case 1: $t \ge r$.}
From Theorem~\ref{thm: unique existence of mild sol}, we have
\begin{equation*}
	x^L \bigl( t; \Hat{\xi} \bigr) = x^L(t - r; \phi_\xi),
\end{equation*}
where the right-hand side is equal to $[T^L(t - r)\phi_\xi](0)$.
Therefore,
\begin{equation*}
	\abs{X^L(t)\xi}
	\le \norm{T^L(t - r)}\norm{\phi_\xi}
\end{equation*}
holds.
Since $\norm{T^L(t - r)} \le M_0\mathrm{e}^{\alpha r}\mathrm{e}^{-\alpha t}$, we obtain
\begin{equation*}
	\abs{X^L(t)}
	\le \Bigl( M_0\mathrm{e}^{\alpha r} \textstyle\sup_{t \in [0, r]} \abs{X^L(t)} \Bigr) \cdot \mathrm{e}^{-\alpha t}
\end{equation*}
by combining the above estimate on $\norm{\phi_\xi}$.

\vspace{0.5\baselineskip}
\noindent
\textbf{Case 2: $t \in [0, r]$.}
In this case, $\abs{X^L(t)}$ is estimated by
\begin{equation*}
	\abs{X^L(t)}
	\le \Bigl( \mathrm{e}^{\alpha r} \textstyle\sup_{t \in [0, r]} \abs{X^L(t)} \Bigr) \cdot \mathrm{e}^{-\alpha t}.
\end{equation*}
Here $1 = \mathrm{e}^{-\alpha t} \mathrm{e}^{\alpha t}$ is used.

\vspace{0.5\baselineskip}
By combining the above estimates, the conclusion is obtained.
\end{proof}

See \cite[Lemmas~6.1, 6.2, and 6.3 in Chapter~6]{Hale_1977} and \cite[Lemmas~5.1, 5.2, and 5.3 in Chapter~6]{Hale--VerduynLunel_1993} for related results.
We note that the statements of Theorems~\ref{thm: exponential stability of T^L} and \ref{thm: exponential stability of X^L} are included in these results, where the detailed proofs are not given.

\section{Principle of linearized stability and Poincar\'{e}-Lyapunov theorem}\label{sec: linearized stability}

Throughout this section, let $\mathbb{K} = \mathbb{R}$ and $h \colon \mathbb{R} \times C([-r, 0], \mathbb{R}^n) \supset \mathbb{R} \times U_0 \to \mathbb{R}^n$ be a continuous map for some open neighborhood $U_0$ of $0$ in $C([-r, 0], \mathbb{R}^n)$.

In this section, we consider a non-autonomous RFDE
\begin{equation}\label{eq: RFDE, L + h}
	\Dot{x}(t) = Lx_t + h(t, x_t).
\end{equation}
For the continuous map $h$, we assume that $h(t, \phi) = o(\norm{\phi})$ holds as $\norm{\phi} \to 0$ uniformly in $t$.
This means that for every $\ep > 0$, there exists a $\delta > 0$ such that for all $(t, \phi) \in \mathbb{R} \times U_0$, $\norm{\phi} < \delta$ implies
\begin{equation*}
	\abs{h(t, \phi)} \le \ep\norm{\phi}.
\end{equation*}
It follows that $h(t, 0) = 0$ for all $t \in \mathbb{R}$, and the RFDE~\eqref{eq: RFDE, L + h} has the zero solution.
\eqref{eq: RFDE, L + h} is considered as a perturbation of the linear RFDE~\eqref{eq: linear RFDE}
\begin{equation*}
	\Dot{x}(t) = Lx_t
	\mspace{25mu}
	(t \ge 0).
\end{equation*}
Let $X^L \colon [-r, \infty) \to M_n(\mathbb{R})$ be the principal fundamental matrix solution and $\bigl( T^L(t) \bigr)_{t \ge 0}$ be the $C_0$-semigroup on $C([-r, 0], \mathbb{R}^n)$ generated by \eqref{eq: linear RFDE}.

We also consider a non-autonomous RFDE
\begin{equation}\label{eq: RFDE, L + N + h}
	\Dot{x}(t) = Lx_t + N(t, x_t) + h(t, x_t)
\end{equation}
for a map $N \colon \mathbb{R} \times C([-r, 0], \mathbb{R}^n) \to \mathbb{R}^n$ with the following properties:
\begin{itemize}
\item For each $t \in \mathbb{R}$, the map $N(t) \colon C([-r, 0], \mathbb{R}^n) \to \mathbb{R}^n$ defined by
\begin{equation*}
	N(t)\phi \coloneqq N(t, \phi)
\end{equation*}
for $\phi \in C([-r, 0], \mathbb{R}^n)$ is a bounded linear operator.
\item $\mathbb{R} \ni t \mapsto N(t) \in \mathcal{B}(C([-r, 0], \mathbb{R}^n), \mathbb{R}^n)$ is continuous.
\item $\lim_{t \to \infty} \norm{N(t)} = 0$ holds.
\end{itemize}
See also \cite[Subsection~6.6.3]{Hale--VerduynLunel_1993} for a related discussion.

\begin{lemma}
The map $N$ is continuous.
\end{lemma}

\begin{proof}
For each fixed $(t^0, \phi^0) \in \mathbb{R} \times C([-r, 0], \mathbb{R}^n)$, we have
\begin{align*}
	\abs{N(t, \phi) - N(t^0, \phi^0)}
	&\le \abs{N(t, \phi) - N(t, \phi^0)} + \abs{N(t, \phi^0) - N(t^0, \phi^0)} \\
	&\le \norm{N(t)}\norm{\phi - \phi^0} + \norm{N(t) - N(t^0)}\norm{\phi^0}
\end{align*}
for all $(t, \phi) \in \mathbb{R} \times C([-r, 0], \mathbb{R}^n)$.
This yields the conclusion.
\end{proof}

\subsection{Variation of constants formula and non-linear equations}

In this subsection, we consider a non-autonomous RFDE
\begin{equation}\label{eq: RFDE, L + f}
	\Dot{x}(t) = Lx_t + f(t, x_t)
\end{equation}
for some continuous map
\begin{equation*}
	f \colon \mathbb{R} \times C([-r, 0], \mathbb{R}^n) \supset \dom(f) \to \mathbb{R}^n.
\end{equation*}
For each $\phi \in C([-r, 0], \mathbb{R}^n)$ and each $T > 0$, a continuous function $x \colon [t_0 - r, t_0 + T] \to \mathbb{R}^n$ is called a \textit{solution} of the RFDE~\eqref{eq: RFDE, L + f} under an initial condition $x_{t_0} = \phi$ if the following conditions are satisfied:
(i) $x_{t_0} = \phi$, (ii) $(t, x_t) \in \dom(f)$ for all $t \in [t_0, t_0 + T]$, and (iii) $x|_{[t_0, t_0 + T]}$ is differentiable and satisfies the RFDE~\eqref{eq: RFDE, L + f} for all $t \in [t_0, t_0 + T]$.
Here the derivative of $x$ at $t_0$ and $t_0 + T$ are understood as the right-hand derivative at $t_0$ and the left-hand derivative at $t_0 + T$, respectively.

\begin{theorem}\label{thm: VOC formula for nonlinear eq}
Let $\phi \in C([-r, 0], \mathbb{R}^n)$, $t_0 \in \mathbb{R}$, and $T > 0$ be given.
Then for a continuous function $x \colon [t_0 - r, t_0 + T] \to \mathbb{R}^n$ satisfying (i) $x_{t_0} = \phi$ and (ii) $(t, x_t) \in \dom(f)$ for all $t \in [t_0, t_0 + T]$, $x$ is a solution of the RFDE~\eqref{eq: RFDE, L + f} under the initial condition $x_{t_0} = \phi$ if and only if $x$ satisfies
\begin{equation*}
	x(t)
	= x^L(t - t_0; \phi, 0)
	+ \int_{t_0}^t X^L(t - u)f(u, x_u) \mspace{2mu} \mathrm{d}u
\end{equation*}
for all $t \in [t_0, t_0 + T]$.
\end{theorem}

We note that the above statement is not a simple application of Corollaries~\ref{cor: decomposition of sol to integral eq with G} and \ref{cor: VOC formula, phi = 0, G = Vg} because there is no method of variation of constants for RFDEs (see Subsection~\ref{subsec: motivation of convolution}).

\begin{proof}[Proof of Theorem~\ref{thm: VOC formula for nonlinear eq}]
Let $x \colon [t_0 - r, t_0 + T] \to \mathbb{R}^n$ be a continuous function satisfying the conditions (i) and (ii) in Theorem~\ref{thm: VOC formula for nonlinear eq}.
Then it is a solution of the RFDE~\eqref{eq: RFDE, L + f} under the initial condition $x_{t_0} = \phi$ if and only if
\begin{equation*}
	x(t) = \phi(0) + \int_{t_0}^t [Lx_s + f(s, x_s)] \mspace{2mu} \mathrm{d}s
\end{equation*}
holds for all $t \in [t_0, t_0 + T]$.
Let $z \colon [-r, T] \to \mathbb{R}^n$ be the function defined by $z(s) \coloneqq x(t_0 + s) - x^L(s; \phi)$ for $s \in [-r, T]$.
Then $z$ satisfies $z_0 = 0$ and an integral equation
\begin{equation*}
	z(s)
	= \int_0^s Lz_u \mspace{2mu} \mathrm{d}u
	+ \int_0^s f(t_0 + u, x_{t_0 + u}) \mspace{2mu} \mathrm{d}u
\end{equation*}
for $s \in [0, T]$.
Since $[0, T] \ni u \mapsto f(t_0 + u, x_{t_0 + u}) \in \mathbb{R}^n$ is continuous, $z|_{[0, T]}$ is expressed by
\begin{equation*}
	z(s) = \int_0^s X^L(s - u)f(t_0 + u, x_{t_0 + u}) \mspace{2mu} \mathrm{d}u
	\mspace{20mu}
	(s \in [0, T])
\end{equation*}
from Theorem~\ref{thm: VOC formula, phi = 0, G = Vf} or Corollary~\ref{cor: VOC formula, phi = 0, G = Vg}.
Therefore, the expression of $x$ is obtained by coming back to the condition on $x$.
\end{proof}

\subsection{Stability part of principle of linearized stability}

The statement in the following theorem is the stability part of the \textit{principle of linearized stability} for RFDEs.

\begin{theorem}[cf.~\cite{Diekmann--vanGils--Lunel--Walther_1995}]\label{thm: principle of linearized stability}
If $X^L$ is exponentially stable, then there exist $M \ge 1$, $\beta > 0$, and a neighborhood $U$ of $0$ in $C([-r, 0], \mathbb{R}^n)$ such that for every $t_0 \in \mathbb{R}$, every $\phi \in U$, and every non-continuable solution $x$ of the RFDE~\eqref{eq: RFDE, L + h} under the initial condition $x_{t_0} = \phi$, $x$ is defined for all $t \ge t_0$ and satisfies
\begin{equation*}
	\norm{x_t} \le M\mathrm{e}^{-\beta(t - t_0)}\norm{\phi}
\end{equation*}
for all $t \ge t_0$.
\end{theorem}

\begin{remark}
See \cite[Chapter~11]{Bellman--Cooke_1963} for the corresponding result for differential difference equations.
See \cite{Desch--Schappacher_1986} for the general result of the principle of linearized stability in the context of nonlinear semigroups.
See also \cite[Chapter~VII]{Diekmann--vanGils--Lunel--Walther_1995} for a general treatment of the principle of linearized stability and its application to RFDEs under the local Lipschitz continuity of $h$.
\end{remark}

In the proof of Theorem~\ref{thm: principle of linearized stability} given below, the Peano existence theorem and the continuation of solutions for RFDEs play key roles.
See \cite[Chapter 2]{Hale_1977} and \cite[Chapter 2]{Hale--VerduynLunel_1993} for the fundamental theory of RFDEs.

\begin{proof}[Proof of Theorem~\ref{thm: principle of linearized stability}]
We divide the proof into the following steps.

\vspace{0.5\baselineskip}
\noindent
\textbf{Step 1: Choice of a neighborhood of $0$ and a non-continuable solution.}
From Lemma~\ref{lem: estimate of X^L(t + theta)} and Theorem~\ref{thm: exponential stability of T^L}, we choose constants $M \ge 1$ and $\alpha > 0$ so that
\begin{equation*}
	\sup_{\theta \in [-r, 0]} \abs{X^L(t + \theta)}
	\le M\mathrm{e}^{-\alpha t}
	\mspace{25mu}
	(t \in \mathbb{R})
\end{equation*}
and
\begin{equation*}
	\norm{T^L(t)} \le M\mathrm{e}^{-\alpha t}
	\mspace{25mu}
	(t \ge 0)
\end{equation*}
hold.
We also choose an $\ep > 0$ so that
\begin{equation*}
	-\beta
	\coloneqq M\ep - \alpha
	< 0.
\end{equation*}
Since $h(t, \phi) = o(\norm{\phi})$ as $\norm{\phi} \to 0$ uniformly in $t$, there exists a $\widetilde{\delta} > 0$ for this $\ep > 0$ with the following properties:
\begin{enumerate}[label=(\roman*)]
\item For all $\phi \in C([-r, 0], \mathbb{R}^n)$, $\norm{\phi} < \widetilde{\delta}$ implies $\phi \in U_0$.
\item $\norm{\phi} < \widetilde{\delta}$ implies $\abs{h(t, \phi)} \le \ep\norm{\phi}$ for all $t \in \mathbb{R}$.
\end{enumerate}
Let $\delta \coloneqq \widetilde{\delta}/M$.
We define open sets $U$ and $\widetilde{U}$ by
\begin{align*}
	U
	&\coloneqq \Set{\phi \in C([-r, 0], \mathbb{R}^n)}{\norm{\phi} < \delta}, \\
	\widetilde{U}
	&\coloneqq \Set{\phi \in C([-r, 0], \mathbb{R}^n)}{\norm{\phi} < \widetilde{\delta}}.
\end{align*}
Then
\begin{equation*}
	U \subset \widetilde{U} \subset U_0
\end{equation*}
holds.
From now on, we fix $t_0 \in \mathbb{R}$ and $\phi \in U$ and proceed with the discussion.
By applying the Peano existence theorem for RFDEs, the RFDE
\begin{equation}\label{eq: x'(t) = (L + h)|_V(x_t)}
	\Dot{x}(t)
	= L|_{\widetilde{U}}(x_t) + h|_{\mathbb{R} \times \widetilde{U}}(t, x_t)
\end{equation}
has a solution under the initial condition $x_{t_0} = \phi$.
Let $x$ be a non-continuable solution of the RFDE~\eqref{eq: x'(t) = (L + h)|_V(x_t)} under this initial condition.
Then its domain of definition is written as $[t_0 - r, t_0 + T)$ for some $T \in (0, \infty]$.

\vspace{0.5\baselineskip}
\noindent
\textbf{Step 2: Estimate by Gronwall's inequality.}
Let $t \in [t_0, t_0 + T)$ and $\theta \in [-r, 0]$.
By applying Theorem~\ref{thm: VOC formula for nonlinear eq},
\begin{equation*}
	x(t) = x^L(t - t_0; \phi, 0) + \int_{t_0}^t X^L(t - u)h(u, x_u) \mspace{2mu} \mathrm{d}u
\end{equation*}
holds for this non-continuable solution $x \colon [t_0 - r, t_0 + T) \to \mathbb{R}^n$.
When $t + \theta \ge t_0$, we have
\begin{align*}
	\abs{x(t + \theta)}
	&\le \abs{x^L(t + \theta - t_0; \phi, 0)}
	+ \int_{t_0}^{t + \theta} \abs{X^L(t - u + \theta)}\abs{h(u, x_u)} \mspace{2mu} \mathrm{d}u \\
	&\le \norm{T^L(t - t_0)\phi}
	+ \int_{t_0}^t M\mathrm{e}^{-\alpha (t - u)}\abs{h(u, x_u)} \mspace{2mu} \mathrm{d}u \\
	&\le M\mathrm{e}^{-\alpha(t - t_0)}\norm{\phi}
	+ \int_{t_0}^t M\mathrm{e}^{-\alpha t}\mathrm{e}^{\alpha u}\ep\norm{x_u} \mspace{2mu} \mathrm{d}u.
\end{align*}
When $t + \theta < t_0$, the estimate
\begin{equation*}
	\abs{x(t + \theta)}
	\le M\mathrm{e}^{-\alpha(t - t_0)}\norm{\phi}
	+ \int_{t_0}^t M\mathrm{e}^{-\alpha t}\mathrm{e}^{\alpha u}\ep\norm{x_u} \mspace{2mu} \mathrm{d}u
\end{equation*}
also holds in view of
\begin{equation*}
	\abs{x(t + \theta)}
	= \abs{\phi(t - t_0 + \theta)}
	= \abs{[T^L(t - t_0)\phi](\theta)}
	\le M\mathrm{e}^{-\alpha(t - t_0)}\norm{\phi}.
\end{equation*}
These estimates yield
\begin{equation*}
	\mathrm{e}^{\alpha(t - t_0)}\norm{x_t}
	\le M\norm{\phi}
	+ \int_{t_0}^t M\ep\mathrm{e}^{\alpha(u - t_0)} \norm{x_u} \mspace{2mu} \mathrm{d}u,
\end{equation*}
and we obtain
\begin{equation*}
	\mathrm{e}^{\alpha(t - t_0)}\norm{x_t}
	\le M\norm{\phi}\mathrm{e}^{M\ep(t - t_0)}
\end{equation*}
by applying Gronwall's inequality (see Lemma~\ref{lem: Gronwall's ineq}).
This means that
\begin{equation}\label{eq: exponential estimate}
	\norm{x_t} \le M\norm{\phi}\mathrm{e}^{-\beta(t - t_0)}
\end{equation}
holds for all $t \in [t_0, t_0 + T)$.

\vspace{0.5\baselineskip}
\noindent
\textbf{Step 3: Proof by contradiction.}
We next show that $T$ is equal to $\infty$, i.e., the non-continuable solution $x$ is defined on $[t_0 - r, \infty)$.
We suppose $T < \infty$ and derive a contradiction.
Since $\norm{x_t} < \widetilde{\delta}$ holds for all $t \in [t_0 - r, t_0 + T)$, we have
\begin{align*}
	\abs{\Dot{x}(t)}
	&\le \norm{L}\norm{x_t} + \abs{h(t, x_t)} \\
	&\le (\norm{L} + \ep)\widetilde{\delta} \\
	&< \infty.
\end{align*}
This shows that $x|_{[t_0, t_0 + T)}$ is Lipschitz continuous.
In particular, $x|_{[t_0, t_0 + T)}$ is uniformly continuous, and therefore, the limit $\lim_{t \uparrow t_0 + T} x(t)$ exists.
Since this yields the existence of the limit
\begin{equation*}
	\lim_{t \uparrow t_0 + T} x_t \eqqcolon \psi \in C([-r, 0], \mathbb{R}^n),
\end{equation*}
we have
\begin{equation*}
	\norm{\psi}
	\le M\norm{\phi}\mathrm{e}^{-\beta T}
	< M\delta
	= \widetilde{\delta},
\end{equation*}
i.e., $\psi \in \widetilde{U}$, by taking the limit as $t \uparrow t_0 + T$ in the inequality~\eqref{eq: exponential estimate}.
Then the RFDE~\eqref{eq: x'(t) = (L + h)|_V(x_t)} has a solution under the initial condition $x_{t_0 + T} = \psi$ by the Peano existence theorem for RFDEs, and one can construct a continuation of $x$.
It contradicts the property that $x$ is non-continuable.
Therefore, $T$ should be infinity.

\vspace{0.5\baselineskip}
The above steps yield the conclusion.
\end{proof}

The above proof of Theorem~\ref{thm: principle of linearized stability} is an appropriate modification of the stability part of the principle of linearized stability for ODEs (e.g., see \cite[Section 2.3]{Chicone_2006}).
It also should be compared with \cite[Theorem~2 and its proof]{Walther_1975}.
We note that the continuity of the higher-order term $h$ in the RFDE~\eqref{eq: RFDE, L + h} is sufficient for the proof.

\subsection{Poincar\'{e}-Lyapunov theorem for RFDEs}

The \textit{Poincar\'{e}-Lyapunov theorem} is also extended to RFDEs as follows.
See \cite[Exercise~2.79]{Chicone_2006} for the theorem for ODEs.

\begin{theorem}\label{thm: Poincare-Lyapunov theorem for RFDEs}
Let $\sigma \in \mathbb{R}$ be given.
If $X^L$ is exponentially stable, then there exist $M \ge 1$, $\beta > 0$, and a neighborhood $U$ of $0$ in $C([-r, 0], \mathbb{R}^n)$ such that for every $t_0 \ge \sigma$, every $\phi \in U$, and every non-continuable solution $x$ of the RFDE~\eqref{eq: RFDE, L + N + h} under the initial condition $x_{t_0} = \phi$, $x$ is defined for all $t \ge t_0$ and satisfies
\begin{equation*}
	\norm{x_t} \le M\mathrm{e}^{-\beta(t - t_0)}\norm{\phi}
\end{equation*}
for all $t \ge t_0$.
\end{theorem}

\begin{proof}
From Lemma~\ref{lem: estimate of X^L(t + theta)} and Theorem~\ref{thm: exponential stability of T^L}, we choose constants $M_0 \ge 1$ and $\alpha > 0$ so that
\begin{equation*}
	\sup_{\theta \in [-r, 0]} \abs{X^L(t + \theta)}
	\le M_0\mathrm{e}^{-\alpha t}
	\mspace{25mu}
	(t \in \mathbb{R})
\end{equation*}
and
\begin{equation*}
	\norm{T^L(t)} \le M_0\mathrm{e}^{-\alpha t}
	\mspace{25mu}
	(t \ge 0)
\end{equation*}
hold.
We also choose an $\ep > 0$ so that
\begin{equation*}
	-\beta
	\coloneqq 2M_0\ep - \alpha
	< 0.
\end{equation*}
Since $\norm{N(t)} \to 0$ as $t \to \infty$, there is an $a \in \mathbb{R}$ for this $\ep > 0$ such that $\norm{N(t)} < \ep$ holds for all $t \ge a$.
There also exists a $\widetilde{\delta} > 0$ with the following properties:
\begin{enumerate}[label=(\roman*)]
\item For all $\phi \in C([-r, 0], \mathbb{R}^n)$, $\norm{\phi} < \widetilde{\delta}$ implies $\phi \in U_0$.
\item $\norm{\phi} < \widetilde{\delta}$ implies $\abs{h(t, \phi)} \le \ep\norm{\phi}$ for all $t \in \mathbb{R}$.
\end{enumerate}

We first consider the case (I) $\sigma \ge a$.
Since
\begin{align*}
	\abs{N(t, \phi) + h(t, \phi)}
	&\le \norm{N(t)}\norm{\phi} + \abs{h(t, \phi)} \\
	&\le 2\ep\norm{\phi}
\end{align*}
holds for all $t \ge \sigma$ and all $\norm{\phi} < \widetilde{\delta}$, the conclusion is obtained in the same way as in the proof of Theorem~\ref{thm: principle of linearized stability}.
We next consider the case (II) $\sigma < a$.
We divide the proof into the following steps.

\vspace{0.5\baselineskip}
\noindent
\textbf{Step 1: Choice of a neighborhood of $0$ and a non-continuable solution.}
We choose an $R > \ep$ so that
\begin{equation*}
	R > \sup_{t \in [\sigma, a]} \norm{N(t)}.
\end{equation*}
Let
\begin{equation*}
	M \coloneqq M_0\mathrm{e}^{M_0(R - \ep)(a - \sigma)}
	\amd
	\delta \coloneqq \frac{\widetilde{\delta}}{M}.
\end{equation*}
We define open sets $U$ and $\widetilde{U}$ by
\begin{align*}
	U
	&\coloneqq \Set{\phi \in C([-r, 0], \mathbb{R}^n)}{\norm{\phi} < \delta}, \\
	\widetilde{U}
	&\coloneqq \Set{\phi \in C([-r, 0], \mathbb{R}^n)}{\norm{\phi} < \widetilde{\delta}}.
\end{align*}
Since $M > M_0 \ge 1$,
\begin{equation*}
	U \subset \widetilde{U} \subset U_0
\end{equation*}
holds.
We now fix $t_0 \ge \sigma$ and $\phi \in U$, and let $x \colon [t_0 - r, t_0 + T) \to \mathbb{R}^n$ be a non-continuable solution of an RFDE
\begin{equation*}
	\Dot{x}(t)
	= L|_{\widetilde{U}}(x_t) + N|_{\mathbb{R} \times \widetilde{U}}(t, x_t) + h|_{\mathbb{R} \times \widetilde{U}}(t, x_t)
\end{equation*}
under the initial condition $x_{t_0} = \phi$.

\vspace{0.5\baselineskip}
\noindent
\textbf{Step 2: Estimate by Gronwall's inequality.}
Let $t \in [t_0, t_0 + T)$ and $\theta \in [-r, 0]$.
By applying Theorem~\ref{thm: VOC formula for nonlinear eq},
\begin{equation*}
	x(t)
	= x^L(t - t_0; \phi, 0)
	+ \int_{t_0}^t X^L(t - u)[N(u, x_u) + h(u, x_u)] \mspace{2mu} \mathrm{d}u
\end{equation*}
holds for the non-continuable solution $x \colon [t_0 - r, t_0 + T) \to \mathbb{R}^n$.
When $t + \theta \ge t_0$, we have
\begin{align*}
	&\abs{x(t + \theta)} \\
	&\le \abs{x^L(t + \theta - t_0; \phi, 0)}
	+ \int_{t_0}^{t + \theta} \abs{X^L(t + \theta - u)}
		\bigl[ \norm{N(u)}\norm{x_u} + \abs{h(u, x_u)} \bigr]
	\mspace{2mu} \mathrm{d}u \\
	&\le \norm{T^L(t - t_0)\phi}
	+ \int_{t_0}^t M_0\mathrm{e}^{-\alpha(t - u)}
		\bigl[ \norm{N(u)}\norm{x_u} + \abs{h(u, x_u)} \bigr]
	\mspace{2mu} \mathrm{d}u \\
	&\le M_0\mathrm{e}^{-\alpha(t - t_0)}\norm{\phi}
	+ \int_{t_0}^t M_0\mathrm{e}^{-\alpha t}\mathrm{e}^{\alpha u}(\norm{N(u)} + \ep) \norm{x_u} \mspace{2mu} \mathrm{d}u.
\end{align*}
When $t + \theta \le 0$, the estimate
\begin{equation*}
	\abs{x(t + \theta)}
	\le M_0\mathrm{e}^{-\alpha(t - t_0)}\norm{\phi}
	+ \int_{t_0}^t M_0\mathrm{e}^{-\alpha t}\mathrm{e}^{\alpha u}(\norm{N(u)} + \ep)\norm{x_u} \mspace{2mu} \mathrm{d}u
\end{equation*}
also holds in view of
\begin{equation*}
	\abs{x(t + \theta)}
	= \abs{\phi(t - t_0 + \theta)}
	= \abs{[T_L(t - t_0)\phi](\theta)}
	\le M_0\mathrm{e}^{-\alpha(t - t_0)}\norm{\phi}.
\end{equation*}
These estimates yield
\begin{equation*}
	\mathrm{e}^{\alpha(t - t_0)}\norm{x_t}
	\le M_0\norm{\phi}
	+ \int_{t_0}^t
		M_0(\norm{N(u)} + \ep) \mathrm{e}^{\alpha (u - t_0)}\norm{x_u}
	\mspace{2mu} \mathrm{d}u,
\end{equation*}
and we obtain
\begin{equation*}
	\mathrm{e}^{\alpha(t - t_0)}\norm{x_t}
	\le M_0\norm{\phi}
		\exp \biggl(
			\int_{t_0}^t M_0(\norm{N(u)} + \ep)
		\mspace{2mu} \mathrm{d}u \biggr)
\end{equation*}
by Gronwall's inequality (see Lemma~\ref{lem: Gronwall's ineq}).
This means that
\begin{equation}\label{eq: exponential estimate'}
	\norm{x_t}
	\le M_0\norm{\phi}\exp \biggl(
		\int_{t_0}^t [M_0(\norm{N(u)} + \ep) - \alpha] \mspace{2mu} \mathrm{d}u
	\biggr)
\end{equation}
holds for all $t \in [t_0, t_0 + T)$.

The remaining consideration is further divided into the following cases:
\begin{itemize}
\item Case: $t_0 + T < a$.
Let $t \in [t_0, t_0 + T)$.
Since $t < a$, the exponential term in \eqref{eq: exponential estimate'} is estimated from above by $\mathrm{e}^{[M_0(R + \ep) - \alpha](t - t_0)}$, which is equal to
\begin{equation*}
	\mathrm{e}^{M_0(R - \ep)(t - t_0)} \mathrm{e}^{-\beta(t - t_0)}
\end{equation*}
by the choice of $\beta$.
In view of $\sigma \le t_0 \le t < a$, the above is also estimated from above by
\begin{equation*}
	\mathrm{e}^{M_0(R - \ep)(a - \sigma)} \mathrm{e}^{-\beta(t - t_0)},
\end{equation*}
and therefore, inequality~\eqref{eq: exponential estimate} holds with $M = M_0\mathrm{e}^{M_0(R - \ep)(a - \sigma)}$.
\item Case: $t_0 + T \ge a$.
Let $t \in [t_0, t_0 + T)$.
The integral in \eqref{eq: exponential estimate'} is estimated from above by
\begin{align*}
	&\int_{t_0}^a [M_0(R + \ep) - \alpha] \mspace{2mu} \mathrm{d}u
	+ \int_a^t (-\beta) \mspace{2mu} \mathrm{d}u \\
	&= [M_0(R + \ep) - \alpha](a - t_0) + (-\beta)(t - a).
\end{align*}
In view of $t - a = (t - t_0) + (t_0 - a)$, the above value becomes
\begin{equation*}
	M_0(R - \ep)(a - t_0) + (-\beta)(t - t_0),
\end{equation*}
which is also estimated from above by
\begin{equation*}
	M_0(R - \ep)(a - \sigma) + (-\beta)(t - t_0)
\end{equation*}
because of $t_0 \ge \sigma$ and $M_0(R - \ep) > 0$.
Therefore, inequality~\eqref{eq: exponential estimate} holds with $M = M_0\mathrm{e}^{M_0(R - \ep)(a - \sigma)}$.
\end{itemize}

\vspace{0.5\baselineskip}
The remainder of the proof is same as in the proof of Theorem~\ref{thm: principle of linearized stability}.
This completes the proof.
\end{proof}

\section*{Acknowledgment}

This work was supported by JSPS Grant-in-Aid for Young Scientists Grant Number JP19K14565.

\appendix

\section{Riemann-Stieltjes integrals with respect to matrix-valued functions}\label{sec: RS integrals wrt matrix-valued functions}

Throughout this appendix, let $\mathbb{K} = \mathbb{R}$ or $\mathbb{C}$, $n \ge 1$ be an integer, and $[a, b]$ be a closed and bounded interval of $\mathbb{R}$.
In this appendix, we study Riemann-Stieltjes integrals with respect to matrix-valued functions.
We refer the reader to \cite[Chapter~1]{Widder_1941} and \cite[Appendix~D]{Shapiro_2018} as references of Riemann-Stieltjes integrals for scalar-valued functions.
See also \cite[Section~3.1]{Hino--Murakami--Naito_1991} and \cite[Section~I.1 in Appendix~I]{Diekmann--vanGils--Lunel--Walther_1995}.

\subsection{Definitions}\label{subsec: definitions on RS integrals}

\begin{definition}\label{dfn: tagged partition}
Let $P : a = x_0 < x_1 < \dots < x_m = b$ be a partition of $[a, b]$ for some integer $m \ge 1$.
For a finite sequence $\xi \coloneqq (\xi_k)_{k = 1}^m$ satisfying
\begin{equation*}
	x_{k - 1} \le \xi_k \le x_k
	\mspace{25mu}
	(k \in \{1, \dots, m\}),
\end{equation*}
we call a pair $(P, \xi)$ a \textit{tagged partition} of $[a, b]$.
For the tagged partition $(P, \xi)$, let
\begin{equation*}
	\abs{(P, \xi)}
	\coloneqq
	\abs{P}
	\coloneqq \max_{1 \le k \le m} (x_k - x_{k - 1}),
\end{equation*}
which is called the \textit{norm} of $(P, \xi)$.
\end{definition}

The above terminology comes from \cite{Gordon_1991}.

\begin{definition}\label{dfn: RS sum}
Let $f, \alpha \colon [a, b] \to M_n(\mathbb{K})$ be functions.
For a tagged partition $(P, \xi)$ of $[a, b]$ given in Definition~\ref{dfn: tagged partition}, let
\begin{equation*}
	S(f; \alpha, (P, \xi))
	\coloneqq \sum_{k = 1}^m \mspace{2mu} [\alpha(x_k) - \alpha(x_{k - 1})] f(\xi_k).
\end{equation*}
We call $S(f; \alpha, (P, \xi))$ the \textit{Riemann-Stieltjes sum} of $f$ with respect to $\alpha$ under the tagged partition $(P, \xi)$.
\end{definition}

\begin{definition}\label{dfn: RS integral}
Let $f, \alpha \colon [a, b] \to M_n(\mathbb{K})$ be functions.
We say that $f$ is \textit{Riemann-Stieltjes integrable with respect to $\alpha$} if there exists a $J \in M_n(\mathbb{K})$ with the following property: For every $\ep > 0$, there exists a $\delta > 0$ such that for all tagged partition $(P, \xi)$ of $[a, b]$, $\abs{(P, \xi)} < \delta$ implies
\begin{equation*}
	\abs{S(f; \alpha, (P, \xi)) - J} < \ep.
\end{equation*}
We note that such a $J$ is unique if it exists.
It is called the \textit{Riemann-Stieltjes integral of $f$ with respect to $\alpha$} and is denoted by $\int_a^b \mathrm{d}\alpha(x) \mspace{2mu} f(x)$.
\end{definition}

\subsubsection{Remarks}

\begin{remark}
One can also consider a sum
\begin{equation*}
	\sum_{k = 1}^m f(\xi_k) [\alpha(x_k) - \alpha(x_{k - 1})],
\end{equation*}
which is different from $S(f; \alpha, (P, \xi))$ in general.
If a limit of the above sum as $\abs{(P, \xi)} \to 0$ exists in the sense of Definition~\ref{dfn: RS integral}, we will write the limit as $\int_a^b f(x) \mspace{2mu} \mathrm{d}\alpha(x)$.
By taking the transpose,
\begin{equation*}
	\left( \int_a^b \mathrm{d}\alpha(x) \mspace{2mu} f(x) \right)^\mathrm{T}
	= \int_a^b f(x)^\mathrm{T} \mspace{2mu} \mathrm{d}\alpha(x)^\mathrm{T}
\end{equation*}
holds.
Here $A^\mathrm{T}$ denotes the transpose of a matrix $A \in M_n(\mathbb{K})$.
When $n = 1$,
\begin{equation*}
	\int_a^b \mathrm{d}\alpha(x) \mspace{2mu} f(x)
	= \int_a^b f(x) \mspace{2mu} \mathrm{d}\alpha(x)
\end{equation*}
holds.
\end{remark}

\begin{remark}
The notions of the Riemann-Stieltjes sum $S(f; \alpha, (P, \xi))$ and the Riemann-Stieltjes integrability of $f$ with respect to $\alpha$ are also defined for functions
\begin{equation*}
	f \colon [a, b] \to \mathbb{K}^n
	\amd
	\alpha \colon [a, b] \to M_n(\mathbb{K}).
\end{equation*}
In this case, the sum $S(f; \alpha, (P, \xi))$ and the integral $\int_a^b \mathrm{d}\alpha(x) \mspace{2mu} f(x)$ belong to $\mathbb{K}^n$.
\end{remark}

\subsection{Reduction to scalar-valued case}\label{subsec: reduction to RS integrals for scalar-valued functions}

Since the linear space $M_n(\mathbb{K})$ is finite-dimensional, the operator norm $\abs{\argdot}$ on $M_n(\mathbb{K})$ is equivalent to the norm $\abs{\argdot}_2$ on $M_n(\mathbb{K})$ defined by
\begin{equation}\label{eq: matrix norm}
	\abs{A}_2
	\coloneqq \sqrt{
		\sum_{i, j \in \{1, \dots, n\}} \abs{a_{i, j}}^2
	},
\end{equation}
where $a_{i, j}$ is the $(i, j)$-component of the matrix $A \in M_n(\mathbb{K})$.
This means that the notion of convergence in $M_n(\mathbb{K})$ can be treated component-wise.

\begin{lemma}\label{lem: reduction to RS integrals of vector-valued functions}
Let $f, \alpha \colon [a, b] \to M_n(\mathbb{K})$ be functions.
Then the following properties are equivalent:
\begin{enumerate}[label=(\alph*)]
\item $f$ is Riemann-Stieltjes integrable with respect to $\alpha$.
\item For each column vector $f_j \colon [a, b] \to \mathbb{K}^n$ of $f = (f_1, \dots, f_n)$, it is Riemann-Stieltjes integrable with respect to $\alpha$.
\end{enumerate}
Furthermore,
\begin{equation*}
	\int_a^b \mathrm{d}\alpha(x) \mspace{2mu} f(x)
	= \left(
		\int_a^b \mathrm{d}\alpha(x) \mspace{2mu} f_j(x)
	\right)_{j = 1}^n
\end{equation*}
holds when one of the above properties are satisfied.
\end{lemma}

The proof is based on the definition of the matrix product and on the property that the operator norm $\abs{\argdot}$ is equivalent to the norm $\abs{\argdot}_2$ given in \eqref{eq: matrix norm}.
Therefore, we omit the proof.

\begin{lemma}\label{lem: reduction to RS integrals of scalar-valued functions}
Let $f \colon [a, b] \to \mathbb{K}^n$ and $\alpha \colon [a, b] \to M_n(\mathbb{K})$ be functions with $f = (f_1, \dots, f_n)$ and $\alpha = (\alpha_{i, j})_{i, j \in \{1, \dots, n\}}$.
If $f_j \colon [a, b] \to \mathbb{K}$ is Riemann-Stieltjes integrable with respect to $\alpha_{i, j} \colon [a, b] \to \mathbb{K}$ for every $i, j \in \{1, \dots, n\}$, then so is $f$ with respect to $\alpha$.
Furthermore,
\begin{equation*}
	\int_a^b \mathrm{d}\alpha(x) \mspace{2mu} f(x)
	= \left(
		\sum_{j = 1}^n \int_a^b f_j(x) \mspace{2mu} \mathrm{d}\alpha_{i, j}(x)
	\right)_{i = 1}^n
\end{equation*}
holds.
\end{lemma}

\begin{proof}
By the definition of the product of a matrix and a vector, the $i$-th component of $S(f; \alpha, (P, \xi)) \in \mathbb{K}^n$ is equal to
\begin{equation*}
	\sum_{j = 1}^n S(f_j; \alpha_{i, j}, (P, \xi)).
\end{equation*}
Therefore, the conclusion is obtained by the triangle inequality.
\end{proof}

The converse of Lemma~\ref{lem: reduction to RS integrals of scalar-valued functions} does not necessarily hold as the following example shows.

\begin{example}
Let $n = 2$ and $g, \beta \colon [a, b] \to \mathbb{K}$ be given functions.
Let
\begin{equation*}
	f \coloneqq (g, -g) \colon [a, b] \to \mathbb{K}^2
	\amd
	\alpha \coloneqq (\beta)_{i, j \in \{1, 2\}} \colon [a, b] \to M_2(\mathbb{K}),
\end{equation*}
i.e., $f_1 = f$, $f_2 = -f$, and $\alpha_{i, j} = \beta$.
Then the Riemann-Stieltjes sum of $f$ with respect to $\alpha$ is equal to $0$ under any tagged partition of $[a, b]$.
This means that $f$ is Riemann-Stieltjes integrable with respect to $\alpha$ for any pair $(g, \beta)$ of functions.
\end{example}

In view of the above example, the Riemann-Stieltjes integration of vector-valued functions with respect to matrix-valued functions is not completely reduced to that for scalar-valued functions.
However, it is often useful to reduce the integration to scalar-valued case in view of Lemma~\ref{lem: reduction to RS integrals of scalar-valued functions}.

\subsection{Fundamental results}\label{subsec: fundamental results on RS integrals}

The following are fundamental results on Riemann-Stieltjes integrals for matrix-valued functions.

\subsubsection{Reversal formula}

\begin{theorem}\label{thm: reversal formula for RS integrals}
Let $f, \alpha \colon [a, b] \to M_n(\mathbb{K)}$ be functions.
We define functions $\Bar{f}, \Bar{\alpha} \colon [-b, -a] \to M_n(\mathbb{K})$ by
\begin{equation*}
	\Bar{f}(y) \coloneqq f(-y), \mspace{15mu} \Bar{\alpha}(y) \coloneqq \alpha(-y)
\end{equation*}
for $y \in [-b, -a]$.
If $f$ is Riemann-Stieltjes integrable with respect to $\alpha$, then so is $\Bar{f}$ with respect to $\Bar{\alpha}$.
Furthermore,
\begin{equation}\label{eq: reversal formula}
	\int_{-b}^{-a} \mathrm{d}\Bar{\alpha}(y) \mspace{2mu} \Bar{f}(y)
	= -\int_a^b \mathrm{d}\alpha(x) \mspace{2mu} f(x)
\end{equation}
holds.
\end{theorem}

We call Eq.~\eqref{eq: reversal formula} the \textit{reversal formula} for Riemann-Stieltjes integrals.
The proof is obtained by returning to the definition of Riemann-Stieltjes integrals.
Therefore, it can be omitted.

\subsubsection{Integration by parts formula}

The following is the \textit{integration by parts formula} for Riemann-Stieltjes integrals with respect to matrix-valued functions.

\begin{theorem}\label{thm: integration by parts for RS integral}
Let $f, \alpha \colon [a, b] \to M_n(\mathbb{K})$ be functions.
If $f$ is Riemann-Stieltjes integrable with respect to $\alpha$, then so is $\alpha$ with respect to $f$.
Furthermore,
\begin{equation*}
	\int_a^b \mathrm{d}\alpha(x) \mspace{2mu} f(x)
	= [\alpha(x)f(x)]_{x = a}^b - \int_a^b \alpha(x) \mspace{2mu} \mathrm{d}f(x)
\end{equation*}
holds.
Here $[\alpha(x)f(x)]_{x = a}^b \coloneqq \alpha(b)f(b) - \alpha(a)f(a)$.
\end{theorem}

The proof is basically same as the proof for the case $n = 1$ (i.e., the scalar-valued case).
See \cite[Proposition~D.3]{Shapiro_2018} for the proof of this case.
See also \cite[Theorems~4a and 4b in Chapter~1]{Widder_1941}.

\subsection{Integrability}\label{subsec: integrability for RS integrals}

\subsubsection{Matrix-valued functions of bounded variation}

We first recall the definition of matrix-valued functions of bounded variation.

\begin{definition}\label{dfn: bounded variation}
Let $\alpha \colon [a, b] \to M_n(\mathbb{K})$ be a function.
For each partition $P : a = x_0 < x_1 < \dots < x_m = b$ of $[a, b]$, let
\begin{equation*}
	\Var(\alpha; P)
	\coloneqq \sum_{k = 1}^m \abs{\alpha(x_k) - \alpha(x_{k - 1})},
\end{equation*}
which is called the \textit{variation} of $\alpha$ under the partition $P$.
The value
\begin{equation*}
	\Var(\alpha)
	\coloneqq \sup\Set{\Var(\alpha; P)}{\text{$P$ is a partition of $[a, b]$}}
\end{equation*}
is called the \textit{total variation} of $\alpha$.
$\alpha$ is said to be \textit{of bounded variation} if $\Var(\alpha) < \infty$.
\end{definition}

Since the operator norm $\abs{\argdot}$ on $M_n(\mathbb{K})$ and the norm $\abs{\argdot}_2$ on $M_n(\mathbb{K})$ given in \eqref{eq: matrix norm} are equivalent, a matrix-valued function $\alpha \colon [a, b] \to M_n(\mathbb{K})$ is of bounded variation if and only if each component function $\alpha_{i, j} \colon [a, b] \to \mathbb{K}$ is of bounded variation.

\begin{remark}
Let $\alpha \colon [a, b] \to M_n(\mathbb{K})$ be a function.
Then for any $c \in (a, b)$,
\begin{equation}\label{eq: total variations on sub-intervals}
	\Var \bigl(\alpha|_{[a, c]} \bigr) + \Var \bigl( \alpha|_{[c, b]} \bigr)
	= \Var(\alpha)
\end{equation}
holds.
This equality is obtained from
\begin{equation*}
	\Var \bigl( \alpha|_{[a, c]}; P_1 \bigr) + \Var \bigl( \alpha|_{[c, b]}; P_2 \bigr)
	= \Var(\alpha; P),
\end{equation*}
where $P_1$ is a partition of $[a, c]$, $P_2$ is a partition of $[c, b]$, and $P$ is the partition of $[a, b]$ obtained by joining $P_1$ and $P_2$.
\end{remark}

\begin{lemma}\label{lem: estimate of norm of RS integral}
Let $f, \alpha \colon [a, b] \to M_n(\mathbb{K})$ be functions.
If $f$ is Riemann-Stieltjes integrable with respect to $\alpha$, then
\begin{equation}\label{eq: estimate of norm of RS integral}
	\abs{\int_a^b \mathrm{d}\alpha(x) \mspace{2mu} f(x)}
	\le \Var(\alpha) \cdot \sup_{x \in [a, b]} \abs{f(x)}
\end{equation}
holds.
\end{lemma}

\begin{proof}
Let $(P, \xi)$ be a tagged partition of $[a, b]$ given in Definition~\ref{dfn: tagged partition}.
Since $\abs{AB} \le \abs{A}\abs{B}$ holds for any $A, B \in M_n(\mathbb{K})$, we have
\begin{equation*}
	\abs{S(f; \alpha, (P, \xi))}
	\le \sum_{k = 1}^m \mspace{2mu} \abs{\alpha(x_k) - \alpha(x_{k - 1})}\abs{f(\xi_k)}
	\le \Var(\alpha) \cdot \sup_{x \in [a, b]} \abs{f(x)}.
\end{equation*}
Then the remaining proof is essentially same as the scalar-valued case.
\end{proof}

\begin{remark}
In the completely similar way, \eqref{eq: estimate of norm of RS integral} also holds for any continuous function $f \colon [a, b] \to \mathbb{K}^n$.
This can also be seen from Lemma~\ref{lem: estimate of norm of RS integral} because for any $A \in M_n(\mathbb{K})$ of the form
\begin{equation*}
	A = (a \mspace{5mu} 0 \mspace{5mu} \cdots \mspace{5mu} 0)
	\mspace{25mu}
	(\text{$a \in \mathbb{K}^n$, $0 \in \mathbb{K}^n$}),
\end{equation*}
$\abs{A} = \abs{a}$ holds.
\end{remark}

\subsubsection{Integrability of matrix-valued functions}

The following is a fundamental theorem on the Riemann-Stieltjes integrability for scalar-valued functions.

\begin{theorem}\label{thm: RS integrability for scalar-valued functions}
Let $f, \alpha \colon [a, b] \to \mathbb{K}$ be functions.
If $f$ is continuous and $\alpha$ is of bounded variation, then $f$ is Riemann-Stieltjes integrable with respect to $\alpha$.
\end{theorem}

See \cite[Theorem~D.1]{Shapiro_2018} for a proof, which is valid for the case $\mathbb{K} = \mathbb{C}$ because it does not use the order structure.
By using Theorem~\ref{thm: RS integrability for scalar-valued functions}, one can obtain the following.

\begin{theorem}\label{thm: RS integrability}
Let $f, \alpha \colon [a, b] \to M_n(\mathbb{K})$ be functions.
If $f$ is continuous and $\alpha$ is of bounded variation, then $f$ is Riemann-Stieltjes integrable with respect to $\alpha$.
\end{theorem}

\begin{proof}
From Lemma~\ref{lem: reduction to RS integrals of vector-valued functions}, the problem is reduced to the Riemann-Stieltjes integrability of each column vector of $f$ with respect to $\alpha$.
From Lemma~\ref{lem: reduction to RS integrals of scalar-valued functions}, it is sufficient to show that each component $f_{i, j} \colon [a, b] \to \mathbb{K}$ of $f$ is Riemann-Stieltjes integrable with respect to each component $\alpha_{i, j} \colon [a, b] \to \mathbb{K}$ of $\alpha$.
Since each $f_{i, j}$ is continuous and each $\alpha_{i, j}$ is of bounded variation, the conclusion is obtained from Theorem~\ref{thm: RS integrability for scalar-valued functions}.
\end{proof}

The following is the result on additivity of Riemann-Stieltjes integrals with respect to matrix-valued functions on sub-intervals.

\begin{theorem}\label{thm: additivity for RS integrals on sub-intervals}
Let $f \colon [a, b] \to M_n(\mathbb{K})$ be a continuous function and $\alpha \colon [a, b] \to M_n(\mathbb{K})$ be a function of bounded variation.
Then for any $c \in (a, b)$,
\begin{equation*}
	\int_a^b \mathrm{d}\alpha(x) \mspace{2mu} f(x)
	= \int_a^c \mathrm{d}\alpha(x) \mspace{2mu} f(x) + \int_c^b \mathrm{d}\alpha(x) \mspace{2mu} f(x)
\end{equation*}
holds.
\end{theorem}

The proof is same as that for the case $n = 1$.
See \cite[Proposition~D.2]{Shapiro_2018} for the proof.
We note that the statement can be proved by considering partitions of $[a, b]$ with $c \in (a, b)$ as an intermediate point.

\begin{remark}
In Theorem~\ref{thm: additivity for RS integrals on sub-intervals}, the assumptions that $f$ is continuous and $\alpha$ is of bounded variation are essential because these assumptions ensure the existence of three integrals (see \eqref{eq: total variations on sub-intervals} and Theorem~\ref{thm: RS integrability}).
Without these assumptions, the integral in the left-hand side does not necessarily exist even if the integrals in the right-hand side exist.
Such a situation will occur when the functions $f$ and $\alpha$ share a discontinuity at $c$.
See \cite[Section~5 in Chapter~I]{Widder_1941} for the detail.
\end{remark}

\subsection{Integration with respect to continuously differentiable functions}\label{subsec: integration wrt C^1-functions}

The following theorem shows a relationship between Riemann-Stieltjes integrals and Riemann integrals.

\begin{theorem}\label{thm: RS integral wrt C^1-function}
Let $f \colon [a, b] \to M_n(\mathbb{K})$ be a Riemann integrable function and $\alpha \colon [a, b] \to M_n(\mathbb{K})$ be a continuously differentiable function.
Then $f$ is Riemann-Stieltjes integrable with respect to $\alpha$, and
\begin{equation*}
	\int_a^b \mathrm{d}\alpha(x) \mspace{2mu} f(x)
	= \int_a^b \alpha'(x)f(x) \mspace{2mu} \mathrm{d}x
\end{equation*}
holds.
Here the right-hand side is a Riemann integral.
\end{theorem}

Since the above statement is not mentioned in \cite{Widder_1941} and \cite{Shapiro_2018} even for the case $n = 1$, we now give an outline of the proof.

\begin{proof}[Outline of the proof of Theorem~\ref{thm: RS integral wrt C^1-function}]
Let $(P, \xi)$ be a tagged partition of $[a, b]$ given in Definition~\ref{dfn: tagged partition}.
Let
\begin{equation*}
	S(\alpha'f; (P, \xi))
	\coloneqq \sum_{k = 1}^n \mspace{2mu} (x_k - x_{k - 1})\alpha'(\xi_k)f(\xi_k).
\end{equation*}
Since
\begin{equation*}
	\alpha(x_k) - \alpha(x_{k - 1})
	= \int_{x_{k - 1}}^{x_k} \alpha'(t) \mspace{2mu} \mathrm{d}t
\end{equation*}
holds for each $k \in \{1, \dots, m\}$ by the fundamental theorem of calculus, we have
\begin{equation*}
	S(f; \alpha, (P, \xi)) - S(\alpha'f; (P, \xi))
	= \sum_{k = 1}^m \int_{x_{k - 1}}^{x_k}
		[\alpha'(t) - \alpha'(\xi_k)]
	\mspace{2mu} \mathrm{d}t
	\cdot f(\xi_k).
\end{equation*}
From this, we also have
\begin{equation*}
	\abs{S(f; \alpha, (P, \xi)) - S(\alpha'f; (P, \xi))}
	\le \sum_{k = 1}^m \int_{x_{k - 1}}^{x_k}
		\abs{\alpha'(t) - \alpha'(\xi_k)}
	\mspace{2mu} \mathrm{d}t
	\cdot \abs{f(\xi_k)}.
\end{equation*}
By combining this and the uniform continuity of $\alpha'$, one can obtain the conclusion.
\end{proof}

When $n = 1$ and $\mathbb{K} = \mathbb{R}$, one can use the mean value theorem for the proof of Theorem~\ref{thm: RS integral wrt C^1-function}.

\subsection{Integration with respect to absolutely continuous functions}\label{subsec: integration wrt AC functions}

The following theorem should be compared with Theorem~\ref{thm: RS integral wrt C^1-function}.

\begin{theorem}\label{thm: RS integral of continuous function wrt AC function}
Let $f \colon [a, b] \to M_n(\mathbb{K})$ be a continuous function and $\alpha \colon [a, b] \to M_n(\mathbb{K})$ be an absolutely continuous function.
Then
\begin{equation*}
	\int_a^b \mathrm{d}\alpha(x) \mspace{2mu} f(x)
	= \int_a^b \alpha'(x)f(x) \mspace{2mu} \mathrm{d}x
\end{equation*}
holds.
Here the right-hand side is a Lebesgue integral.
\end{theorem}

See \cite[Theorem~6a in Chapter~I]{Widder_1941} for the proof of the scalar-valued case.
We note that the existence of the Riemann-Stieltjes integral in the left-hand side is ensured by Theorem~\ref{thm: RS integrability} because the absolutely continuous function $\alpha$ is of bounded variation.
We also note that the function $[a, b] \ni x \mapsto \alpha'(x)f(x) \in M_n(\mathbb{K})$ is Lebesgue integrable because it is measurable and
\begin{equation*}
	\int_a^b \abs{\alpha'(x)f(x)} \mspace{2mu} \mathrm{d}x
	\le \int_a^b \abs{\alpha'(x)}\abs{f(x)} \mspace{2mu} \mathrm{d}x
	\le \norm{\alpha'}_1\norm{f}
	< \infty
\end{equation*}
holds.

Since it is interesting to compare the proof of Theorem~\ref{thm: RS integral wrt C^1-function} and the proof of Theorem~\ref{thm: RS integral of continuous function wrt AC function}, we now give an outline of the proof.

\begin{proof}[Outline of the proof of Theorem~\ref{thm: RS integral of continuous function wrt AC function}]
Let $(P, \xi)$ be a tagged partition of $[a, b]$ given in Definition~\ref{dfn: tagged partition}.
Since $\alpha = \alpha(0) + V\alpha'$,
\begin{equation*}
	S(f; \alpha, (P, \xi))
	= \sum_{k = 1}^m \int_{x_{k - 1}}^{x_k} \alpha'(t) \mspace{2mu} \mathrm{d}t \cdot f(\xi_k)
\end{equation*}
holds.
Therefore, we have
\begin{equation*}
	\int_a^b \alpha'(x)f(x) \mspace{2mu} \mathrm{d}x - S(f; \alpha, (P, \xi))
	= \sum_{k = 1}^m \int_{x_{k - 1}}^{x_k} \alpha'(t)[f(t) - f(\xi_k)] \mspace{2mu} \mathrm{d}t.
\end{equation*}
In combination with the uniform continuity of $f$, the conclusion is obtained by taking the limit as $\abs{(P, \xi)} \to 0$.
\end{proof}

\subsection{Proof of the theorem on iterated integrals}\label{subsec: proof of thm on iterated integrals}

In this subsection, we give a proof of Theorem~\ref{thm: iterated integrals for RS integral}.

\begin{proof}[Proof of Theorem~\ref{thm: iterated integrals for RS integral}]
We define a bounded linear operator $T \colon C([a, b], M_n(\mathbb{K})) \to M_n(\mathbb{K})$ by
\begin{equation*}
	Tg \coloneqq \int_a^b \mathrm{d}\alpha(x) \mspace{2mu} g(x)
\end{equation*}
for $g \in C([a, b], M_n(\mathbb{K}))$.
From Lemma~\ref{lem: Riemann integral of two variables function}, the left-hand side of \eqref{eq: iterated RS integrals} is equal to
\begin{equation*}
	T \int_c^d f(\argdot, y) \mspace{2mu} \mathrm{d}y,
\end{equation*}
which is also equal to $\int_c^d Tf(\argdot, y) \mspace{2mu} \mathrm{d}y$ since $T$ is a bounded linear operator.
By the definition of $T$, this integral is equal to the right-hand side of \eqref{eq: iterated RS integrals}.
This completes the proof.
\end{proof}

\section{Riesz representation theorem}\label{sec: Riesz rep thm}

Throughout this appendix, let $\mathbb{K} = \mathbb{R}$ or $\mathbb{C}$ and let $[a, b]$ be a closed and bounded interval of $\mathbb{R}$.

The following is the cerebrated \textit{Riesz representation theorem}.

\begin{theorem}\label{thm: Riesz rep thm}
For any continuous linear functional $A \colon C([a, b], \mathbb{K}) \to \mathbb{K}$, there exists a function $\alpha \colon [a, b] \to \mathbb{K}$ with the following properties: (i) $\Var(\alpha) = \norm{A}$ and (ii) every $f \in C([a, b], \mathbb{K})$ is Riemann-Stieltjes integrable with respect to $\alpha$, and (iii)
\begin{equation*}
	A(f) = \int_a^b f(x) \mspace{2mu} \mathrm{d}\alpha(x)
\end{equation*}
holds for all $f \in C([a, b], \mathbb{K})$.
\end{theorem}

In a proof of Theorem~\ref{thm: Riesz rep thm} (e.g., see discussions on \cite[Chapter~9]{Shapiro_2018}), we construct such a function $\alpha$ by using a continuous linear extension
\begin{equation*}
	\Bar{A} \colon B([a, b], \mathbb{K}) \to \mathbb{K}
\end{equation*}
of $A$ with $\norm{\Bar{A}} = \norm{A}$.
Here $B([a, b], \mathbb{K})$ denotes the linear space of all bounded functions from $[a, b]$ to $\mathbb{K}$ endowed with the supremum norm.
Its existence is ensured by the Hahn-Banach extension theorem in normed spaces (see \cite[Theorem~1 in Section~5 of Chapter~IV]{Yosida_1980}).
See also \cite[Section~4 of Chapter~IV]{Banach}.

\begin{remark}
The Riemann-Stieltjes integrability of any $f \in C([a, b], \mathbb{K})$ with respect to the constructed function $\alpha$ is also obtained in the proof.
This should be compared with Theorem~\ref{thm: RS integrability for scalar-valued functions}.
\end{remark}

The following is a corollary of Theorem~\ref{thm: Riesz rep thm}.

\begin{corollary}\label{cor: a cor of Riesz rep thm}
For any integer $n \ge 1$ and any continuous linear map $A \colon C([a, b], \mathbb{K}^n) \to \mathbb{K}^n$, there exists a function $\alpha \colon [a, b] \to M_n(\mathbb{K})$ of bounded variation such that
\begin{equation*}
	A(f) = \int_a^b \mathrm{d}\alpha(x) \mspace{2mu} f(x)
\end{equation*}
holds for all $f \in C([a, b], \mathbb{K}^n)$.
\end{corollary}

Corollary~\ref{cor: a cor of Riesz rep thm} has been used in the literature of RFDEs (e.g., see \cite{Hale_1971_Springer}, \cite{Hale_1977}, \cite{Hale--VerduynLunel_1993}, and \cite{Diekmann--vanGils--Lunel--Walther_1995}).
We now give the proof of Corollary~\ref{cor: a cor of Riesz rep thm} because it is not given in these references.

\begin{proof}[Proof of Corollary~\ref{cor: a cor of Riesz rep thm}]
Let $(\bm{e}_1, \dots, \bm{e}_n)$ be the standard basis of $\mathbb{K}^n$.
For each $g \in C([a, b], \mathbb{K})$ and each $j \in \{1, \dots, n\}$, let $g\bm{e}_j \in C([a, b], \mathbb{K}^n)$ be defined by
\begin{equation*}
	(g\bm{e}_j)(x) \coloneqq g(x)\bm{e}_j
\end{equation*}
for $x \in [a, b]$.
For each $i, j \in \{1, \dots, n\}$, we define a functional $A_{i, j} \colon C([a, b], \mathbb{K}) \to \mathbb{K}$ by
\begin{equation*}
	A_{i, j}(g) \coloneqq A(g\bm{e}_j)_i.
\end{equation*}
Here $y_i$ denotes the $i$-th component of $y \in \mathbb{K}^n$.
Since $A_{i, j}$ is a continuous linear functional, one can choose a function $\alpha_{i, j} \colon [a, b] \to \mathbb{K}$ of bounded variation so that
\begin{equation*}
	A_{i, j}(g) = \int_a^b g(x) \mspace{2mu} \mathrm{d}\alpha_{i, j}(x)
\end{equation*}
holds for all $g \in C([a, b], \mathbb{K})$ from Theorem~\ref{thm: Riesz rep thm}.
By using $f = \sum_{j = 1}^n f_j\bm{e}_j$ for $f = (f_1, \dots, f_n)$, we have
\begin{equation*}
	A(f)_i
	= \sum_{j = 1}^n A(f_j\bm{e}_j)_i
	= \sum_{j = 1}^n A_{i, j}(f_j)
	= \sum_{j = 1}^n \int_a^b f_j(x) \mspace{2mu} \mathrm{d}\alpha_{i, j}(x).
\end{equation*}
From Lemma~\ref{lem: reduction to RS integrals of scalar-valued functions}, this yields that
\begin{equation*}
	A(f) = \int_a^b \mathrm{d}\alpha(x) \mspace{2mu} f(x)
\end{equation*}
holds for all $f \in C([a, b], \mathbb{K}^n)$ by defining a matrix-valued function $\alpha \colon [a, b] \to M_n(\mathbb{K})$ of bounded variation by $\alpha \coloneqq (\alpha_{i, j})_{i, j}$.
This completes the proof.
\end{proof}

\section{Variants of Gronwall's inequality}\label{sec: variants of Gronwall's ineq}

Throughout this appendix, let $[a, b]$ be a closed and bounded interval of $\mathbb{R}$.

\subsection{Gronwall's inequality and its generalization}\label{subsec: Gronwall's ineq}

The following is known as Gronwall's inequality.

\begin{lemma}[ref.\ \cite{Hale_1980}]\label{lem: Gronwall's ineq}
Let $\alpha \in \mathbb{R}$ be a constant and $\beta \colon [a, b] \to [0, \infty)$ be a continuous function.
If a continuous function $u \colon [a, b] \to \mathbb{R}$ satisfies
\begin{equation*}
	u(t) \le \alpha + \int_a^t \beta(s)u(s) \mspace{2mu} \mathrm{d}s
\end{equation*}
for all $t \in [a, b]$, then
\begin{equation*}
	u(t) \le \alpha \exp \biggl( \int_a^t \beta(s) \mspace{2mu} \mathrm{d}s \biggr)
\end{equation*}
holds for all $t \in [a, b]$.
\end{lemma}

\begin{proof}[Outline of the proof]
To use a technique for scalar homogeneous linear ODEs, let $v(t) \coloneqq \int_a^t \beta(s)u(s) \mspace{2mu} \mathrm{d}s$.
Then the given inequality becomes
\begin{equation*}
	\Dot{v}(t) \le \beta(t)[v(t) + \alpha]
	\mspace{25mu}
	(t \in [a, b]),
\end{equation*}
where the non-negativity of $\beta$ is used.
Since the left-hand side is the derivative of the function $t \mapsto v(t) + \alpha$, it is natural to consider the derivative of
\begin{equation*}
	t \mapsto \exp\biggl( -\int_a^t \beta(s) \mspace{2mu} \mathrm{d}s \biggr)[v(t) + \alpha].
\end{equation*}
Then it holds that this function is strictly monotonically decreasing, which yields the conclusion.
\end{proof}

The following is a generalized version of Gronwall's inequality.

\begin{lemma}[refs.\ \cite{Hale_1977}, \cite{Hale_1980}, \cite{Hale--VerduynLunel_1993}]\label{lem: generalized Gronwall's ineq}
Let $\alpha \colon [a, b] \to \mathbb{R}$ and $\beta \colon [a, b] \to [0, \infty)$ be given continuous functions.
If a continuous function $u \colon [a, b] \to \mathbb{R}$ satisfies
\begin{equation*}
	u(t) \le \alpha(t) + \int_a^t \beta(s)u(s) \mspace{2mu} \mathrm{d}s
\end{equation*}
for all $t \in [a, b]$, then
\begin{equation*}
	u(t)
	\le \alpha(t) + \int_a^t \alpha(s)\beta(s) \exp\biggl( \int_s^t \beta(\tau) \mspace{2mu} \mathrm{d}\tau \biggr) \mspace{2mu} \mathrm{d}s
\end{equation*}
holds for all $t \in [a, b]$.
Furthermore, if $\alpha$ is monotonically increasing, then
\begin{equation*}
	u(t) \le \alpha(t)\exp\biggl( \int_a^t \beta(s) \mspace{2mu} \mathrm{d}s \biggr)
\end{equation*}
holds.
\end{lemma}

By letting $v(t) \coloneqq \int_a^t \beta(s)u(s) \mspace{2mu} \mathrm{d}s$, one can obtain
\begin{equation*}
	\dv{t} \mspace{2mu} \exp\biggl( -\int_a^t \beta(s) \mspace{2mu} \mathrm{d}s \biggr)v(t)
	\le \exp\biggl( -\int_a^t \beta(s) \mspace{2mu} \mathrm{d}s \biggr)\beta(t)\alpha(t).
\end{equation*}
Then the first inequality is obtained by integrating both sides in combination with $u(t) \le \alpha(t) + v(t)$.
See \cite[Section~I.6]{Hale_1980}, \cite[Lemma~3.1 in Section~1.3]{Hale_1977}, and \cite[Lemma~3.1 in Section~1.3]{Hale--VerduynLunel_1993} for the detail of the proof.

\subsection{Gronwall's inequality and RFDEs}\label{subsec: Gronwall's ineq and RFDEs}

In this subsection, let $r > 0$ and $E = (E, \norm{\argdot})$ be a normed space.
For each continuous function $u \colon [a - r, b] \to E$ and each $t \in [a, b]$, let $u_t \in C([-r, 0], E)$ be defined by
\begin{equation*}
	u_t(\theta) \coloneqq u(t + \theta)
	\mspace{25mu}
	(\theta \in [-r, 0]).
\end{equation*}
It holds that the function $[a, b] \ni t \mapsto u_t \in C([-r, 0], E)$ is continuous.

In the context of RFDEs, it is often convenient to use the following result rather than to use Gronwall's inequality directly.

\begin{lemma}[cf.\ \cite{Hartung_2011}]\label{lem: Gronwall's ineq and RFDEs}
Let $\alpha \in \mathbb{R}$ be a constant and $\beta \colon [a, b] \to [0, \infty)$ be a given continuous function.
If a continuous function $u \colon [a - r, b] \to E$ satisfies
\begin{equation*}
	\norm{u(t)} \le \alpha + \int_a^t \beta(s)\norm{u_s} \mspace{2mu} \mathrm{d}s
\end{equation*}
for all $t \in [a, b]$, then
\begin{equation*}
	\norm{u_t} \le \max\{\norm{u_a}, \alpha\} \exp\biggl( \int_a^t \beta(s) \mspace{2mu} \mathrm{d}s \biggr)
\end{equation*}
holds for all $t \in [a, b]$.
\end{lemma}

This should be compared with \cite[Lemma~2.1]{Hartung_2011}.
We note that the argument of the proof has appeared in \cite[Theorem~1.1 in Chapter~6]{Hale_1977} and \cite[Theorem~1.1 in Chapter~6]{Hale--VerduynLunel_1993}.

A generalization of Lemma~\ref{lem: Gronwall's ineq and RFDEs} is possible by using Lemma~\ref{lem: generalized Gronwall's ineq}.

\begin{lemma}\label{lem: generalized Gronwall's inequality and RFDEs}
Let $\alpha \colon [a, b] \to \mathbb{R}$ and $\beta \colon [a, b] \to [0, \infty)$ be given continuous functions.
If a continuous function $u \colon [a - r, b] \to E$ satisfies
\begin{equation*}
	\norm{u(t)} \le \alpha(t) + \int_a^t \beta(s)\norm{u_s} \mspace{2mu} \mathrm{d}s
\end{equation*}
for all $t \in [a, b]$ and $\alpha$ is monotonically increasing, then
\begin{equation*}
	\norm{u_t} \le \max\{\norm{u_a}, \alpha(t)\} \exp\biggl( \int_a^t \beta(s) \mspace{2mu} \mathrm{d}s \biggr)
\end{equation*}
holds for all $t \in [a, b]$.
\end{lemma}

\begin{proof}
Let $t \in [a, b]$ be fixed and $\theta \in [-r, 0]$ be given.
When $t + \theta \ge a$, we have
\begin{align*}
	\norm{u(t + \theta)}
	&\le \alpha(t + \theta) + \int_a^{t + \theta} \beta(s)\norm{u_s} \mspace{2mu} \mathrm{d}s \\
	&\le \alpha(t) + \int_a^t \beta(s)\norm{u_s} \mspace{2mu} \mathrm{d}s.
\end{align*}
Here the property that $\alpha$ is monotonically increasing and the non-negativity of $\beta$ are used.
When $t + \theta \le a$, we have
\begin{equation*}
	\norm{u(t + \theta)} \le \norm{u_a}.
\end{equation*}
By combining the above inequalities, we obtain
\begin{equation*}
	\norm{u_t}
	\le \max\{\norm{u_a}, \alpha(t)\} + \int_a^t \beta(s)\norm{u_s} \mspace{2mu} \mathrm{d}s.
\end{equation*}
Since the functions $[a, b] \ni t \mapsto \norm{u_t} \in [0, \infty)$ and $[a, b] \ni t \mapsto \max\{\norm{u_a}, \alpha(t)\} \in [0, \infty)$ are continuous, the conclusion is obtained by applying Lemma~\ref{lem: generalized Gronwall's ineq}.
\end{proof}

\section{Lemmas on fixed point argument}\label{sec: lemmas on fixed pt argument}

Let $E = (E, \norm{\argdot})$ be a normed space and $r > 0$ be a constant.
For each $\gamma > 0$, let
\begin{equation*}
	Y_\gamma
	\coloneqq \Set{y \in C([-r, \infty), E)}{\text{$y_0 = 0$, $\norm{y}_\gamma < \infty$}}
\end{equation*}
be a normed space endowed with the norm $\norm{\argdot}_\gamma$ given by
\begin{equation*}
	\norm{y}_\gamma
	\coloneqq \sup_{t \ge 0} \mspace{2mu} (\mathrm{e}^{-\gamma t}\norm{y_t})
	< \infty.
\end{equation*}
For the notation $\norm{y_t}$, see Subsection~\ref{subsec: Gronwall's ineq and RFDEs}.

\begin{lemma}\label{lem: norm{argdot}_gamma}
For any continuous function $y \colon [-r, \infty) \to E$ with $y_0 = 0$,
\begin{equation*}
	\norm{y}_\gamma = \sup_{t \ge 0} \mspace{2mu} (\mathrm{e}^{-\gamma t}\norm{y(t)})
\end{equation*}
holds.
\end{lemma}

\begin{proof}
Since $\norm{y(t)} \le \norm{y_t}$ holds for all $t \ge 0$,
\begin{equation*}
	\sup_{t \ge 0} \mspace{2mu} (\mathrm{e}^{-\gamma t}\norm{y(t)}) \le \norm{y}_\gamma
\end{equation*}
holds.
The reverse inequality also follows in view of
\begin{equation*}
	\mathrm{e}^{-\gamma t}\norm{y(t + \theta)}
	= \mathrm{e}^{-\gamma(t + \theta)}\norm{y(t + \theta)} \cdot \mathrm{e}^{\gamma \theta}
	\le \sup_{t \ge 0} \mspace{2mu} (\mathrm{e}^{-\gamma t}\norm{y(t)})
\end{equation*}
for $t \ge 0$ and $\theta \in [-r, 0]$.
Here $y_0 = 0$ and $\mathrm{e}^{\gamma \theta} \le 1$ are used.
\end{proof}

\begin{lemma}\label{lem: completeness of Y_gamma}
If $E$ is a Banach space, then $Y_\gamma$ is also a Banach space.
\end{lemma}

\begin{proof}
Let $(y^k)_{k = 1}^\infty$ be a Cauchy sequence in $Y_\gamma$.
We choose $\varepsilon > 0$.
Then for all sufficiently large $k, \ell \ge 1$, we have $\norm{y^k - y^\ell}_\gamma \le \varepsilon$.
From Lemma~\ref{lem: norm{argdot}_gamma}, this means that for all sufficiently large $k, \ell \ge 1$,
\begin{equation*}
	\norm{y^k(t) - y^\ell(t)} \le \varepsilon \mathrm{e}^{\gamma t}
\end{equation*}
holds for all $t \ge 0$.
This implies that $(y^k(t))_{k = 1}^\infty$ is a Cauchy sequence for each $t \ge 0$, and therefore, $(y^k)_{k = 1}^\infty$ has the limit function $y \colon [-r, \infty) \to E$ with $y_0 = 0$.
Since the above relation shows that the convergence of $(y^k)_{k = 1}^\infty$ to $y$ is uniform on each closed and bounded interval of $\mathbb{R}$ by taking the limit as $\ell \to \infty$, the limit function $y$ is continuous.
Then it is concluded that
\begin{equation*}
	\norm{y^k - y}_\gamma \le \varepsilon
\end{equation*}
holds for all sufficiently large $k \ge 1$, which implies that $(y^k)_{k = 1}^\infty$ converges to $y$ in $Y_\gamma$.
\end{proof}

\section{Convolution continued}\label{sec: convolution continued}

In this appendix, we discuss the convolution for functions in $\mathcal{L}^1_\mathrm{loc}([0, \infty), M_n(\mathbb{K}))$.
The purpose here is to share results on the convolution and their proofs in the literature of RFDEs.
The results discussed here extend the results in Subsection~\ref{subsec: convolution and RS convolution}, but they will not be used in this paper.
See also \cite[Proposition~A.4, Theorems~A.5, A.6, A.7 in Appendix~A]{Kappel_2006}.

\subsection{Convolution for locally essentially bounded functions and locally Lebesgue integrable functions}\label{sebsec: convolution, L^infty_loc and L^1_loc}

We first recall that a function $g \in \mathcal{L}^1_\mathrm{loc}([0, \infty), M_n(\mathbb{K}))$ is said to be \textit{locally essentially bounded} if
\begin{equation*}
	\esssup_{t \in [0, T]} \abs{g(t)}
	\coloneqq \inf \Set{M > 0}{\text{$\abs{g(t)} \le M$ holds for almost all $t \in [0, T]$}}
\end{equation*}
is finite for all $T > 0$.
Let
\begin{align*}
	&\mathcal{L}^\infty_\mathrm{loc}([0, \infty), M_n(\mathbb{K})) \\
	&\coloneqq \Set{g \in \mathcal{L}^1_\mathrm{loc}([0, \infty), M_n(\mathbb{K}))}{\text{$g$ is locally essentially bounded}},
\end{align*}
which is a linear subspace of $\mathcal{L}^1_\mathrm{loc}([0, \infty), M_n(\mathbb{K}))$.
As in Definition~\ref{dfn: convolution for locally Riemann integrable functions}, we introduce the following.

\begin{definition}\label{dfn: convolution for L^infty_loc and L^1_loc}
For each $f \in \mathcal{L}^1_\mathrm{loc}([0, \infty), M_n(\mathbb{K}))$ and each $g \in \mathcal{L}^\infty_\mathrm{loc}([0, \infty), M_n(\mathbb{K}))$, we define a function $g * f \colon [0, \infty) \to M_n(\mathbb{K})$ by
\begin{equation*}
	(g * f)(t)
	\coloneqq \int_0^t g(t - u)f(u) \mspace{2mu} \mathrm{d}u
	= \int_0^t g(u)f(t - u) \mspace{2mu} \mathrm{d}u
\end{equation*}
for $t \ge 0$.
Here the integrals are Lebesgue integrals.
The function $g * f$ is called the \textit{convolution} of $g$ and $f$.
\end{definition}

We note that
\begin{equation*}
	\abs{(g * f)(t)}
	\le \esssup_{u \in [0, t]} \abs{g(u)} \cdot \int_0^t \abs{f(u)} \mspace{2mu} \mathrm{d}u
\end{equation*}
holds for all $t \ge 0$.
The following result should be compared with Lemma~\ref{lem: continuity of convolution, loc Riemann integrable function and continuous function}.

\begin{lemma}\label{lem: continuity of convolution, L^infty_loc and L^1_loc}
Let $f \in \mathcal{L}^1_\mathrm{loc}([0, \infty), M_n(\mathbb{K}))$ and $g \in \mathcal{L}^\infty_\mathrm{loc}([0, \infty), M_n(\mathbb{K}))$.
Then $g * f$ is continuous.
\end{lemma}

\begin{proof}[Outline of the proof]
We show the continuity of $g * f$ on $[0, T]$ for each fixed $T > 0$.
We define a function $\Tilde{f} \colon \mathbb{R} \to M_n(\mathbb{K})$ by
\begin{equation*}
	\Tilde{f}(t) \coloneqq
	\begin{cases}
		f(t) & (t \in \dom(f) \cap [0, T]), \\
		O & (\text{otherwise}).
	\end{cases}
\end{equation*}
Then $\Tilde{f} \in \mathcal{L}^1(\mathbb{R}, M_n(\mathbb{K}))$, and
\begin{equation*}
	(g * f)(t) = \int_0^t g(u)\Tilde{f}(t - u) \mspace{2mu} \mathrm{d}u
\end{equation*}
holds for all $t \in [0, T]$.
We fix $t_0 \in [0, T]$.
By the reasoning as in the proof of Lemma~\ref{lem: continuity of RS convolution}, we have
\begin{align*}
	&(g * f)(t) - (g * f)(t_0) \\
	&= \int_0^{t_0} g(u) \bigl[ \Tilde{f}(t - u) - \Tilde{f}(t_0 - u) \bigr] \mspace{2mu} \mathrm{d}u
	+ \int_{t_0}^t g(u)\Tilde{f}(t - u) \mspace{2mu} \mathrm{d}u
\end{align*}
for all $t \in [0, T]$.
Therefore, the continuity of $g * f$ on $[0, T]$ is obtained by H\"{o}lder's inequality, the continuity of the translation in $\mathcal{L}^1$, and the integrability of $\Tilde{f}$.
\end{proof}

\subsection{Convolution for locally Lebesgue integrable functions}

The notion of convolution in Definition~\ref{dfn: convolution for L^infty_loc and L^1_loc} is not satisfactory in the sense that the condition on $f$ and $g$ is not symmetry.
To introduce the notion of convolution for functions in $\mathcal{L}^1_\mathrm{loc}([0, \infty), M_n(\mathbb{K}))$, we need the following.

\begin{theorem}\label{thm: convolution thm}
Let $f, g \in \mathcal{L}^1_\mathrm{loc}([0, \infty), M_n(\mathbb{K}))$ be given.
Then the following statements hold:
\begin{enumerate}
\item For almost all $t > 0$, $u \mapsto g(t - u)f(u)$ belongs to $\mathcal{L}^1_\mathrm{loc}([0, t], M_n(\mathbb{K}))$.
\item The function $g * f$ defined by
\begin{equation*}
	(g * f)(t) \coloneqq \int_0^t g(t - u)f(u) \mspace{2mu} \mathrm{d}u
\end{equation*}
for almost all $t \ge 0$ belongs to $\mathcal{L}^1_\mathrm{loc}([0, \infty), M_n(\mathbb{K}))$.
\item For all $T \ge 0$,
\begin{equation*}
	\int_0^T \abs{(g * f)(t)} \mspace{2mu} \mathrm{d}t
	\le \biggl( \int_0^T \abs{g(t)} \mspace{2mu} \mathrm{d}t \biggr)
	\cdot \biggl( \int_0^T \abs{f(t)} \mspace{2mu} \mathrm{d}t \biggr)
\end{equation*}
holds.
\end{enumerate}
\end{theorem}

In the following, we give a direct proof of Theorem~\ref{thm: convolution thm} by using Fubini's theorem and Tonelli's theorem for functions on the Euclidean space $\mathbb{R}^d$.
See \cite[Theorems~3.1 and 3.2 in Section~3 of Chapter~2]{Stein--Shakarchi_2005} for these statements and their proofs.

\begin{proof}[A direct proof of Theorem~\ref{thm: convolution thm}]
Let $\Bar{f} \colon \mathbb{R} \to M_n(\mathbb{K})$ be the function defined by
\begin{equation*}
	\Bar{f}(t) \coloneqq
	\begin{cases}
		f(t) & (t \in \dom(f)), \\
		O & (t \in \mathbb{R} \setminus \dom(f)).
	\end{cases}
\end{equation*}
In the same way, we define the function $\Bar{g} \colon \mathbb{R} \to M_n(\mathbb{K})$.
Then $\Bar{f}, \Bar{g} \colon \mathbb{R} \to M_n(\mathbb{K})$ are locally Lebesgue integrable functions.

Let $T > 0$ be fixed.
The remainder of the proof is divided into the following steps.

\vspace{0.5\baselineskip}
\noindent
\textbf{Step 1: Setting of triangle region and function.}
We consider a closed set $A$ of $\mathbb{R}^2$ given by
\begin{equation*}
	A \coloneqq \Set{(t, u) \in \mathbb{R}^2}{\text{$t \in [0, T]$, $u \in [0, t]$}}.
\end{equation*}
See Fig.~\ref{fig: subset A} for the picture of $A$.
\begin{figure}[tbp]
\centering
\begin{tikzpicture}
	\draw[-latex,semithick] (-1,0) -- (4,0) node[right]{$t$}; 
	\draw[-latex,semithick] (0,-1) -- (0,4) node[above]{$u$};
	\draw (0,0) node[below left]{O}; %origin
	\fill[lightgray] (0,0) -- (3,3) -- (3,0) -- cycle;
	\draw (0,0) -- (3,3) -- (3,0) -- cycle;
	\draw (3,0) node[below]{$T$};
	\draw[dotted] (0,3) node[left]{$T$} -- (3,3);
\end{tikzpicture}
\caption{The light gray region is the subset $A$.}
\label{fig: subset A}
\end{figure}
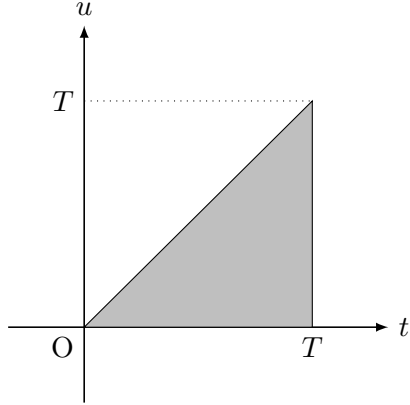
Then the characteristic function $\bm{1}_A$ is measurable and
\begin{equation*}
	\bm{1}_{[0, T]}(t)\bm{1}_{[0, t]}(u) = \bm{1}_A(t, u)
\end{equation*}
holds for all $(t, u) \in \mathbb{R}^2$.
We define a function $h \colon \mathbb{R}^2 \to M_n(\mathbb{K})$ by
\begin{equation*}
	h(t, u)
	\coloneqq \bm{1}_A(t, u)\Bar{g}(t - u)\Bar{f}(u).
\end{equation*}
Then $h$ is measurable because
\begin{equation*}
	\mathbb{R}^2 \ni (t, u) \mapsto \Bar{g}(t - u) \in M_n(\mathbb{K})
	\amd
	\mathbb{R}^2 \ni (t, u) \mapsto \Bar{f}(u) \in M_n(\mathbb{K})
\end{equation*}
are measurable.\footnote{See \cite[Corollary~3.7 and Proposition~3.9 in Section~3 of Chapter~2]{Stein--Shakarchi_2005} for the results of scalar-valued case.}
This implies that the function $\mathbb{R}^2 \ni (t, u) \mapsto \abs{h(t, u)} \in [0, \infty)$ is also measurable.

\vspace{0.5\baselineskip}
\noindent
\textbf{Step 2: Application of Tonelli's theorem.}
By applying Tonelli's theorem, the following statements hold:
\begin{itemize}
\item For almost all $u \in \mathbb{R}$, the function $\mathbb{R} \ni t \mapsto \abs{h(t, u)} \in [0, \infty)$ is measurable.
\item For almost all $t \in \mathbb{R}$, the function $\mathbb{R} \ni u \mapsto \abs{h(t, u)} \in [0, \infty)$ is measurable.
\item The functions
\begin{equation*}
	u \mapsto \int_{\mathbb{R}} \abs{h(t, u)} \mspace{2mu} \mathrm{d}t \in [0, \infty],
	\mspace{15mu}
	t \mapsto \int_{\mathbb{R}} \abs{h(t, u)} \mspace{2mu} \mathrm{d}u \in [0, \infty]
\end{equation*}
are measurable functions defined almost everywhere.
\item We have
\begin{equation*}
	\int_{\mathbb{R}}
		\biggl( \int_{\mathbb{R}} \abs{h(t, u)} \mspace{2mu} \mathrm{d}t \biggr)
	\mspace{2mu} \mathrm{d}u
	= \int_{\mathbb{R}}
		\biggl( \int_{\mathbb{R}} \abs{h(t, u)} \mspace{2mu} \mathrm{d}u \biggr)
	\mspace{2mu} \mathrm{d}t
	= \int_{\mathbb{R}^2} \abs{h(t, u)} \mspace{2mu} \mathrm{d}(t, u)
\end{equation*}
including the possibility that all the unsigned Lebesgue integrals are $\infty$.
\end{itemize}

\vspace{0.5\baselineskip}
\noindent
\textbf{Step 3: Application of Fubini's theorem.}
By Step 2, we have
\begin{align*}
	\int_{\mathbb{R}^2} \abs{h(t, u)} \mspace{2mu} \mathrm{d}(t, u)
	&= \int_{\mathbb{R}}
		\biggl( \int_{\mathbb{R}} \abs{h(t, u)} \mspace{2mu} \mathrm{d}u \biggr)
	\mspace{2mu} \mathrm{d}t \\
	&\le \int_0^T
		\biggl( \int_u^T \abs{\Bar{g}(t - u)} \mspace{2mu} \mathrm{d}t \biggr)
	\abs{\Bar{f}(u)} \mspace{2mu} \mathrm{d}u \\
	&\le \biggl( \int_0^T \abs{g(t)} \mspace{2mu} \mathrm{d}t \biggr) \cdot \biggl( \int_0^T \abs{f(t)} \mspace{2mu} \mathrm{d}t \biggr).
\end{align*}
Since the last term is finite, it holds that $h$ is integrable.
By applying Fubini's theorem component-wise, the following statements hold:
\begin{itemize}
\item For almost all $u \in \mathbb{R}$, the function $\mathbb{R} \ni t \mapsto h(t, u) \in M_n(\mathbb{K})$ is Lebesgue integrable.
\item For almost all $t \in \mathbb{R}$, the function $\mathbb{R} \ni u \mapsto h(t, u) \in M_n(\mathbb{K})$ is Lebesgue integrable.
\item The functions
\begin{equation*}
	u \mapsto \int_{\mathbb{R}} h(t, u) \mspace{2mu} \mathrm{d}t \in M_n(\mathbb{K}),
	\mspace{15mu}
	t \mapsto \int_{\mathbb{R}} h(t, u) \mspace{2mu} \mathrm{d}u \in M_n(\mathbb{K})
\end{equation*}
belong to $\mathcal{L}^1(\mathbb{R}, M_n(\mathbb{K}))$.
\item The equalities
\begin{equation*}
	\int_{\mathbb{R}}
		\biggl( \int_{\mathbb{R}} h(t, u) \mspace{2mu} \mathrm{d}t \biggr)
	\mspace{2mu} \mathrm{d}u
	= \int_{\mathbb{R}}
		\biggl( \int_{\mathbb{R}} h(t, u) \mspace{2mu} \mathrm{d}u \biggr)
	\mspace{2mu} \mathrm{d}t
	= \int_{\mathbb{R}^2} h(t, u) \mspace{2mu} \mathrm{d}(t, u)
\end{equation*}
hold.
\end{itemize}

\vspace{0.5\baselineskip}
\noindent
\textbf{Step 4: Conclusion.}
For each $t \in [0, T]$,
\begin{equation*}
	h(t, u) = g(t - u)f(u)
\end{equation*}
holds for almost all $u \in [0, t]$.
Therefore, for almost all $t \in [0, T]$, the function $u \mapsto g(t - u)f(u)$ belongs to $\mathcal{L}^1_\mathrm{loc}([0, t], M_n(\mathbb{K}))$.
Furthermore, we have
\begin{equation*}
	\int_\mathbb{R} h(t, u) \mspace{2mu} \mathrm{d}u
	= \int_0^t g(t - u)f(u) \mspace{2mu} \mathrm{d}u
\end{equation*}
for almost all $t \in [0, T]$, and it holds that the function
\begin{equation*}
	t \mapsto \int_0^t g(t - u)f(u) \mspace{2mu} \mathrm{d}u
\end{equation*}
is a Lebesgue integrable function defined almost everywhere on $[0, T]$.
Since $T > 0$ is arbitrary, the statements 1 and 2 hold.
The statement 3 also holds because we have
\begin{align*}
	\int_0^T \abs{(g * f)(t)} \mspace{2mu} \mathrm{d}t
	&\le \int_0^T
		\biggl( \int_0^t \abs{g(t - u)f(u)} \mspace{2mu} \mathrm{d}u \biggr)
	\mspace{2mu} \mathrm{d}t \\
	&= \int_{\mathbb{R}}
		\biggl( \int_{\mathbb{R}} \abs{h(t, u)} \mspace{2mu} \mathrm{d}u \biggr)
	\mspace{2mu} \mathrm{d}t \\
	&\le \biggl( \int_0^T \abs{g(t)} \mspace{2mu} \mathrm{d}t \biggr) \cdot \biggl( \int_0^T \abs{f(t)} \mspace{2mu} \mathrm{d}t \biggr),
\end{align*}
where the calculation in Step 3 is used.

This completes the proof.
\end{proof}

\begin{proof}[Another proof of Theorem~\ref{thm: convolution thm}]
Let $T > 0$ be fixed.
We define $\Tilde{f} \colon \mathbb{R} \to M_n(\mathbb{K})$ by
\begin{equation*}
	\Tilde{f}(t) \coloneqq
	\begin{cases}
		f(t) & (t \in \dom(f) \cap [0, T]), \\
		O & (\text{otherwise}).
	\end{cases}
\end{equation*}
In the same way, we define the function $\Tilde{g} \colon \mathbb{R} \to M_n(\mathbb{K})$.
Since $\Tilde{f}, \Tilde{g} \colon \mathbb{R} \to M_n(\mathbb{K})$ are Lebesgue integrable functions, one can prove the following statements as in the scalar-valued case:\footnote{See \cite[Exercise~21 in Chapter~2]{Stein--Shakarchi_2005} and \cite[8.13 and 8.14 of Chapter~8]{Rudin_1987} for the scalar-valued case.}
\begin{enumerate}[label=\arabic*$'$.]
\item For almost all $t \in \mathbb{R}$, the function $u \mapsto \Tilde{g}(t - u)\Tilde{f}(u)$ is a Lebesgue integrable function defined almost everywhere.
\item The function $\Tilde{g} \star \Tilde{f}$ defined by
\begin{equation*}
	\bigl( \Tilde{g} \star \Tilde{f} \bigr)(t)
	\coloneqq \int_\mathbb{R} \Tilde{g}(t - u)\Tilde{f}(u) \mspace{2mu} \mathrm{d}u
\end{equation*}
for almost all $t \in \mathbb{R}$ belongs to $\mathcal{L}^1(\mathbb{R}, M_n(\mathbb{K}))$.
\item An estimate
\begin{equation*}
	\int_\mathbb{R} \abs{\bigl( \Tilde{g} \star \Tilde{f} \bigr)(t)} \mspace{2mu} \mathrm{d}t
	\le \norm{\Tilde{g}}_1 \cdot \bigl\| \Tilde{f} \bigr\|_1
\end{equation*}
holds.
\end{enumerate}

1. For each $t \in [0, T]$, we have
\begin{equation*}
	\Tilde{g}(t - u)\Tilde{f}(u) = g(t - u)f(u)
\end{equation*}
for almost all $u \in [0, t]$.
By combining this and the above statement 1$'$, it holds that for almost all $t \in [0, T]$, $u \mapsto g(t - u)f(u)$ is a Lebesgue integrable function defined almost everywhere on $[0, t]$.
Since $T > 0$ is arbitrary, the statement 1 holds.
\vspace{0.2\baselineskip}

2. By the definitions of $\Tilde{f}$ and $\Tilde{g}$,
\begin{equation*}
	\bigl( \Tilde{g} \star \Tilde{f} \bigr)(t)
	= \int_0^t \Tilde{g}(t - u)\Tilde{f}(u) \mspace{2mu} \mathrm{d}u
	= \int_0^t g(t - u)f(u) \mspace{2mu} \mathrm{d}u
\end{equation*}
holds for all $t \in \dom \bigl( \Tilde{g} \star \Tilde{f} \bigr) \cap [0, T]$.
Since $T > 0$ is arbitrary, this shows that
\begin{equation*}
	t \mapsto \int_0^t g(t - u)f(u) \mspace{2mu} \mathrm{d}u
\end{equation*}
is a measurable function defined almost everywhere on $[0, \infty)$ from the statement 2$'$.
Furthermore, we also have
\begin{equation*}
	\int_0^T
		\abs{\int_0^t g(t - u)f(u) \mspace{2mu} \mathrm{d}u} 
	\mspace{2mu} \mathrm{d}t
	= \int_0^T \abs{\bigl( \Tilde{g} \star \Tilde{f} \bigr)(t)} \mspace{2mu} \mathrm{d}t
	< \infty.
\end{equation*}
Since $T > 0$ is arbitrary, the statement 2 holds.
\vspace{0.2\baselineskip}

3. By combining the proof of the statement 2 and the inequality in the statement 3$'$, we have
\begin{equation*}
	\int_0^T \abs{(g * f)(t)} \mspace{2mu} \mathrm{d}t
	\le \norm{\Tilde{g}}_1 \cdot \bigl\| \Tilde{f} \bigr\|_1
\end{equation*}
Here
\begin{equation*}
	\bigl\| \Tilde{f} \bigr\|_1
	= \int_0^T \abs{f(t)} \mspace{2mu} \mathrm{d}t,
	\mspace{15mu}
	\norm{\Tilde{g}}_1
	= \int_0^T \abs{g(t)} \mspace{2mu} \mathrm{d}t
\end{equation*}
holds since $f(t) = g(t) = O$ for $t \in (-\infty, 0) \cup (T, \infty)$.
Therefore, the inequality in the statement 3 is obtained.
\end{proof}

The above proof of Theorem~\ref{thm: convolution thm} is not given in \cite{VerduynLunel_1989}, \cite{Diekmann--vanGils--Lunel--Walther_1995}, \cite{Kappel_2006}, and \cite{Gripenberg--Londen--Staffans_1990}.
Based on Theorem~\ref{thm: convolution thm}, we introduce the following.

\begin{definition}\label{dfn: convolution}
Let $f, g \in \mathcal{L}^1_\mathrm{loc}([0, \infty), M_n(\mathbb{K}))$.
We call $g * f \in \mathcal{L}^1_\mathrm{loc}([0, \infty), M_n(\mathbb{K}))$ in Theorem~\ref{thm: convolution thm} defined by
\begin{equation*}
	(g * f)(t)
	\coloneqq \int_0^t g(t - u)f(u) \mspace{2mu} \mathrm{d}u
	= \int_0^t g(u)f(t - u) \mspace{2mu} \mathrm{d}u
\end{equation*}
the \textit{convolution} of $f$ and $g$.
\end{definition}

\subsection{Convolution under Volterra operator}

The convolution for functions in $\mathcal{L}^1_\mathrm{loc}([0, \infty), M_n(\mathbb{K}))$ and the Volterra operator are related in the following way.

\begin{theorem}\label{thm: convolution under Volterra operator}
For any pair of $f, g \in \mathcal{L}^1_\mathrm{loc}([0, \infty), M_n(\mathbb{K}))$,
\begin{equation}\label{eq: convolution under Volterra operator}
	V(g * f) = (Vg) * f = g * (Vf) 
\end{equation}
holds.
\end{theorem}

The above theorem is an extension of Corollary~\ref{cor: convolution of loc BV function and continuous function under Volterra operator}.

\begin{proof}[Proof of Theorem~\ref{thm: convolution under Volterra operator}]
For each $t > 0$,
\begin{equation*}
	V(g * f)(t)
	= \int_0^t
		\biggl( \int_0^s g(s - u)f(u) \mspace{2mu} \mathrm{d}u \biggr)
	\mspace{2mu} \mathrm{d}s
\end{equation*}
holds by the definition of convolution.
By applying Fubini's theorem in the similar way as in the direct proof of Theorem~\ref{thm: convolution thm}, the right-hand side is calculated as
\begin{align*}
	\int_0^t
		\biggl( \int_0^s g(s - u)f(u) \mspace{2mu} \mathrm{d}u \biggr)
	\mspace{2mu} \mathrm{d}s
	&= \int_0^t
		\biggl( \int_{u}^t g(s - u) \mspace{2mu} \mathrm{d}s \biggr) f(u)
	\mspace{2mu} \mathrm{d}u \\
	&= \int_0^t (Vg)(t - u)f(u) \mspace{2mu} \mathrm{d}u,
\end{align*}
where the last term is equal to $[(Vg) * f](t)$.
Therefore, the integration by parts formula for matrix-valued absolutely continuous functions (see Subsection~\ref{subsec: derivation of VOC formula}) yields
\begin{align*}
	[(Vg) * f](t)
	&= [(Vg)(t - u)(Vf)(u)]_{u = 0}^t + \int_0^t g(t - u)(Vf)(u) \mspace{2mu} \mathrm{d}u \\
	&= [g * (Vf)](t),
\end{align*}
where $(Vg)(0) = (Vf)(0) = O$ is used.
This completes the proof.
\end{proof}

\begin{remark}
Eq.~\eqref{eq: convolution under Volterra operator} is a special case of the \textit{associativity of convolution}
\begin{equation}\label{eq: associativity of convolution}
	(h * g) * f = h * (g * f)
\end{equation}
for $f, g, h \in \mathcal{L}^1_\mathrm{loc}([0, \infty), M_n(\mathbb{K}))$ because
\begin{equation*}
	(f * \mathcal{I})(t)
	= (\mathcal{I} * f)(t)
	= \int_0^t f(s) \mspace{2mu} \mathrm{d}s
	= (Vf)(t)
	\mspace{25mu}
	(t \ge 0)
\end{equation*}
holds for any $f \in \mathcal{L}^1_\mathrm{loc}([0, \infty), M_n(\mathbb{K}))$.
Here $\mathcal{I} \colon [0, \infty) \to M_n(\mathbb{K})$ denote the constant function whose value is equal to the identity matrix.
\end{remark}

The following is a result on the regularity of convolution.
It should be compared with Theorem~\ref{thm: convolution of loc BV function and continuous function}.

\begin{theorem}\label{thm: convolution of loc AC function and loc L^1 function}
Let $f \in \mathcal{L}^1_\mathrm{loc}([0, \infty), M_n(\mathbb{K}))$ and $g \colon [0, \infty) \to M_n(\mathbb{K})$ be a locally absolutely continuous function.
Then $g * f$ is expressed by
\begin{equation}
	g * f = V \bigl( g(0)f + g' * f \bigr). \label{eq: g * f, Volterra operator} 
\end{equation}
Consequently, $g * f$ is locally absolutely continuous, differentiable almost everywhere, and satisfies
\begin{equation*}
	(g * f)'(t) = g(0)f(t) + (g' * f)(t) 
\end{equation*}
for almost all $t \ge 0$.
\end{theorem}

We note that for a locally absolutely continuous function $g \colon [0, \infty) \to M_n(\mathbb{K})$, the derivative $g'$ belongs to $\mathcal{L}^1_\mathrm{loc}([0, \infty), M_n(\mathbb{K}))$.
Therefore, the convolution $g' * f$ makes sense from Theorem~\ref{thm: convolution thm}.

\begin{proof}[Proof of Theorem~\ref{thm: convolution of loc AC function and loc L^1 function}]
Since $g = g(0) + Vg'$, we obtain
\begin{equation*}
	g * f
	= g(0)Vf + (Vg') * f
	= g(0)Vf + V(g' * f)
\end{equation*}
by using Theorem~\ref{thm: convolution under Volterra operator}.
This yields the expression~\eqref{eq: g * f, Volterra operator} because the Volterra operator is linear.
The remaining properties of $g * f$ are derived by the properties of Volterra operator.
\end{proof}

See also \cite[7.4 Corollary in Chapter~3]{Gripenberg--Londen--Staffans_1990} for related results.

\begin{remark}
From Theorem~\ref{thm: convolution of loc AC function and loc L^1 function}, we have
\begin{equation*}
	(Vg) * f = V(g * f)
\end{equation*}
for $f, g \in \mathcal{L}^1_\mathrm{loc}([0, \infty), M_n(\mathbb{K}))$.
We also have
\begin{equation*}
	g * (Vf) = V(g * f)
\end{equation*}
in the similar way.
This means that Theorems~\ref{thm: convolution under Volterra operator} and \ref{thm: convolution of loc AC function and loc L^1 function} are logically equivalent.
\end{remark}

We note that the statement 2 of Corollary~\ref{cor: convolution of loc AC function and continuous function} also follows by Lemma~\ref{lem: continuity of convolution, L^infty_loc and L^1_loc} and Theorem~\ref{thm: convolution of loc AC function and loc L^1 function}.

\end{document}